\documentclass[11pt,twoside,reqno]{amsart}
\usepackage{amsmath, amsfonts, amsthm, amssymb,amscd, graphicx, amscd}
\usepackage{float,epsf}
\usepackage[english]{babel}
\usepackage{enumerate}
\usepackage{tikz}
\usepackage{mathrsfs}
\usepackage[compress,sort]{cite}

\usepackage{srcltx}
\usepackage{geometry}
\usepackage{verbatim}
\usepackage{mathrsfs}
\usepackage{hyperref}
\usepackage{enumitem} 
\usepackage{bbm} 
\usepackage{stmaryrd}

\usepackage{relsize}
\usepackage{exscale}

\usepackage{mathtools}
\mathtoolsset{showonlyrefs}

\newcommand{\R}{\mathbb{R}}

 \newcommand{\GG}{\mathcal{G}}

\newcommand{\eps}{\varepsilon} 
 
\newcommand{\loc}{{\text{\rm loc}}}
\newcommand{\spt}{\operatorname{spt}}

\renewcommand{\d}{\delta}

 \newcommand{\supp}{\text{\rm supp}\,}

 \renewcommand{\supp}{\text{\rm supp}\,}

 \newcommand{\Lip}{\text{\rm Lip}}
\newcommand{\ban}[1]{\left\langle  #1 \right\rangle}
 \newcommand{\F}{\mathcal{ F}}
\newtheorem{theorem}{Theorem}[section]
\newtheorem{lemma}[theorem]{Lemma}

\newtheorem{definition}[theorem]{Definition}

\newtheorem{corollary}[theorem]{Corollary}

\theoremstyle{definition}
\newtheorem{remark}[theorem]{Remark}
\newtheorem{example}[theorem]{Example}

\numberwithin{equation}{section}
\numberwithin{figure}{section}

\newcommand{\abs}[1]{\left\vert{#1}\right\vert}

\newcommand{\norm}[1]{\left\lVert #1\right\rVert}

\newcommand{\FF}{{\boldsymbol F}}
\renewcommand{\GG}{{\boldsymbol G}}

\newcommand{\DM}{\mathcal D\mathcal M} 
\renewcommand{\div}{{\rm div}} 
\newcommand{\dist}{\mathrm{dist}} 
\newcommand{\weakstarto}{\xrightharpoonup{\,\, * \,\,}}

\newcommand{\res}{\mathop{\hbox{\vrule height 7pt width .5pt depth 0pt
\vrule height .5pt width 6pt depth 0pt}}\nolimits}

\renewcommand{\div}{\operatorname{div}}
\newcommand{\dr}{{\rm d}}

\def\rightangle{\vcenter{\hsize5.5pt
    \hbox to5.5pt{\vrule height7pt\hfill}
    \hrule}}
\def\rtangle{\mathrel{\rightangle}}
\def\intave#1{\int_{#1}\hbox{\llap{$\raise2.3pt\hbox{\vrule
height.9pt width7pt}\phantom{\scriptstyle{#1}}\mkern-2mu$}}}

\renewcommand{\d}{\mathrm{d}}
\newcommand{\threepartdef}[6] { \left\{ \begin{array}{ll} #1 & #2 \\ #3 & #4 \\ #5 & #6 \end{array} \right.}
\DeclareMathOperator{\ext}{ext}

\newcommand{\BV}{BV}
\newcommand{\Sob}{{W}}
\newcommand{\Leb}{L}
\DeclareMathOperator{\interior}{int}
\newcommand{\bfu}{{\bf u}}
\newcommand{\bff}{{\bf f}}
\newcommand{\bfq}{{\bf q}}

\makeatletter
\def\@citestyle{\m@th\upshape\mdseries}
\let\citeform\@firstofone
\def\@cite#1#2{{%
  \@citestyle[\citeform{#1}\if@tempswa, #2\fi]}}
\makeatother

\hypersetup{bookmarksdepth=2}

\begin{document}
\title[Extended Divergence-Measure Fields and Cauchy Fluxes]{Extended Divergence-Measure Fields,
\\ the Gauss-Green Formula\\ and Cauchy Fluxes}

\author{Gui-Qiang G. Chen}
\address{Gui-Qiang G. Chen, Mathematical Institute, University of Oxford,
Oxford, OX2 6GG, UK}
\email{chengq@maths.ox.ac.uk}

\author{Christopher Irving}
\address{Christopher Irving,
Faculty of Mathematics, Technical University Dortmund,
Vogelpothsweg 87,
44227, Dortmund, Germany}
\email{christopher.irving@tu-dortmund.de}

\author{Monica Torres}
\address{Monica Torres,
Department of Mathematics, Purdue University,
 West Lafayette, IN 47907-2067, USA}
 \email{torres@math.purdue.edu}


\begin{abstract}
We establish the Gauss-Green formula for extended divergence-measure
fields ({\it i.e.}, vector-valued measures whose distributional divergences are Radon measures) over open sets.
We prove that, for {\it almost every open set}, the normal trace is a measure supported on the boundary of the set.
Moreover, for any open set, we provide a representation of the normal trace of the field over the boundary of the open set as the limit of measure-valued
normal traces over the boundaries of approximating sets.
Furthermore, using this theory, we extend the balance law from classical continuum physics
to
a general framework
in which the production on any open set is measured with a Radon measure
and the associated Cauchy flux is bounded by a Radon measure concentrated on the boundary of the set.
We prove that there exists an extended divergence-measure field such that the Cauchy flux can be recovered
through the field, {\it locally on almost every open set} and {\it globally on every open set}.
Our results generalize the classical Cauchy's Theorem (that is only valid for continuous vector fields)
and extend the previous formulations of the Cauchy flux (that generate vector fields within $L^{p}$).
Thereby, we establish the equivalence between entropy solutions of
the multidimensional nonlinear partial differential equations of divergence form
and of the mathematical formulation of physical balance laws via the Cauchy flux through the constitutive relations
in the axiomatic foundation of Continuum Physics.
\end{abstract}
\maketitle

\bigskip
\bigskip
\section{Introduction}
Divergence-measure fields are defined as vector-valued fields $\FF=(F_1,F_2,\cdots, F_n)$
whose distributional divergences are represented by (signed) Radon measures.
An underlying connection between divergence-measure fields and hyperbolic conservation laws was first observed in \cite{CF1},
and such vector fields over domains with Lipschitz
boundary were analyzed in \cite{CF1,CF2}.
Since then, the analysis of divergence-measure fields has depended essentially on the regularity of $\FF$.
For example, the divergence-measure fields were extensively analyzed first in $\Leb^\infty$ in \cite{ctz}
and then in $\Leb^p$ in \cite{ChenComiTorres}.
See also \cite{ComiTorres, anzellotti1983traces,ACM,ChenLiTorres,ChenTorres,comiextended,comimagnani,comi2017locally,crasta2017anzellotti,crastadecicco2,SchonherrSchuricht,Sch,Silhavy1,S4,Silhavy2}
and the references therein for further developments for the theory of divergence-measure fields.
In this paper, we focus on the case when $\FF$
is only a vector-valued Radon measure.
More precisely,
we analyze {\it extended} divergence-measure fields which are defined as vector-valued Radon measures
whose distributional divergences
are Radon measures.

Our approach in this paper is motivated by
the previous results in the $\Leb^p$ setting in
\cite{ChenComiTorres,ctz}.
However, the case of extended divergence-measure fields is more delicate,
since $\FF$ may concentrate on lower dimensional sets (for instance, rectifiable curves).
We prove that, for {\it almost every open set}, the normal trace of an extended divergence-measure field
is a Radon measure supported on the boundary of the set.
Moreover, for {\it every open set}, the normal trace distribution can be computed as the limit of measure-valued normal
traces over the boundaries of approximating sets.
Equipped with these results,
we further develop a theory of Cauchy fluxes, starting from the balance law and establishing a one-to-one correspondence:
\begin{equation*}
\left\{\text{Cauchy fluxes } \mathcal F \text{ in } \Omega \right\}\,
\longleftrightarrow\, \left\{\text{Extended divergence-measure fields } \FF
\text{ in } \Omega\right\}
\end{equation*}
via the normal trace. The precise statement is given in \S \ref{sec:flux_mainresults}.

In the development of a theory of divergence-measure fields, one of the fundamental issues is whether
a Gauss-Green formula involving these weakly differentiable vector fields can still be provided.
We refer the reader to \cite{ChenTorresNotices} for a detailed exposition on the development
of this fundamental formula, starting from Lagrange (1762) and culminating with the classical formula
\begin{equation}
\label{comienzo}
\int_{U} \phi \, \div \FF \,\d x + \int_{U} \nabla \phi \cdot \FF \,\d x = - \int_{\partial U} \phi\,\FF \cdot \nu \, \d \mathcal{H}^{n-1},
\end{equation}
valid for any smooth vector field $\FF$, smooth test function $\phi$, and open set $U$ with smooth boundary
and interior unit normal $\nu$.
A first extension of this formula was achieved
by Federer and De-Giorgi \cite{degiorgi1961complementi, degiorgi1961frontiere, Federer1, Federer2}
for the cases of Lipschitz vector fields and sets with irregular boundaries (sets of finite perimeter)
by using tools of geometric-measure theory.

Further extensions of \eqref{comienzo} to divergence-measure fields require a notion of normal trace
on the boundary $\partial U$ of any open $U$.
For the case of bounded vector fields and sets of finite perimeter,
the approach in  \cite{ctz} consisted in constructing essentially
{\it interior} and {\it exterior}
approximations of the sets of finite perimeter with smooth sets and then obtaining the normal traces
as the limits of classical normal traces on the smooth approximating sets,
which leads to
the existence of interior and exterior traces for {\it every } set of finite perimeter (see also \cite{ComiTorres}).
This approach is consistent with applications to hyperbolic conservation laws since solutions
to these equations have jumps across the shock waves ({\it cf}. \cite{Chen,CF1,dafermos2010hyperbolic,Lax1973}).
For a divergence-measure field whose underlying field is bounded, given any set of finite perimeter, it was shown in \cite{ctz}
that the normal trace is a bounded function supported on the reduced boundary of the set.
For the case of an unbounded vector field, the normal trace is classical for {\it almost every open set},
and it was shown in \cite{ChenComiTorres} that the normal trace distribution on {\it every open set}
can be computed as the limit of the classical normal traces on the boundaries of approximating sets.
For unbounded vector fields, several counterexamples show that the normal trace distribution
can not be represented in general as a measure supported on the boundary of
the set (see for instance
\cite{ChenComiTorres,ChenTorresNotices,Silhavy2}).
The technique that we use in this paper in order to generalize \eqref{comienzo} to the
case of extended-divergence fields is the {\it disintegration of measures}.

As indicated earlier, divergence-measure fields arise naturally in the field of nonlinear hyperbolic conservation laws:
\begin{equation}
\partial_t{\bf u}+ \sum_{j=1}^m \partial_{x_j} {\bf f}_j({\bf u})=0 \qquad \mbox{for $(t,x) \in \mathbb R_+ \times \mathbb{R}^m$},
\end{equation}
or in short form
 \begin{equation}\label{codo2}
\partial_t{\bf u}+ \div_x \, {\bf f}({\bf u})=0 \qquad \mbox{for $(t,x) \in \mathbb R_+ \times \mathbb{R}^m$},
\end{equation}
where $n = m+1$,
taking the row-wise divergence of ${\bf f}({\bf u}) = (\bff_1(\bfu),\bff_2(\bfu),\cdots,\bff_m(\bfu))$,
where
where ${\bf f}_i^{\intercal}:\R^N \to \R^m$
for $i=1, \cdots, m$ and ${\bf u} = (u_1,u_2,\cdots,u_N)^{\intercal}$.
One of the main features of \eqref{codo2} is that, no matter how smooth the initial data start with, the solution
may develop singularities to become discontinuous or singular (unbounded or measure-valued).
Physical relevant solutions, so-called {\it entropy solutions}, of \eqref{codo2} are required to be
characterized by the entropy inequality:
\begin{equation}
\label{entropyinequality}
\partial_t\eta(\mathbf u) + \div_x \,
{\bf q}({\bf u}) \leq 0
\end{equation}
that holds in the distributional sense for any entropy-entropy flux
pair $(\eta, {\bf q})=(\eta, q_1,\cdots, q_m)$ ({\it i.e.}, $\nabla q_i({\bf u})=\nabla\eta({\bf u}){\bf f}_i({\bf u}), i=1,\cdots, m$) which is \emph{convex} in that $\nabla^2\eta({\bf u})\ge 0$,
for which the field $(\eta({\bf u}(t,x)), {\bf q}({\bf u}(t,x)))$ is defined.
From \eqref{entropyinequality}, it follows that there is a non-negative measure $\sigma_{\eta}$ such that
\begin{equation}
\label{aquitrabajo}
- \div_{t,x} (\eta({\bf u}(t,x)), {\bf q}({\bf u}(t,x)))= \sigma_{\eta}.
\end{equation}
Therefore, $(\eta({\bf u}(t,x)), {\bf q}({\bf u}(t,x)))$ is an extended divergence-measure field.
To study the jumps of entropy solutions across shock waves,
we want to obtain the interior and exterior normal traces of entropy-entropy flux fields
of the solutions
on the shock waves by approximating them with smooth surfaces.

The theory of divergence-measure fields in $L^\infty$ has been applied to the analysis of properties of entropy solutions
of nonlinear hyperbolic conservation laws \eqref{codo2};
see for instance \cite{CF1,ChenTorresStructure,ChenTorresNotices}
and the references cited therein, as well as \S \ref{sec:conversionlaws}
below for the details.
Divergence-measure fields also appear in many other areas of analysis, including the study of prescribed mean curvature equations,
the $1$-Laplacian, the continuity equation, and related topics.
We refer to \cite{LeoSaracco1, LeoSaracco2, LeoComi,crippa2024,scheven2016bv,scheven2017anzellotti,scheven2016dual,kawohl2007dirichlet}
and the references therein.
The theory of divergence-measure fields (such as normal traces, Gauss-Green formulas, and product rules,
among others)
provide a mathematical foundation for developing new techniques and tools for entropy methods, measure-theoretic analysis,
partial differential equations, free boundary problems, and related areas.

Furthermore, there are underlying intrinsic connections between divergence-measure fields
and the Cauchy fluxes for the physical balance laws in Continuum Physics.
In this paper, we further analyze such connections and present how the Cauchy fluxes can be
represented by extended divergence-measure fields.
The origin of the study
of Cauchy fluxes dates back to the fundamental paper by Cauchy \cite{Cauchy1}
who considered the balance law in Classical Physics:
\begin{equation}\label{CauchyBalanceLawint}
\int_{U} p(x) \, \d x = \int_{\partial U} f(x, \nu(x)) \,\d \mathcal{H}^{n-1}(x) \qquad
\mbox{for any $U \Subset \Omega$},
\end{equation}
where $\Omega$ is a bounded open set, $\nu$ is the interior unit normal,
the production $p(x)$ is a bounded function in $x \in \Omega$,
and the {\it density function} $f(x,\nu)$ is continuous in $x \in \Omega$.
The balance law postulates that the production of a quantity in any bounded open set $U \Subset \Omega$
is balanced by the flux of this quantity through
$\partial U$.
It was shown
in \cite{Cauchy1} that there exists a continuous vector field $\FF$ such that
\begin{equation} \label{elpuntoint}
  f(x, \nu)= \FF (x) \cdot \nu.
\end{equation}
From classical continuum physics,
it follows that the object of study should be the total flux across a surface $S$ contained in $\partial U$,
that is,
\begin{equation}
\label{objectint}
  \mathcal{F}(S)= \int_{S} f(x, \nu(x)) \,\d \mathcal{H}^{n-1}(x).
\end{equation}
In this paper, we formulate the conditions on the Cauchy flux $\mathcal{F}$
which guarantee the existence of an extended divergence-measure field $\FF$
with the property that $\mathcal{F}(S)$ can be recovered through the normal trace
of $\FF$ on open sets.
This recovery is global for {\it every} open set, and local for {\it almost every open set}.
Since the fundamental work of Cauchy \cite{Cauchy1},
some important
developments
have been made
on the problem of removing the continuity assumption;
see \cite{ChenComiTorres,ctz,degiovanni1999cauchy,gm,Noll2,Sch,S2,S3,Silhavy1,S4,Ziemer1} and
the references cited therein.
We refer the reader to \S \ref{sec:flux_history} for a more detailed description of the history
of Cauchy fluxes and contributions of the aforementioned references.

Classically, the derivation of nonlinear system \eqref{codo2} of conservation laws
is carried out directly from the balance law \eqref{CauchyBalanceLawint} under the assumption
that the vector field is classically differentiable.
This procedure is not rigorous, since the solutions to \eqref{codo2} may be discontinuous, or even measure-valued, in general.
We seek to rigorously derive \eqref{elpuntoint}
under a general framework with
much weaker assumptions.
We consider the generalized balance law:
\begin{equation}
\label{balancelaw3int}
\sigma (U) = \mathcal{F}(\partial U) \qquad\,\, \text{ for any $ U \Subset \Omega$},
 \end{equation}
 where the production $\sigma$ is a Radon measure and the Cauchy Flux $\mathcal{F}$
 is bounded by a Radon measure concentrated on the boundary $\partial U$
 of $U$ (see Definition \ref{DefinitionCauchyFlux}).
 Moreover, we seek to \emph{recover the flux locally} as
\begin{equation}\label{eq:want_localrecovery}
    \mathcal F_U(S) = ( \FF \cdot \nu)_{\partial U}(S),
\end{equation}
valid for \emph{almost all} open $U \Subset \Omega$ and all Borel $S \subset \partial U$,
where the right-hand side is the normal trace of $\FF$ to be made precise
in Definition \ref{defn:normal_trace}.
An extensive analysis has been made when the underlying field is bounded or lies in $\Leb^p$, however a complete treatment in the measure-valued case has remained open, which is one of our motivations.

In this paper, we propose a general formulation for the Cauchy flux in Definition \ref{DefinitionCauchyFlux},
which encapsulates the case of measure-valued fields.
One of our main results, Theorem \ref{maintheorem}, states that the balance law, together with the conditions
imposed on the Cauchy flux, implies the existence of an extended divergence-measure field $\FF$
such that the Cauchy flux can be locally recovered through the normal trace of $\FF$
on the boundary of almost every open set in the sense of \eqref{eq:want_localrecovery}.

The outline of this paper is as follows:
In \S \ref{sec:definitions}, we introduce the distributional normal trace and a product rule
for extended divergence-measure fields.
We show in  \S \ref{sec:disintegration}, more precisely in Theorem \ref{principal},
that the disintegration of measures yields the Gauss-Green formulas for extended divergence-measure fields.
Further properties of this disintegration are studied in \S \ref{sec:prop_disintegration}.
In particular, Theorem \ref{thm:reconstruct_normaltrace} is a coarea-type formula for divergence-measure fields.
In \S \ref{sec:localisation}, the localization properties of the normal trace are developed,
which further motivates our notion of the Cauchy flux in the sequel.
In \S \ref{sec:flux_definition}--\S \ref{sec:flux_localrecovery},
we are devoted to the analysis of the Cauchy flux and the proof of Theorem \ref{maintheorem}.
The extensions of the aforementioned results to general open sets $U \subset \Omega$ are
analyzed in \S \ref{sec:extended_normaltrace}.
In \S \ref{sec:alt_existence}, we discuss the solvability of the equation $-\div \FF = \sigma$.
Finally, in \S \ref{sec:conversionlaws},
we apply the theory of Cauchy fluxes to study the equivalence between entropy solutions
of the multidimensional nonlinear partial differential equations
of balance laws and of the mathematical formulation of physical balance laws
through the constitutive relations in the axiomatic foundation of Continuum Physics.

\section{Extended Divergence-Measure Fields and Distributional Normal Traces}\label{sec:definitions}
We start this section by introducing some basic notation and recalling some properties of Radon measures,
and then introduce divergence-measure fields and the normal traces.

\subsection{Preliminary notions}
Throughout this paper, we work with open subsets $\Omega\subset \mathbb R^n$,
with the standing assumption
that $n \geq 2$ unless otherwise specified.
We use $B_r(x)$ to denote an open ball of radius $r>0$ centered at $x \in \mathbb R^n$; more generally,
we write $B_r(A) = \bigcup_{a \in A} B_r(a)$ for a set $A \subset \mathbb R^n$.
For any set $A \subset \mathbb R^n$, denote the characteristic function of $A$ by $\mathbbm{1}_A$,
and use $A^{\rm c} = \mathbb R^n \setminus A$ to denote the complement
of $A$.
We also write $A \Subset \Omega$ if $\overline{A}$ is compact and $\overline A \subset \Omega$.

For any space $X(\Omega,\mathbb R^n)$ of vector function fields on $\Omega \subset \mathbb R^n$,
denote by $X_{\loc}(\Omega,\mathbb R^n)$ the space of vector function fields $f$ for
which $f \rvert_{\Omega'} \in X(\Omega', \mathbb{R}^{n})$ for
all $\Omega' \Subset \Omega$, and by $X_{\mathrm{c}}(\Omega,\mathbb R^n)$ the space of vector function fields
$f \in X(\Omega,\mathbb R^n)$ which are compactly supported in $\Omega$.
We write $X(\Omega)$ for the space of scalar functions
and define $X_{\loc}(\Omega)$ and $X_{\mathrm{c}}(\Omega)$ analogously.

Denote by $\Lip_{\rm b}(\Omega)$ the space of Lipschitz-continuous functions that are bounded in the sense that
\begin{equation}
    \lVert \phi \rVert_{\Lip_{\rm b}(\Omega)} := \lVert \phi \rVert_{\Leb^{\infty}(\Omega)} + \Lip(\phi) < \infty.
\end{equation}
Rademacher's Theorem implies
that $\Lip_{\rm b}(\Omega) \subset \Sob^{1,\infty}(\Omega)$; however, this inclusion may be strict
for a general open set $\Omega$.

For any open set $U \subset \Omega$ and $\eps>0$, throughout the paper, we use the notation:
\begin{equation}
  U^{\eps} := \{ x \in U\,:\,\dist(x,\partial U) > \eps \}.
\end{equation}
We also set $U^0 = U$ and
\begin{equation}
U^{-\eps} := \{ x \in \mathbb R^n \,:\,\dist(x, U) < \eps \}.
\end{equation}
We see that $\partial U^{\eps}$ are the \emph{interior approximations} of $\partial U$,
while
$\partial U^{-\eps}$ give \emph{exterior approximations}.
We often abbreviate the distance function $d(x) = d_U(x) = \dist(x,\partial U)$
if the underlying set is clear from the context.

By a \emph{standard mollifier} we mean a non-negative function $\rho \in C^{\infty}_{\mathrm{c}}(\mathbb R^n)$,
supported in the unit ball, such that $\int_{\mathbb R^n} \rho(x) \,\d x = 1$.
For $\delta>0$, we set $\rho_{\delta}(x) = \delta^{-n} \rho(\frac{x}{\delta})$
and, for $f \in \Leb^1_{\loc}(\Omega)$, we often write $f_{\delta} = f \ast \rho_{\delta}$
for the mollification defined in $\Omega^{\delta}$.

For any open set $\Omega \subset \mathbb{R}^{n}$,
denote  $\mathcal{M} (\Omega, \mathbb{R}^{n})$ as the space of all finite vector-valued Radon measures on $\Omega$
and $\mathcal M(\Omega)$ as the space of finite signed Radon measures.
If $\mu_k$ is a sequence of vector-valued Radon measures in $\mathcal{M} (\Omega, \mathbb{R}^{n})$,  we use the notation:
\begin{equation}
\mu_k \xrightharpoonup{\,\, * \,\,}
\mu
\end{equation}
to denote that the sequence converges to $\mu$ in the weak*--topology.
If $\mu$ is a Radon measure, denote its total variation as $\lvert\mu\rvert$.
We say that $\mu$ is concentrated in a set $E \subset \mathbb{R}^{n}$ if $\abs{\mu} (\mathbb{R}^{n} \setminus E)=0$.
The support of $\mu$, denoted as $\spt(\mu)$, is the intersection of all closed sets $E$ such that $\mu$ is concentrated on $E$.
In particular,
\begin{equation}
\mathbb{R}^{n} \setminus \spt( \mu)= \{ x \in \mathbb{R}^n\,:\,\abs{\mu} (B_r(x))=0\, \text{ for some } r>0 \}.
\end{equation}

Let $\mu$ and $\lambda$ be non-negative Radon measures on $\mathcal{M}_{\loc}(\Omega)$.
Define $D_{\mu}^+ \lambda: \spt(\mu) \to [0, \infty]$ and $D_{\mu}^-\lambda: \spt(\mu) \to [0, \infty]$ as
$$
D_{\mu}^+ \lambda (x) := \limsup_{r \to 0} \frac{\lambda (B_r(x))}{\mu (B_r(x))},
\quad D_{\mu}^- \lambda(x) := \liminf_{r \to 0} \frac{\lambda (B_r(x))}{\mu (B_r(x))}
\qquad\,\,\,\mbox{for $x \in \spt(\mu)$}.
$$
If $D_{\mu}^+ \lambda (x)=D_{\mu}^- \lambda (x)$,
we denote this value as $D_{\mu} \lambda (x)$.
The function $D_{\mu} \lambda(x)$ is called the $\mu$-density of $\lambda$ at $x$.
We now recall the well-known Lebesgue-Besicovitch Differentiation Theorem
(see for instance \cite[{\S}2.4]{afp} and \cite[{\S}1.6]{eg}) which states
that
\begin{equation*}
\mu(\{ x: D_{\mu}^- \lambda (x) < D_{\mu}^+ \lambda (x) \}) =0,
\qquad
\mu(\{ x: D_{\mu}^{+} \lambda (x) = \infty \}) =0,
\end{equation*}
that is, $D_{\mu} \lambda$ is defined and finite $\mu$--\textit{a.e.}\,\,on $\mathbb{R}^{n}$.

Moreover, $D_{\mu} \lambda$ is Borel measurable and locally $\mu$-integrable, and satisfies
\begin{equation}
\label{Lebesgue-Besicovitch}
\lambda = (D_{\mu} \lambda)\, \mu + \lambda_{\text{sing}}.
\end{equation}
Here the Radon measure $\lambda_{\text{sing}}$ is defined as
\begin{equation*}
\lambda_{\text{sing}}= \lambda \rtangle Y
\end{equation*}
with
\begin{equation*}
Y=  \spt(\mu)^\mathrm{c}\,\mathsmaller{\bigcup}\, \{ x \in \spt(\mu) : D_{\mu}^{+} \lambda (x) = \infty \}
\end{equation*}
that satisfies $\mu(Y) = 0$.
We will make frequent use of the following property:

\begin{lemma}\label{elchiste}
Let $\mu$ and $\lambda$ be non-negative Radon measures on $\Omega$. Then
$D_{\mu} \lambda_{\textnormal{sing}} (x)=0$ for $\mu$--\textit{a.e.}\,\,$\,x \in \Omega$.
\end{lemma}
\begin{proof}
Let $\lambda_{\mathrm{sing}} = \lambda \res Y$ as above. Consider the decomposition:
  \begin{equation}
      \lambda_{\mathrm{sing}} = (D_{\mu}\lambda_{\mathrm{sing}})\, \mu + (\lambda_{\mathrm{sing}})_{\mathrm{sing}}.
  \end{equation}
  Since all the measures are non-negative and
  $\lambda_{\mathrm{sing}}(\mathbb R^n \setminus Y) = (\lambda_{\mathrm{sing}})_{\mathrm{sing}}(\mathbb R^n \setminus Y)=0$,
  it follows that
  \begin{equation}
      \int_{\mathbb R^n \setminus Y} (D_{\mu}\lambda_{\mathrm{sing}})(x) \,\d \mu (x) = 0,
  \end{equation}
so that $D_{\mu}\lambda_{\mathrm{sing}}(x) = 0$ for $\mu$--\textit{a.e.}\,\,$x \in \mathbb R^n \setminus Y$.
Since $Y$ is $\mu$-null, then the conclusion follows.
\end{proof}

The above extends to the case when $\lambda$ is a signed Radon measure by considering
the Jordan decomposition $\lambda = \lambda^+ - \lambda^-$ with $\lambda^+$ and $\lambda^-$ non-negative Radon measures (see \cite[{\S}1.1]{afp}).
Working componentwise, we can also allow for vector-valued measures.

\begin{lemma}\label{lem:measure_coincidence}
Let $\mu$ and $\lambda$ be Radon measures on an open set $\Omega \subset \mathbb R^n$.
Suppose that $\mathcal A$ is a family of Borel subsets which is closed under intersections and generates
the Borel $\sigma$-algebra.
If $\mu(A) = \lambda(A)$ for all $A \in \mathcal A$, then $\mu=\lambda$ as measures.
\end{lemma}

This is shown in \cite[Proposition 1.8]{afp} by noting that any Radon measure is $\sigma$-finite.
While the aforementioned proposition is only stated for non-negative measures,
an inspection of the proof reveals that the result also holds for signed and vector-valued measures.

\subsection{Divergence-measure fields and normal traces}

We are now ready to introduce our central notions of interest:

\begin{definition}
If  $\FF \in \mathcal{M}(\Omega, \mathbb{R}^n)$, define
\begin{equation}
\label{allthetime}
 \lvert\div \FF\rvert (\Omega):= \sup \Big\{ \int_{\Omega}  \nabla \varphi  \cdot \, \dr \FF \,:\, \varphi \in C_{\mathrm{c}}^1(\Omega),\, \lvert\varphi\rvert\le 1\Big\}.
\end{equation}
We say that $\FF$ is an extended divergence-measure field over $\Omega$
if $\FF \in \mathcal{M}(\Omega, \mathbb{R}^n)$ and $\lvert\div \FF\rvert (\Omega) < \infty$.
In this case, the Riesz Representation Theorem implies that $\div \FF \in \mathcal {M} (\Omega)$
and the total variation measure $\abs{\div \FF}$ is computed as
$$
\lvert\div \FF\rvert (A)
= \sup \Big \{ \int_{\Omega}  \nabla \varphi  \cdot \, \dr \FF \,: \, \varphi \in C_{\mathrm{c}}^1(A), \,
\lvert\varphi\rvert\le 1 \Big \} \qquad\,\,\mbox{for every open set $A \subset \Omega$},
$$
and $\abs{\div \FF} (E)= \inf \{ \abs{\div \FF} (A):  A \textnormal{ open set}, \, A \supset E\}$ for arbitrary measurable sets $E$.

If $\FF \in \Leb^{p}(\Omega, \mathbb{R}^n)$ for $1 \leq p \leq \infty$,
then we may view $\FF$ as a measure in the sense{\rm:}
\begin{equation}\label{eq:lp_product_represenation}
\int_{\Omega}  \nabla \varphi  \cdot \, \dr \FF = \int_{\Omega} \nabla \varphi\cdot \FF \,\dr x.
\end{equation}
\end{definition}

The spaces of all $\Leb^{p}$-divergence-measure fields and extended divergence-measure fields
are denoted as $\mathcal{DM}^p(\Omega)$ and $\mathcal{DM}^{\textnormal{ext}}(\Omega)$, respectively.
They are Banach spaces with norms given respectively by
$$
\norm{\FF}_{\mathcal{DM}^p(\Omega)}= \norm{\FF}_{\Leb^p(\Omega,\mathbb{R}^n)} + \lvert\div \FF\rvert(\Omega)
$$
and
$$
\norm{\FF}_{\mathcal{DM}^{\textnormal{ext}}(\Omega)}= \lvert\FF\rvert(\Omega) + \lvert\div \FF\rvert(\Omega).
$$
Since the distributional divergences $\ban{\div \FF,\,\cdot\,}$ of these vector fields are measures,
we see that, for any $\varphi \in C_{\mathrm{c}}^{\infty}(\Omega)$,
\begin{equation}
 \ban{\div \FF,\varphi}= \int_{\Omega} \varphi \, \dr(\div \FF)
 =- \int_{\Omega}  \nabla \varphi\cdot \FF \, \dr x
   \qquad  \mbox{ for $\FF \in \Leb^p(\Omega, \mathbb{R}^n)$},
\end{equation}
and
\begin{equation}
 \ban{\div \FF,\varphi}= \int_{\Omega} \varphi \, \dr(\div \FF) =- \int_{\Omega} \nabla \varphi \cdot \,\dr  \FF
 \qquad  \mbox{ for $\FF \in \mathcal{M}(\Omega, \mathbb{R}^n)$}.
\end{equation}
In what follows, we drop the term ``extended'' and refer to fields $\FF \in \mathcal{DM}^{\ext}(\Omega)$
simply as divergence-measure fields.
For such fields, we can make sense of its normal trace in a distributional sense.

\begin{definition}\label{defn:normal_trace}
Let $\FF \in \mathcal{DM}^{\textnormal{ext}}(\Omega)$, and let $E \Subset \Omega$ be a Borel set.
The \emph{normal trace} of $\FF$ on the boundary of $E$ is defined as
\begin{equation}\label{normal_trace}
    \langle \FF\cdot \nu, \,\phi \rangle_{\partial E} := -\int_{E} \nabla \phi \cdot \d \FF - \int_{E} \phi \,\d(\div \FF)
    \qquad\,\, \mbox{for $\phi \in C_{\mathrm{c}}^1({\Omega})$}.
\end{equation}
Equivalently, we can write \eqref{normal_trace} as an equality of distributions{\rm :}
  \begin{equation}\label{eq:normaltrace_eq}
    \langle \FF \cdot \nu, \,\cdot\,\rangle_{\partial E} =  \div(\mathbbm{1}_E \FF)-\mathbbm{1}_E \div \FF \qquad \text{in } \mathcal{D}'(\Omega).
  \end{equation}
If the distribution \eqref{eq:normaltrace_eq} can be identified as a measure in $\Omega$, then we
denote this measure by $(\FF \cdot \nu)_{\partial E}$.
\end{definition}

Note that the integrals in \eqref{normal_trace} are well-defined for Borel sets $E$, since $\FF$ and $\div \FF$ are Radon measures on $\Omega$.
In the case of a field $\FF \in \mathcal{DM}^p(\Omega)$,  the normal trace is defined by identifying $\FF$ with the vector-valued Radon measure $\FF \mathcal L^n$ as in \eqref{eq:lp_product_represenation}.

\begin{remark}
We have opted to define the normal trace as a distribution in $\Omega$;
while one can more generally consider the normal trace as a distribution on $\mathbb R^n$, we will postpone this discussion to \S \ref{sec:extended_normaltrace}.

Moreover, compared to the prior works such as \cite{ChenComiTorres,CF2,Sch,Silhavy2},
our definition differs by a minus sign, which corresponds to the use of the \emph{interior} unit normal $\nu$ in the classical case.
\end{remark}
\begin{lemma}
For $\FF \in \mathcal{DM}^{\ext}(\Omega)$ and any Borel set $E \Subset \Omega$,
the normal trace $\langle \FF \cdot \nu, \,\cdot \,\rangle_{\partial E}$ is represented
by a measure on $\Omega$ if and only if
    \begin{equation}
        \mathbbm{1}_E \FF \in \mathcal{DM}^{\ext}(\Omega).
    \end{equation}
\end{lemma}

\begin{proof}
By definition of the normal trace, we have
    \begin{equation}
        \langle \FF \cdot \nu,\,\cdot \,\rangle_{\partial E} = \div(\mathbbm{1}_E\FF) -\mathbbm{1}_E\,\div\FF
    \end{equation}
as distributions in $\Omega$.
Since $\mathbbm{1}_E\,\div\FF$ is a Radon measure on $\Omega$,
we see that the normal trace is a measure if and only if $\div(\mathbbm{1}_E\FF)$ is a measure.
This precisely corresponds to imposing that $\mathbbm{1}_E \FF \in \mathcal{DM}^{\ext}(\Omega)$.
\end{proof}

While we can view the normal trace as a distribution in $\Omega$,
analogously to the classical case,
one may wonder if it can be represented by a measure on $\partial E$,
or equivalently as a distribution of order zero.
While counterexamples show this
need not hold
in general (see \cite[Example 2.5]{Silhavy2}),
one can expand the allowed test functions $\phi$ in the definition of normal trace.
This will follow from the following product rule for $\mathcal{DM}^{\ext}$--fields
due to \v{S}ilhav\'y \cite{Silhavy2} based on the notion of \emph{pairings} introduced by Anzelotti \cite{Anzellotti_1983}
(also see \cite{CF1,CF2}).
Far-reaching generalizations of this pairing have recently been obtained by Comi, Cicco \& Scilla \cite{comiextended}.

\begin{theorem}
\label{productrule}
  If $\FF \in \mathcal{DM}^{\textnormal{ext}}(\Omega)$ and $\phi \in \Sob^{1,\infty}(\Omega)$,
  then $\phi \FF \in \mathcal{DM}^{\textnormal{ext}}(\Omega)$ and
  \begin{equation}
   \div(\phi \FF) = \phi \,\div \FF + \overline{\nabla\phi\cdot \FF}, \label{productrule1}
  \end{equation}
  where $\overline{\nabla\phi \cdot \FF}$ is a signed measure on $\Omega$ characterized by
  \begin{equation}\label{eq:pairing_limit}
   \int_{\Omega} \psi\, \d\overline{\nabla\phi \cdot \FF}
   = \lim_{\delta \to 0} \int_{\Omega} \psi \nabla\phi_{\delta} \cdot \d \FF \qquad
   \text{for any }\psi \in C_{\mathrm{c}}(\Omega) ,
  \end{equation}
  and satisfies the estimate{\rm :}
  \begin{equation}\label{eq:product_measure_bound}
      \lvert \overline{\nabla\phi \cdot \FF} \rvert \leq \lVert\nabla\phi\rVert_{\Leb^{\infty}(\Omega)} \lvert \FF \rvert
  \end{equation}
  as measures in $\Omega$.
\end{theorem}

\begin{proof}
Let $\rho_{\delta}$ be a standard mollifer, and set $\phi_{\delta} = \phi \ast \rho_{\delta}$.
We first claim that
\begin{equation}\label{segundo}
    \div(\phi_{\delta} \FF) = \phi_{\delta} \, \div \FF + \nabla \phi_{\delta} \cdot \FF
    \qquad \mbox{in $\Omega^{\delta}$},
\end{equation}
and notice that the right-hand side is well-defined since both terms are a product of
a Borel measure with a continuous function.
To see this, we mollify the field $\FF$ by $\FF_{\eps} = \FF \ast \rho_{\eps}$,
so that the classical product rule
\begin{equation}\label{eq:classical_productrule}
\div(\phi_{\delta} \FF_{\varepsilon})= \phi_{\delta}\, \div \FF_{\varepsilon} + \nabla \phi_{\delta} \cdot \FF_{\varepsilon}
\end{equation}
holds in $\Omega^{\delta}$ for $0<\eps<\delta$.
Let $\psi \in C_{\mathrm{c}}(\Omega)$ be supported in $\Omega^{\delta}$.
Since $\div \FF_{\eps} = (\div \FF) \ast \rho_{\eps}$, $(\psi \phi_{\delta}) \in C_{\mathrm{c}}(\Omega)$,
and the mollifications converge weakly${}^{\ast}$, we have
\begin{equation}
    \lim_{\eps \to 0} \int_{\Omega} \psi \phi_{\delta} \, \div \FF_{\eps} \,\d x = \int_{\Omega} \psi \phi_{\delta} \,\d(\div \FF).
\end{equation}
In particular, $\phi_{\delta}\, \div \FF_{\eps} \weakstarto\phi_{\delta}\, \div \FF$ as measures in $\Omega^{\delta}$.
Similarly, $\nabla \phi_{\delta} \cdot \FF_{\eps} \weakstarto \nabla \phi_{\delta} \cdot \FF$ in $\Omega^{\delta}$.
Thus, since $\div(\phi_{\delta}\FF_{\eps}) \xrightharpoonup{\mathcal{D}'} \div(\phi_{\delta} \FF)$
as distributions in $\Omega^{\delta}$ and this limit is unique, sending $\eps \to 0$ in \eqref{eq:classical_productrule} gives
that $\div(\phi_{\delta}\FF)$ is a measure on $\Omega^{\delta}$ given by \eqref{segundo} as claimed.

We now send $\delta \to 0$. By continuity of $\phi$,
we see that $\phi_{\delta} \to \phi$ uniformly in $\Omega^{\delta_0}$ for each $\delta_0>0$,
which implies that
$\phi_{\delta}\, \div \FF \weakstarto \phi\,\div \FF$ in each $\Omega^{\delta_0}$.
Also, since $\phi$ is Lipschitz, we now show that the second term in \eqref{segundo} is uniformly bounded in $\delta < \delta_0$:
\begin{equation}\label{eq:product_approx_bound}
    \int_{\Omega^{\delta_0}}\lvert\nabla \phi_{\delta}\rvert \, \d \lvert\FF\rvert
        \leq \norm{\nabla \phi_{\delta}}_{\Leb^{\infty}(\Omega^{\delta_0})}\lvert \FF \rvert(\Omega^{\delta_0}) \leq \norm{\nabla\phi}_{\Leb^{\infty}(\Omega)} \lvert \FF \rvert(\Omega) < \infty.
\end{equation}
Thus, for any $\delta_k \to 0$, we can find a further subsequence $\delta_{k_m} \to 0$
for which $\nabla\phi_{\delta_{k_m}}\!\!\cdot \FF$ converges weakly${}^\ast$ in $\Omega^{\delta_0}$
to some limiting measure as $m \to \infty$.
Moreover, for each $\delta_0>0$, we have
\begin{equation}\label{eq:product_distributional_conv}
\nabla\phi_{\delta} \cdot \FF
= \div(\phi_{\delta}\FF) - \phi_{\delta} \, \div \FF \xrightharpoonup{\,\,\mathcal{D}'\,\,} \div(\phi\FF) - \phi \, \div \FF
\end{equation}
as distributions in $\Omega^{\delta_0}$,
so the limiting measure is uniquely determined in $\mathcal D'(\Omega^{\delta_0})$ for each $\delta_0>0$.
Thus, we deduce the existence of a Radon measure in $\Omega$, denoted as $\overline{\nabla \phi \cdot \FF}$, such that
\begin{equation}
\nabla \phi_{\delta}\cdot \FF \xrightharpoonup{\,\, * \,\,}
\overline{ \nabla \phi \cdot \FF} \qquad\text{as } \delta\to0 \text{ in } \Omega^{\delta_0}
\end{equation}
for each $\delta_0>0$.
Then we can upgrade the convergence in \eqref{eq:product_distributional_conv} to hold weakly${}^{\ast}$ as measures, and conclude that $\div (\phi \FF)$ is a measure in $\Omega$ given by
\begin{equation}
\div (\phi \FF)= \phi \,\div \FF + \overline{\nabla \phi \cdot \FF}.
\end{equation}
Finally, to show \eqref{eq:product_measure_bound}, we use \eqref{eq:product_approx_bound} to estimate
\begin{equation}
    \left\lvert \int_{\Omega^{\delta_0}} \psi\, \nabla\phi_{\delta} \cdot \d\FF \right\rvert \leq \norm{\nabla\phi}_{\Leb^{\infty}(\Omega)} \int_{\Omega^{\delta_0}} \psi \,\d\lvert\FF\rvert
\qquad\,\,\mbox{for any $\psi \in C_{\mathrm{c}}(\Omega^{\delta_0})$ and $0<\delta<\delta_0$.}
\end{equation}
Sending $\delta \to 0$ and noting that $\psi$ is arbitrary, we deduce that
\begin{equation}
    \lvert \overline{\nabla\phi\cdot\FF} \lvert \,\res \Omega^{\delta_0} \leq \norm{\nabla\phi}_{\Leb^{\infty}(\Omega)}\lvert \FF \rvert
\end{equation}
as measures. Since this estimate is uniform in $\delta_0$, it holds in $\Omega$.
\end{proof}

\begin{remark}
As the proof illustrates, if $\phi \in C^{1}(\overline\Omega)$, the product rule reduces to
\begin{equation}
\label{productrulesmooth}
\div (\phi \FF)= \phi \, \div \FF + \nabla \phi \cdot \FF,
\end{equation}
where the right-hand side can be understood classically via multiplying a Radon measure by a continuous function.
However, if $\phi$ is merely Lipschitz, then $\nabla \phi$ is only defined almost everywhere.
Since the measure $\FF$ may concentrate on the points of non-differentiablity, it is necessary to understand
the product via a suitable pairing.
Moreover, if $\phi \in C^1_{\rm c}(\Omega)$, we apply \eqref{productrulesmooth} to write the normal trace as
\begin{equation}
    \langle \FF \cdot \nu, \,\phi\rangle_{\partial E} = - \div(\phi \FF)(E) \qquad\text{for any } E \Subset \Omega \text{ Borel}.
\end{equation}
We observe that, owing to Theorem \ref{productrule},
the right-hand side is well defined even when $\phi \in \Sob^{1,\infty}(\Omega)$.
This leads to the following corollary.
\end{remark}

\begin{corollary}\label{cor:normaltrace_lipschitz_extension}
Let $\FF \in \mathcal{DM}^{\ext}(\Omega)$, and let $E \Subset \Omega$ be a Borel set.
Then the normal trace extends to a bounded linear functional on $\Sob^{1,\infty}(\Omega)$ by setting
\begin{equation}
\langle \FF \cdot \nu, \,\phi \rangle_{\partial E}
= -\div(\phi\FF)(E) = -\int_{E} \phi \,\d(\div \FF) - \int_{E} \d \overline{\nabla \phi \cdot \FF}
\end{equation}
for any $\phi \in \Sob^{1,\infty}(\Omega)$.
\end{corollary}

\begin{proof}
By the product rule (Theorem \ref{productrule}), the extension is well-defined
and agrees with Definition \ref{defn:normal_trace} when $\phi \in C^1_{\rm c}(\Omega)$.
Moreover, by linearity of the distributional divergence,
we see that $\langle \FF \cdot\nu, \,\cdot\,\rangle_{\partial E}$ is linear.
Its boundedness follows from the estimate:
    \begin{equation}
    \begin{split}
      \lvert\langle\FF\cdot\nu, \,\phi\rangle_{\partial E}\rvert
      &\leq \int_{\Omega} \lvert\phi\rvert\,\d\lvert\div\FF\rvert + \lvert\overline{\nabla \phi\cdot\FF}\rvert(\Omega) \\
      &\leq \lVert\phi\rVert_{\Leb^{\infty}(\Omega)} \lvert\div \FF\rvert(\Omega) + \lVert \nabla\phi\rVert_{\Leb^{\infty}(\Omega)} \lvert\FF\rvert(\Omega)\\
      &\leq \norm{\phi}_{\Sob^{1,\infty}(\Omega)} \norm{\FF}_{\mathcal{DM}^{\ext}(\Omega)},
    \end{split}
    \end{equation}
    valid for any $\phi \in \Sob^{1,\infty}(\Omega)$, where we have used \eqref{eq:product_measure_bound}.
\end{proof}

In what follows, we
always take this extension when we test the normal trace against a function in $\Sob^{1,\infty}(\Omega)$.
Note in particular that
\begin{equation}\label{eq:trace_balance_law}
    \langle \FF \cdot \nu, \,\mathbbm{1}_{\Omega}\rangle_{\partial U} = -(\div \FF)(U)
    \qquad\text{for any } U \Subset \Omega,
\end{equation}
since $\mathbbm{1}_{\Omega} \in \Sob^{1,\infty}(\Omega)$ is a valid test function and $\overline{\nabla\mathbbm{1}_{\Omega} \cdot \FF} = 0 $.

\begin{lemma}
\label{muyutil}
Let $\FF \in \mathcal{DM}^{\textnormal{ext}}(\Omega)$ with $\spt (\abs{\FF}) \subset \Omega$. Then
\begin{equation}
\div \FF (\Omega) =0.
\end{equation}
\end{lemma}

\begin{proof}
We extend
$\FF$ to a measure
$\tilde{\FF}$ in $\mathbb{R}^{n}$ by setting $\lvert\tilde{\FF}\rvert (\mathbb{\R}^{n} \setminus \Omega)=0$.
We claim that $\tilde{\FF} \in \mathcal{DM}^{\ext}(\mathbb R^n)$
and that $\div \tilde{\FF}$ is the zero-extension of $\div \FF$ to $\mathbb R^n$.
Indeed, for $\phi \in C^{\infty}_{\rm c}(\mathbb R^n)$,
$\spt(\phi) \cap \spt(\lvert \FF \rvert) \subset \Omega$ is compact,
so that we can find $\chi \in C_{\mathrm{c}}^{\infty}(\Omega)$ such that $\chi = 1$
in a neighborhood of $\spt(\phi) \cap \spt(\lvert \FF \rvert)$.
Then $\chi \phi \in C^{\infty}_{\rm c}(\Omega)$ and
\begin{equation}
 \int_{\mathbb R^n} \nabla \phi \cdot \, \d \tilde\FF
 = \int_{\Omega} \nabla(\chi\phi) \cdot \,\d\FF = - \int_{\Omega} \chi\phi \,\d(\div \FF) = - \int_{\Omega} \phi \,\d(\div \FF)
\end{equation}
by our choice of $\chi$.
Since $\phi$ was arbitrary, it follows that $\tilde \FF \in \mathcal{DM}^{\ext}(\Omega)$ with the claimed divergence.

Now, let $\chi_k \in C^{\infty}_{\rm c}(\Omega)$ such that $\mathbbm{1}_{B_k(0)} \leq \chi \leq \mathbbm{1}_{B_{k+1}(0)}$
in $\mathbb R^n$ and $\lVert \chi_k \rVert_{\Leb^{\infty}(\Omega)} \leq 2$.
By definition of the distributional derivative,
\begin{equation}
    \int_{\Omega} \chi_k\, \d(\div \FF) = \int_{\mathbb R^n} \chi_k \,\d(\div \tilde \FF) = - \int_{\mathbb R^n} \nabla \chi_k \cdot \d\tilde \FF.
\end{equation}
Since $\chi_k$ is uniformly bounded in $\Leb^{\infty}$ and converges pointwise to zero as $k \to \infty$,
the Dominated Convergence Theorem leads to
\begin{equation}
    \lim_{k \to \infty} \int_{\mathbb R^n} \nabla \chi_k \cdot \d\tilde \FF = 0.
\end{equation}
Similarly, since $\chi_k$ converges pointwise to $\mathbbm{1}_{\mathbb R^n}$ and $\div \FF$ is a finite measure,
we have
\begin{equation}
    \div \FF(\Omega) = \lim_{k \to \infty} \int_{\Omega} \chi_k \,\d(\div \FF) = \lim_{k \to \infty} \int_{\mathbb R^n} \nabla \chi_k \cdot \d\tilde \FF = 0,
\end{equation}
as required.
\end{proof}

\begin{corollary}\label{cor:normaltrace_closure}
Let $\FF \in \mathcal{DM}^{\ext}(\Omega)$. Then, for any open set $U \Subset \Omega$,
\begin{equation}
 \langle \FF \cdot \nu, \,\phi \rangle_{\partial\overline U}
 = \int_{\Omega \setminus \overline U} \d \overline{\nabla\phi\cdot \FF}
 + \int_{\Omega \setminus \overline U} \phi \,\d(\div \FF) \qquad\text{for any } \phi \in \Sob_{\mathrm{c}}^{1,\infty}(\Omega).
\end{equation}
\end{corollary}

\begin{proof}
 By Lemma \ref{muyutil} and the product rule, for $\phi \in \Sob^{1,\infty}_{\mathrm{c}}(\Omega)$,
 we see that $\phi \FF \in \mathcal{DM}^{\ext}(\Omega)$ is compactly supported, so that
    \begin{equation}
        0 = \int_{\Omega} \d(\div(\phi \FF))
        = \int_{\Omega} \d\overline{\nabla\phi\cdot \FF} + \int_{\Omega} \phi \,\d(\div \FF).
    \end{equation}
By splitting the latter integrals to integrate over $\overline U$ and $\Omega \setminus \overline U$ respectively,
we deduce
 \begin{equation}
 0 = - \langle \FF \cdot \nu, \,\phi \rangle_{\partial\overline U}
   + \int_{\Omega \setminus \overline U} \d\overline{\nabla\phi\cdot \FF}
     + \int_{\Omega \setminus \overline U} \phi \,\d(\div \FF),
    \end{equation}
    from which the result follows.
\end{proof}

\begin{remark}\label{rem:inner_outer_trace}
Using Corollary \ref{cor:normaltrace_closure},
we can interpret $-\langle \FF \cdot \nu, \,\cdot\,\rangle_{\partial\overline U}$
as the \emph{exterior normal trace} of $\FF$ on $\partial U$.
Since
\begin{equation}\label{eq:trace_jump}
    \langle \FF \cdot \nu, \,\phi \rangle_{\partial U} - \langle \FF \cdot \nu, \,\phi \rangle_{\partial \overline U}
    =  \int_{\partial U} \d \overline{\nabla\phi\cdot \FF} + \int_{\partial U} \phi\, \d(\div \FF),
\end{equation}
the interior and exterior traces agree $($up to a sign$)$, provided that
$\lvert \FF \rvert(\partial U) = \lvert \div \FF \rvert(\partial U) = 0$,
by using the fact that $\lvert \overline{\nabla\phi\cdot \FF}\rvert \ll \lvert \FF \rvert$ by \eqref{eq:product_measure_bound}.
In particular, this holds for $\partial U^{\eps}$ for all but countably many $\eps>0$.
In general, this need not hold, for which we say that there is a \emph{jump} across the boundary $\partial U$
if the two traces do not coincide.
We emphasize that this does not occur if $\abs{\div \FF} \ll \mathcal{L}^{n}$ and $\abs{\FF} \ll \mathcal{L}^{n}$,
provided that $\mathcal L^n(\partial U) = 0$,
so it is necessary to consider measure-valued fields to model such phenomena.
\end{remark}

\begin{example}\label{eq:boundary_jump}
We now show, by means of an example, that the jump in \eqref{eq:trace_jump} can occur
due to the concentration of field $\FF$ itself.
For this, consider $\Omega = \mathbb R^2$ and
\begin{equation}
 \FF = e_1 \, \mathcal H^1 \res \{ x \in \mathbb R^2 : x_2 = 0 \},
\end{equation}
where $e_1 = (1,0)$.
Since $\FF$ is a Radon measure on $\mathbb R^2$ and
\begin{equation}
\int_{\mathbb R^2} \nabla \phi \cdot \d\FF = \int_{\mathbb R} \partial_{x_1}\phi(x_1,0) \,\d x_1 = 0
\qquad \mbox{ for any $\phi \in C^{1}_{\mathrm{c}}(\mathbb R^2)$},
\end{equation}
it follows that $\FF \in \mathcal{DM}^{\ext}_{\loc}(\Omega)$.
Then, considering $Q = (0,1)^2 \Subset \mathbb R^2$,  we have
\begin{align}
\langle \FF \cdot \nu, \,\phi \rangle_{\partial Q} = 0 \qquad \mbox{ for $\phi \in C^{1}_{\mathrm{c}}(\mathbb R^2)$},
\end{align}
since $\supp(\FF) \cap Q = \varnothing.$
On the other hand, we have
\begin{equation}
\langle \FF \cdot \nu, \,\phi \rangle_{\partial \overline Q}
= \int_0^1 \partial_{x_1} \phi(x_1,0) \,\d x_1 = \phi(1,0) - \phi(0,0).
 \end{equation}
This implies that the normal traces on $Q$ and $\overline Q$ are represented by measures such that
\begin{equation}
(\FF \cdot \nu)_{\partial \overline Q} = \delta_{(1,0)}-\delta_{(0,0)} \neq 0 = (\FF \cdot \nu)_{\partial Q},
\end{equation}
exhibiting a jump across the boundary $\partial Q$,
despite the fact that $\div \FF = 0$ in $\mathbb R^2$.
This is in contrast to the $\mathcal{DM}^p$ setting for which,
for sufficiently regular domains ({\it i.e.}, when $\mathcal L^n(\partial U) = 0$),
there is a jump across the boundary if and only if $\lvert\div \FF\rvert(\partial U) \neq 0$.
\end{example}

We can also infer from Lemma \ref{muyutil} that the normal trace is supported on $\partial E$.
We state this result by using the language of distributions.

\begin{lemma}\label{soporte}
Let $\FF \in \mathcal{DM}^{\textnormal{ext}}(\Omega)$, and let $E \Subset \Omega$ a Borel set.
Then the normal trace $\langle \FF\cdot \nu, \,\phi \rangle_{\partial E}$ is a distribution of order $1$
supported on $\partial E$.
\end{lemma}

\begin{proof}
For $\phi \in C^1_{\mathrm{c}}(\Omega)$, we can estimate
\begin{equation}
\begin{split}
    \left\lvert\langle \FF \cdot \nu, \,\phi \rangle_{\partial E}\right\rvert
    &\leq \int_{\Omega} \lvert \phi \rvert \,\d\lvert\div \FF\rvert + \int_{\Omega} \lvert\nabla \phi\rvert \, \d \lvert\FF\rvert\\
    &\leq \lVert \phi \rVert_{\Leb^{\infty}(\Omega)} \lvert\div\FF\rvert(\Omega) + \lVert\nabla\phi\rVert_{\Leb^{\infty}(\Omega)} \lvert\FF \rvert(\Omega),
\end{split}
\end{equation}
from which it follows that the normal trace is a distribution of order $1$ in $\Omega$.

For the second part, let $\phi \in C^1_{\mathrm{c}}(\Omega)$ be supported in $\Omega \setminus \partial E$.
Since $\partial E$ is compact, $\dist(\spt(\phi),\partial E)>0$ so that
$V := \spt(\phi) \cap E \Subset E$.
Then, by the product rule, we have
\begin{equation}\label{eq:trace_div_V}
\langle \FF \cdot \nu, \,\phi \rangle_{\partial E} = -\div(\phi \FF)(E) = -\div(\phi \FF)(V),
\end{equation}
since $\phi$ vanishes in the relatively open set $E \setminus V$.
Now, let $\chi \in C^1_{\mathrm{c}}(\Omega)$ be a cutoff such that $\chi \equiv 1$ on $V$ and vanishes on $\Omega \setminus \overline{E}$.
Then $\chi \phi \FF \in \mathcal{DM}^{\ext}(\Omega)$ is compactly supported,
agrees with $\phi \FF$ in a neighborhood of $E$, and vanishes outside $V$.
Applying Lemma \ref{muyutil} leads to
\begin{equation}
    \div(\phi \FF)(V) = \div(\chi\phi\FF)(V) = \div(\chi\phi \FF)(\Omega) = 0.
\end{equation}
We combine this  with \eqref{eq:trace_div_V} to conclude the result.
\end{proof}

We can further refine Lemma \ref{soporte} to allow for test functions that vanish merely on $\partial E$.
The following is due to
\cite{Silhavy2} for the case of open sets;
we record the proof for completeness and observe that the result applies to more general cases.

\begin{theorem}\label{prop:normaltrace_support}
Let $\FF \in \mathcal{DM}^{\ext}(\Omega)$, and let $E \Subset \Omega$ be either open or closed,
or Borel satisfying $\lvert \FF \rvert(\partial E) = 0$.
Then, if $\phi \in \Sob^{1,\infty}(\Omega)$ vanishes on $\partial E$,
    \begin{equation}
        \langle \FF \cdot \nu, \,\phi \rangle_{\partial E}  = 0.
    \end{equation}
\end{theorem}

\begin{proof}
We first consider the case when $\spt(\phi) \cap \partial E = \varnothing$;
in this case, the result follows by arguing exactly as in Lemma \ref{soporte}.
For $\phi \in \Sob^{1,\infty}(\Omega)$, we see that \eqref{eq:trace_div_V} remains true
by Corollary \ref{cor:normaltrace_lipschitz_extension}, so the identical argument goes through.

For general $\phi \in \Sob^{1,\infty}(\Omega)$ vanishing on $\partial E$, we
reduce to the first case via an approximation argument.
For $\delta>0$, define
\begin{equation}\label{eq:ddelta_defn}
        d_{\delta}(x) = \begin{cases}
            0 & \text{ if } 0 \leq \dist(x,\partial E) < \delta,\\
            \frac1{\delta}(\dist(x,\partial E) - \delta) & \text{ if } \delta \leq \dist(x,\partial E) < 2\delta,\\
            1 & \text{ if } \dist(x,\partial E) \geq 2\delta.
        \end{cases}
\end{equation}
Then $d_{\delta}$ is $\frac1{\delta}$-Lipschitz and vanishes in a neighborhood of $\partial E$.
Since $d_{\delta} \phi$ is supported away from $\partial E$, we know that
\begin{equation}\label{eq:tracezero_delta}
        0 = \langle \FF \cdot \nu, \,d_{\delta}\phi \rangle_{\partial E} = \int_{E} d_{\delta}\phi \,\d(\div \FF) + \int_{E} \d\overline{\nabla(d_{\delta}\phi)\cdot\FF}
\end{equation}
by above.
We now argue that we can pass to the limit as $\delta \to 0$.

This can be seen as follows:  Since $\phi$ vanishes on $\partial E$, by the Lipschitz property, we have
\begin{equation}
\sup_{B_{2\delta}(\partial E)} \lvert \phi(x)\rvert \leq  2\delta \lVert\nabla\phi\rVert_{\Leb^{\infty}(\Omega)},
\end{equation}
where $B_{2\delta}(\partial E) = \{ x \in \mathbb R^n : \dist(x,\partial E) < 2\delta\}$.
By the product rule (namely, using \eqref{eq:pairing_limit}), we see that, for any $\psi \in C_{\mathrm{c}}(\Omega)$,
    \begin{equation}
    \begin{split}
        \int_{\Omega} \psi \,\d\overline{\nabla(d_{\delta}\phi)\cdot\FF}
        &= \lim_{\eps \to 0} \int_{\Omega} \psi \nabla(d_{\delta}\phi)\ast\rho_{\eps} \cdot \d\FF\\
        &= \lim_{\eps \to 0} \int_{\Omega} \psi (d_{\delta}\nabla \phi + \phi \nabla d_{\delta}) \ast \rho_{\eps} \cdot \d \FF \\
        &= \int_{\Omega} \psi \,d_{\delta} \,\d\overline{\nabla \phi \cdot \FF}
         + \int_{\Omega} \psi \,\phi \,\d\overline{\nabla d_{\delta} \cdot \FF},
    \end{split}
    \end{equation}
which implies
    \begin{equation}
        \overline{\nabla(d_{\delta}\phi)\cdot\FF} = d_{\delta}\,\overline{\nabla \phi \cdot \FF} + \phi\,\overline{\nabla d_{\delta} \cdot \FF} \qquad\,\,\mbox{as measures in $\Omega$}.
    \end{equation}
Moreover, since $\spt(\nabla d_{\delta})\Subset
A_{\delta} := \{ \delta \leq \dist(x,\partial E) \leq 2\delta \} \subset B_{2\delta}(\partial E)$,
$\overline{\nabla d_{\delta} \cdot \FF}$ is also supported in $A_{\delta}$.
Hence, using \eqref{eq:product_measure_bound}, we can estimate
\begin{equation}\label{eq:ddelta_zero_conv}
\begin{split}
\left\lvert\int_{E} \phi\,\d\overline{\nabla d_{\delta} \cdot \FF}\right\rvert
&\leq \sup_{B_{2\delta}(\partial E)}\lvert\phi\rvert \lVert \nabla d_{\delta} \rVert_{\Leb^\infty(\Omega)}\lvert\FF \rvert(A_{\delta})\\
&\leq 2 \lVert\nabla\phi\rVert_{\Leb^\infty(\Omega)} \lvert \FF \rvert(A_{\delta}),
\end{split}
\end{equation}
which vanishes as $\delta\to0$ since $\displaystyle\limsup_{\delta\to0} A_{\delta} = \varnothing$.
Thus, passing to the limit in \eqref{eq:tracezero_delta} and
noting that $d_{\delta} \to \mathbbm{1}_{\Omega \setminus \partial E}$ pointwise,
    \begin{equation}\label{eq:delta_limit_zerotrace}
    \begin{split}
        0 &= \lim_{\delta \to 0} \langle \FF \cdot \nu, \,d_{\delta}\phi\rangle_{\partial E}\\
        &= \lim_{\delta \to 0} \int_E d_{\delta}\phi\, \d(\div \FF) + \lim_{\delta \to 0} \int_E d_{\delta} \,\d\overline{\nabla\phi\cdot\FF} \\
        &= \int_{E \setminus \partial E} \phi\, \d(\div \FF)
        + \int_{E \setminus \partial E} \d\overline{\nabla\phi\cdot\FF},
    \end{split}
    \end{equation}
by \eqref{eq:ddelta_zero_conv} and the Dominated Convergence Theorem.

Moreover, since $\phi$ vanishes in $\partial E$, it follows that
    \begin{equation}
        \int_{E \setminus \partial E} \phi\,\d(\div \FF) = \int_E \phi\,\d(\div \FF).
    \end{equation}
Now, if $E$ is open, then $E \setminus \partial E = 0$ and \eqref{eq:delta_limit_zerotrace}
implies that $\langle \FF \cdot \nu, \,\phi \rangle_{\partial E} = 0$.
If $E$ is closed, the same argument applies to $\Omega \setminus E$ so that, by Corollary \ref{cor:normaltrace_closure},
\begin{equation}
\langle \FF \cdot \nu \rangle_{\partial E}
= \int_{\Omega\setminus E} \phi \,\d(\div \FF)
+ \int_{\Omega \setminus E} \d \overline{\nabla \phi\cdot\FF} = 0.
\end{equation}
Finally, if $E$ is merely Borel with $\lvert \FF \rvert(\partial E) = 0$,
then, since $\overline{\nabla \phi \cdot \FF} \ll \lvert\FF\rvert$ by \eqref{eq:product_measure_bound}, it follows that
    \begin{equation}
        \int_{E \cap \partial E} \d\overline{\nabla \phi\cdot\FF} = 0.
    \end{equation}
Combining this with \eqref{eq:delta_limit_zerotrace}, we conclude the proof.
\end{proof}

\begin{remark}
We expect that the conclusion of Theorem \ref{prop:normaltrace_support} holds
for any Borel set $E \Subset \Omega$.
However, we were unable to establish this without a further topological
or measure-theoretic condition to ensure that $\lvert \overline{\nabla \phi\cdot\FF} \rvert(E \cap \partial E) = 0$.
Note that, for $\mathcal{DM}^p(\Omega)$--fields $\FF$,
under the mild regularity condition that $\mathcal L^n(\partial E)= 0$,
the condition that $\lvert \FF \rvert(\partial E) = \int_{\partial E} \lvert\FF\rvert\,\d\mathcal L^n= 0$
is always satisfied.
\end{remark}

\begin{corollary}\label{eq:normaltrace_lipschitz_dual}
Let $\FF \in \mathcal{DM}^{\ext}(\Omega)$, and let $U \Subset \Omega$ be open.
    Then there is a linear functional $N_U \in \Lip_{\mathrm{b}}(\partial U)^{\ast}$ satisfying
    \begin{equation}
        \langle \FF \cdot \nu, \,\phi \rangle_{\partial U}
        = N_U(\phi\rvert_{\partial U}) \qquad\text{for any } \phi \in \Lip_{\rm b}(\Omega).
    \end{equation}
\end{corollary}

\begin{proof}
Given $\phi_0 \in \Lip_{\mathrm{b}}(\partial U)$, let $\phi \in \Lip_{\mathrm{b}}(\Omega)$ be an extension of $\phi_0$ to $\Omega$; this is possible by \cite[{\rm Proposition}\,2.12]{afp}
and multiplying the obtained extension by the Lipschitz cutoff
$\chi(x) = \max\{0,1-\dist(x,\partial U)\}$ to ensure that it is bounded.
We then define
\begin{equation}
N_U(\phi_0) := \langle \FF \cdot \nu, \phi \rangle_U.
\end{equation}
Note that this extension can be chosen to satisfy that
$\lVert\phi\rVert_{\Lip_{\mathrm{b}}(\Omega)}
\leq C\lVert\phi\rVert_{\Lip_{\mathrm{b}}(\partial U)},$
so it follows from the boundedness of
$\langle \FF \cdot \nu, \,\cdot \,\rangle_{\partial U}$ that $N_U$ is bounded.
To show that it is well-defined, observe that,
if $\phi_1, \phi_2 \in \Lip_{\mathrm{b}}(\Omega) \subset \Sob^{1,\infty}(\Omega)$ are two such extensions,
then $\phi_1-\phi_2$ vanishes on $\partial U$ so that
\begin{equation}
\langle \FF \cdot \nu, \,\phi_1 \rangle_U - \langle \FF \cdot \nu, \,\phi_2 \rangle_U
= \langle \FF \cdot \nu, \,\phi_1 - \phi_2 \rangle = 0,
\end{equation}
by Theorem \ref{prop:normaltrace_support}.
Therefore, $N_U$ is well-defined, independent of the choice of extension.
\end{proof}

\section{Representation and Limit Formula for the Normal Trace via Disintegration}\label{sec:disintegration}

In this section, we show the normal trace admits a measure representation
on \emph{almost every} open set and derive a limit formula in the general case.
A key tool in our analysis is the disintegration of measures,
which we apply in the following form:

\begin{theorem}\label{prop:disintegration}
Let $\mu$ be a finite Radon measure on an open set $U \subset\mathbb R^n$ such that $U \neq \varnothing, \mathbb R^n$,
and let $d=\dist(\,\cdot\,,\partial U)$ be the distance function.
Then there exist both a non-negative and finite Radon measure $\tau$ on $(0, \infty)$
and a family of measures $\mu_t$ in $U$ such that the mapping{\rm :} $t \mapsto \mu_t$ is $\tau$-measurable.
For $\tau$--\textit{a.e.}\,\,$t \in (0,\infty)$, $\mu_t$ is supported on $d^{-1} (t)$ with $\abs{\mu_t}(d^{-1}(t)) = 1$ and
  \begin{equation}
    \mu(A \cap U^{t_1}\setminus \overline{U}^{t_2}) = \int_{t_1}^{t_2} \int_{d^{-1} (t)} \chi_A(x) \,\d\mu_t(x)\,\d\tau(t),
  \end{equation}
where
the $\tau$-integral is understood to be over the open interval $(t_1,t_2)$.
Furthermore, for any bounded Borel function $\phi$ on $U$,
  \begin{equation}
  \label{localdisintegration}
    \int_U \phi(x) \,\d \mu(x) = \int_0^{\infty} \int_{d^{-1}(t)} \phi(x) \,\d\mu_t(x)\,\d \tau(t).
  \end{equation}
We write
  \begin{equation}\label{eq:disint_notation}
      \mu = \tau \otimes_{\partial U^t} \mu_t
  \end{equation}
  as a shorthand for this decomposition.
\end{theorem}

\begin{proof}
We define the function $\Phi\colon  U \to (0, \infty) \times U$ as $\Phi(x)=(d(x),x)$.
Then the push-forward measure $\Phi_{\#}\mu$ is a measure on $(0, \infty) \times U$.
By the Disintegration Theorem presented
in \cite[{\rm Theorem}~2.28]{afp}
applied to
$\Phi_{\#}\mu$,
it follows that there exists a family of measures $\mu_t$ satisfying
\begin{equation}
\label{clave}
\Phi_{\#}\mu = \tau \otimes \mu_t, \qquad \abs{\Phi_{\#}\mu} = \tau \otimes \abs{\mu_t},
\end{equation}
where $\tau$ is a measure on $(0, \infty)$ defined as
\begin{equation}
\tau= \pi_{\#} (\abs{\Phi_{\#}\mu})\qquad  \text{ with $\pi\colon (0, \infty) \times U \to (0, \infty)$
and $\pi(t,x)=t$},
\end{equation}
and $\mu_t$ is a family of measures on $U$
with $\abs{\mu_t}(U)=1$ for $\tau$--\textit{a.e.}\,\,$t \in (0, \infty)$.
For a Borel set $A \subset U \setminus \overline U^t$,
since $\Phi$ is injective, $\Phi^{-1}(\Phi(A)) = A$ so that
\begin{equation}\label{eq:disintegration_intermediate}
   \mu(A) = \Phi_{\#}\mu(\Phi(A)) = \int_0^t \int_U \chi_{\Phi^{-1}(A)}(t,x)\,\d \mu_t(x)\,\d \tau(t).
\end{equation}
We now claim that, for $\tau$--\textit{a.e.}\,\,$t \in (0,\infty)$,
$\mu_t$ is supported on $\partial U_t$.
Indeed, for any $t_1 < t_2$, we have
\begin{align}
\tau((t_1,t_2))&= \pi_{\#} (\abs{\Phi_{\#}\mu}) ((t_1,t_2))=\abs{\Phi_{\#}\mu}(\pi^{-1}((t_1,t_2))) \nonumber\\
&= \abs{\Phi_{\#}\mu}((t_1,t_2)\times U)= \Phi_{\#}\abs{\mu}((t_1,t_2)\times U) \nonumber \\
&= \abs{\mu} (\Phi^{-1}((t_1,t_2) \times U))= \abs{\mu}(U^{t_2}\setminus \overline{U}^{t_1}), \label{aqui}
\end{align}
and
\begin{equation}
\abs{\mu} (U^{t_2}\setminus \overline{U}^{t_1})=
\abs{\Phi_{\#}\mu}\big((t_1,t_2) \times (U^{t_2}\setminus \overline{U}^{t_1})\big)=\int_{t_1}^{t_2}\int_{U^{t_2} \setminus \overline{U^{t_1}}}  \,\d\abs{\mu_t}
\,\d\tau(t).
\end{equation}
Combining the above, we infer that
\begin{equation}
    1 = \frac{\lvert\mu\rvert(U^{t_2} \setminus \overline{U^{t_1}})}{\tau((t_1,t_2))} = \frac1{\tau((t_1,t_2))} \int_{t_1}^{t_2} \lvert\mu_t\rvert(U^{t_2} \setminus \overline{U^{t_1}})\,\d \tau(t).
\end{equation}
Let $t>0$ and $0 < \eps < \delta$. Then we can employ the above
with $t_1 = t-\eps$ and $t_2 = t+\eps$, and apply that
$U^{t+\eps} \setminus \overline{U^{t-\eps}}\subset U^{t+\delta} \setminus \overline{U^{t-\delta}}$
to obtain
\begin{equation}
    \frac1{\tau((t-\eps,t+\eps))} \int_{t-\eps}^{t+\eps} \lvert \mu_s \rvert(U^{t+\delta} \setminus \overline{U^{t-\delta}}) \,\d \tau(s) \geq 1.
\end{equation}
Since the function:
\begin{equation}\label{eq:mut_partialmap}
    s \mapsto \lvert\mu_s \rvert(U^{t+\delta} \setminus \overline{U^{t-\delta}})
\end{equation}
is $\tau$-measurable for $t \geq 0$ by \cite[(2.19)]{afp},
there is a $\tau$-null set $\mathcal N \subset [0,\infty)$
such that every $t \notin \mathcal N$ is a Lebesgue point for \eqref{eq:mut_partialmap}
for each $\delta = \frac{1}{k}$ with $k \in \mathbb N$.
Then, for such  $t$, we have
\begin{equation}
    1 \leq \lim_{\eps \to 0} \frac1{\tau((t-\eps,t+\eps))} \int_{t-\eps}^{t+\eps} \lvert \mu_s \rvert(U^{s+\frac1k} \setminus \overline{U^{s-\frac1k}}) \,\d \tau(s) = \lvert \mu_t \rvert(U^{t+\frac1k} \setminus \overline{U^{t-\frac1k}}).
\end{equation}
Since this holds for all $k$, we infer that
\begin{equation}
    \lvert \mu_t \rvert(\partial U^t) = \lim_{k \to \infty} \lvert\mu_t\rvert(U^{t+\frac1k} \setminus\overline{U^{t-\frac1k}}) \geq 1.
\end{equation}
However, since $\lvert \mu_t\rvert(U)=1$, it follows that $\mu_t$ is supported on $\partial U^t$.

Since $\tau$-almost every $\mu_t$ is supported on $\partial U^t$, \eqref{eq:disintegration_intermediate} simplifies to give
\begin{equation}
    \mu(A) = \int_0^t \int_{\partial U^t} \chi_A(x)\,\d\mu(x)\,\d \tau(t),
\end{equation}
which is what we set out to prove.
Moreover, an approximation argument implies \eqref{eq:disint_notation}.
\end{proof}

\begin{lemma}\label{lem:disint_polar}
Let $\mu$ be a Radon measure on an open set $U \subset \mathbb R^n$.
Consider the decomposition
  \begin{equation}
      \mu = \tau \otimes_{\partial U^t} \mu_t
  \end{equation}
  from  {\rm Theorem \ref{prop:disintegration}}.
  Then, for $\tau$--\textit{a.e.}\,\,$t \in [0,\infty)$,
   \begin{equation}
    D_{\lvert \mu\rvert}\mu(x) = D_{\lvert \mu^t\rvert}\mu^t(x),
    \qquad\,\,\mbox{$\lvert \mu^t\rvert$--\textit{a.e.}\,\,on $\partial U^t$}.
  \end{equation}
\end{lemma}

\begin{proof}
We know from the above proof that
$\Phi_{\#}\mu = \tau \otimes \mu_t$
and $\lvert \Phi_{\#}\mu\rvert  = \tau \otimes \lvert \mu_t \rvert$.
Thus, if $\phi$ is a bounded Borel-measurable function on $U$, we have
  \begin{equation}
      \int_U \phi \,\d \mu = \int_0^{\infty} \int_{\partial U^t} \phi \,\d\mu_t\,\d\tau= \int_0^{\infty} \int_{\partial U^t} \phi \, D_{\lvert\mu_t\rvert}\mu_t \,\d \lvert\mu_t\rvert\,\d\tau,
  \end{equation}
  and
  \begin{equation}
      \int_U \phi \,\d \mu = \int_U \phi\,D_{\lvert\mu\rvert}\mu \,\d\lvert\mu \rvert = \int_0^{\infty} \int_{\partial U^t} \phi \, D_{\lvert\mu\rvert}\mu  \,\d\lvert\mu_t\rvert\,\d\tau.
  \end{equation}
Now, replacing $\phi$ by $\mathbbm{1}_{U^{t_2} \setminus \overline{U^{t_1}}}\phi$
and then combining the above, we obtain
  \begin{equation}
      \int_{t_1}^{t_2} \int_{\partial U^t} \phi\, D_{\lvert\mu_t\rvert}\mu_t \,\d \lvert\mu_t\rvert\,\d\tau = \int_{t_1}^{t_2} \int_{\partial U^t} \phi\,D_{\lvert\mu\rvert}\mu  \,\d\lvert\mu_t\rvert\,\d\tau.
  \end{equation}
  Let $\{\phi_j\}$ be a countable and uniformly dense subset of $C_{\mathrm{c}}(U)$.
  Then, for $\tau$--\textit{a.e.}\,\,$t \geq 0$,
  \begin{equation}
      \begin{split}
      \int_{\partial U^t} \phi_j\,  D_{\lvert\mu_t\rvert}\mu_t \,\d \lvert\mu_t\rvert
      &= \lim_{\eps \to 0} \int_{t-\eps}^{t+\eps} \int_{\partial U^s} \phi_j \,D_{\lvert\mu_s\rvert}\mu_s \,\d \lvert\mu_s\rvert\,\d\tau \\
      &=\lim_{\eps \to 0} \int_{t-\eps}^{t+\eps} \int_{\partial U^s} \phi_j\,D_{\lvert\mu\rvert}\mu \,\d \lvert\mu_s\rvert\,\d\tau \\
      &= \int_{\partial U^t} \phi_j\,D_{\lvert\mu\rvert}\mu \,\d \lvert\mu_t\rvert.
      \end{split}
  \end{equation}
  By density of $\{\phi_j\}$, it follows that, for any such $t$,
   \begin{equation}
      (D_{\lvert\mu_t\rvert}\mu_t)  \lvert\mu_t\rvert  = (D_{\lvert\mu\rvert}\mu)  \lvert\mu_t\rvert
      \qquad\,\,\mbox{as measures on $\partial U^t$},
  \end{equation}
  from which the result follows.
\end{proof}

Let $\tau$ be a non-negative Radon measure on $\mathbb{R}$.
In the subsequent proof, we apply the Lebesgue-Besicovitch Differentiation Theorem to write
\begin{equation}
\label{primera}
 \tau= (D_{\mathcal{L}^1} \tau) \mathcal{L}^1 + \tau_{\text{sing}},
\end{equation}
where the Radon measure $\tau_{\text{sing}}$ satisfies
\begin{equation}
\tau_{\text{sing}}= \tau \,\, \rightangle Y_1,  \qquad Y_1= \{ D_{\mathcal{L}^1}^+ \tau = \infty \}, \qquad \mathcal{L}^1 (Y_1)=0.
\end{equation}
By Lemma \ref{elchiste}, $D_{\mathcal L^1}\tau_{\mathrm{sing}}(t) = 0$
for $\mathcal L^1$--\textit{a.e.}\,\,$t>0$.

\begin{theorem}\label{principal}
Let $F\in \mathcal{DM}^{\textnormal{ext}}(\Omega)$, and let $U \Subset \Omega$ be open.
Then there is an $\mathcal L^1$--null set $\mathcal N \subset (0,\infty)$ such that,
for any $\eps \notin \mathcal N$, there exists a measure
$(\FF \cdot \nu)_{\partial U^{\varepsilon}}$ supported on $\partial U^{\varepsilon}$ such that
\begin{equation}
\langle \FF \cdot \nu, \,\phi \rangle_{\partial U^{\varepsilon}}
= \int_{\partial U^{\varepsilon}} \phi \,\dr (\FF \cdot \nu)_{\partial U^{\varepsilon}}
\qquad\,\,\text{for any } \phi \in \Sob^{1,\infty}(\Omega).
 \end{equation}
Moreover, for every $\varepsilon_k \to 0$ with $\varepsilon_k \notin \mathcal{N}$,
the normal trace of $\FF$ on $\partial U$ can be represented as the limit of trace measures{\rm :}
\begin{equation}
\label{main2}
\langle \FF \cdot \nu, \,\phi\rangle_{ \partial U}
= \lim_{k \to \infty} \int_{\partial U^{\eps_{k}}}  \phi \,\dr (\FF \cdot \nu)_{\partial U^{\eps_{k}}}
\qquad \text{for any } \phi \in \Sob^{1,\infty}(\Omega).
\end{equation}
\end{theorem}

\begin{proof} We divide the proof into three steps.

\smallskip
{\bf 1}. For each $0<t<s$, we define $\psi_{t,s}^U \in \Lip_{\mathrm{c}}(\Omega)$ by
  \begin{equation}\label{eq:psi_st_definition}
    \psi_{t,s}^U(x) = \threepartdef{s-t}{\text{ if } x \in U^{s},}{d(x)-t}{\text{ if } x \in U^t \setminus U^{s},}{0}{\text{ if } x \not\in U^t.}
  \end{equation}
We first show that
\begin{equation}\label{eq:psi_st_equality}
  \int_{U} \psi_{t,s}^U  \, \d (\div(\phi \FF))
  = -\int_{U^t \setminus \overline{U^{s}}} \phi \, \d \overline{\nabla d \cdot \FF}
\end{equation}
for any $\phi \in \Sob^{1,\infty}(\Omega)$ and all $0 < t < s$ for which
  \begin{equation}\label{eq:nablad_endpoint_zero}
      \lvert \overline{\nabla d \cdot \FF}\rvert(\partial U^t) = \lvert \overline{\nabla d \cdot \FF}\rvert(\partial U^s) = 0
  \end{equation}
 that holds for all but countably many $s$ and $t$.

 Since $\psi^U_{t,s}$ is supported away from $\partial U$, by Lemma \ref{soporte}, we have
  \begin{equation}
      0 = \langle \FF \cdot \nu, \,\psi_{t,s}^U \phi \rangle_{\partial U}
      = \div(\phi \psi_{t,s}^U \FF)(U) = \langle \phi \FF \cdot \nu, \,\psi_{t,s}^U \rangle_{\partial U}.
  \end{equation}
Now, using the product rule (Theorem \ref{productrule}), we have
  \begin{eqnarray}
  \label{codo}
   0= \langle \phi  \FF \cdot \nu, \,\psi_{t,s}^U \rangle_{\partial U}
   = \int_{U} \psi_{t,s}^U  \,\dr (\div(\phi \FF)) + \int_{U} \dr \overline{ \nabla \psi_{t,s}^U  \cdot (\phi\FF) }.
  \end{eqnarray}
  By definition of the pairing measure,
  \begin{equation}
  \begin{split}
      \overline{\nabla\psi_{t,s}^U \cdot (\phi \FF)}
      &= \operatornamewithlimits{w*-lim}_{\delta \to 0}
        \big(\nabla(\psi_{t,s}^U \ast \rho_{\delta}) \phi \FF \big)\\
      &= \operatornamewithlimits{w*-lim}_{\delta \to 0}
       \big((\mathbbm{1}_{U^t \setminus \overline{U^s}}\nabla d) \ast \rho_{\delta} \phi \FF \big) \\
      &= \phi\overline{\nabla d \cdot \FF} \res \left( U^t \setminus \overline{U^s}\right),
      \end{split}
  \end{equation}
  where we have used \eqref{eq:nablad_endpoint_zero} to justify the weak${}^\ast$--limit in the last equality.
  Combining this  with \eqref{codo} yields
  \eqref{eq:psi_st_equality}, which can be written as
  \begin{equation}\label{eq:averaged_st}
      (s-t)\int_{U^{s}}\d (\div (\phi\FF))
      + \int_{U^t \setminus \overline{U}^{s}}(d(x)-t) \, \d (\div(\phi \FF))
      = -\int_{U^t \setminus \overline{U^{s}}} \phi \, \d \overline{\nabla d \cdot \FF},
  \end{equation}
  by definition of $\psi_{t,s}^U$.

\smallskip
{\bf 2}. We now apply the disintegration result (Theorem \ref{prop:disintegration}) with $\mu = \overline{\nabla d \cdot \FF}$ to write
  \begin{equation}
      \int_{U^t \setminus \overline{U^s}} \phi \,\d\overline{\nabla d \cdot \FF}
      = \int_t^s \int_{\partial U^r} \phi \,\d\mu_r\,\d\tau(r).
  \end{equation}
  Now, given $\eps, h>0$ with $h < \eps$, we apply \eqref{eq:averaged_st}
  with $s = \eps+h$ and $t = \eps-h$ to obtain
  \begin{equation}\label{eq:disint_prelimit}
  \begin{split}
    &- \int_{\eps-h}^{\eps+h} \int_{\partial U^t} \phi \,\d\mu_t\,\d\tau(t) \\
    &= 2h \int_{U^{\eps+h}} \d(\div(\phi\FF)) + \int_{U^{\eps-h} \setminus \overline{U^{\eps+h}}} (d(x) -\eps+h)\,\d (\div(\phi\FF)) .
      \end{split}
  \end{equation}
 This is valid for all but countably many $h>0$ depending on $\eps$.
 We also impose that
 \begin{equation}\label{eq:eps_no_concentration}\lvert \FF \rvert(\partial U^{\eps}) = \lvert \div \FF \rvert(\partial U^{\eps}) = 0,\end{equation}
 which holds for all but countably many $\eps$.

We now divide both sides by $2h$ and study the limit as $h \to 0$.
Let $h_k \searrow 0$ be any sequence such that \eqref{eq:disint_prelimit} is valid with each $h_k$ in place of $h$.
  Since $\lvert d(x) -\eps + h_k\rvert \leq 2h_k$ on $U^{\eps-h_k}\setminus \overline{U}^{\eps +h_k}$ and $\lvert\div(\phi \FF)\rvert (U^{\eps-h_k}\setminus \overline{U}^{\eps +h_k}) \to 0$ as $k \to \infty$ by \eqref{eq:eps_no_concentration}, we have
  \begin{equation}
      \lim_{k \to \infty} \frac1{2h_k}\int_{U^{\eps-h_k} \setminus \overline{U^{\eps+h_k}}} (d(x) -\eps+h_k)\,\d (\div(\phi\FF)) = 0.
  \end{equation}
  The Dominated Convergence Theorem gives
  \begin{equation}
      \lim_{k \to \infty} \int_{U^{\eps+h_k}} \d(\div(\phi\FF)) = \int_{U^\eps} \d(\div(\phi\FF))
      = -\langle \FF \cdot \nu, \,\phi \rangle_{\partial U^{\eps}}.
  \end{equation}
  Hence, we infer that
  \begin{equation}\label{eq:normaltrace_diff_prelimit}
      \langle \FF \cdot \nu, \,\phi \rangle_{\partial U^{\eps}}
      = \lim_{k \to \infty} \frac1{2h_k} \int_{\eps-h_k}^{\eps+h_k} \int_{\partial U^t} \phi \,\d\mu_t\,\d\tau(t).
  \end{equation}

\smallskip
{\bf 3}. To proceed, we decompose
measure $\tau$ in two parts:
the absolutely continuous part with respect to $\mathcal{L}^{1}$ and the singular part as
 \begin{equation*}
\tau = (D_{\mathcal{L}^{1}} \tau) \mathcal{L}^{1} + \tau_{\text{sing}},
\qquad D_{\mathcal{L}^{1}} \tau \in \Leb^1_{\text{loc}}((0,\infty)).
\end{equation*}
Hence, with $T:=D_{\mathcal{L}^{1}} \tau$, we have
\begin{equation} \label{calculoprincipal}
\int_{\eps - h_k}^{\eps + h_k }\int_{\partial U^t}\phi\,\d\mu_t \,\d \tau(t)
=\int_{\eps - h_k}^{\eps + h_k }\int_{\partial U^t}\phi\,\d\mu_t \,T(t) \,\d t
 +\int_{\eps - h_k}^{\eps + h_k }\int_{\partial U^t}\phi\,\d\mu_t \,\d \tau_{\text{sing}}(t).
\end{equation}
We can estimate the first integral on the right-hand side of \eqref{calculoprincipal}
by using
the Lebesgue Differentiation Theorem.
Indeed, for $\mathcal{L}^1$\textit{--a.e.}\,\,$\eps$, we obtain
 \begin{equation}\label{eq:required_eps1}
\lim_{k \to \infty} \frac{1}{2h_k} \int_{\eps - h_k}^{\eps + h_k }
\Big(\int_{\partial U^t}\phi \,\d \mu_t \Big) T(t)\,\d t
=\int_{\partial U^\eps}\phi \,T(\eps)\,\d \mu_{\eps},
\end{equation}
so that
\begin{equation}\label{eq:trace_equals_disint}
(\FF \cdot \nu)_{\partial U^{\eps}} = T(\eps) \mu_{\eps},
\end{equation}
from \eqref{eq:normaltrace_diff_prelimit}.
Since $\abs{\mu_t}(\partial U^{t})=1$ for $\tau_{\text{sing}}$\textit{--a.e.}\,\,$t$,  we have
\begin{equation*}
\Big\lvert \int_{\partial U^{t}} \phi \,\d \mu_{t} \Big\rvert
\leq \norm{\phi}_{\Leb^{\infty}(\Omega)}\abs{\mu_{t}} (\partial U^{t})=\norm{\phi}_{\Leb^{\infty}(\Omega)}.
\end{equation*}
Using this and Lemma \ref{elchiste}, we can estimate the second integral on the right-hand side of \eqref{calculoprincipal} as
\begin{align}
\lim_{k \to \infty} \frac{1}{2h_k}
\Big\lvert\int_{\eps - h_k}^{\eps + h_k }\int_{\partial U^t}\phi \,\d\mu_t \d \tau_{\textnormal{sing}}(t)\Big\rvert
&\leq  \norm{\phi}_{\Leb^{\infty}(\Omega)} \lim_{h_k \to 0}
\frac {\tau_{\text{sing}} (\eps-h_k, \eps + h_k)} {\mathcal{L}^{1} (\eps -h_k, \eps + h_k)} \nonumber \\
 &=  \norm{\phi}_{\Leb^{\infty}(\Omega)} D_{\mathcal{L}^1} \tau_{\textnormal{sing}}(\eps)=0
 \label{eq:required_eps2}
 \end{align}
 for $\mathcal{L}^{1}$--\textit{a.e.}\,\,$\eps>0$.
 Thus, defining $\mathcal N \subset (0,\infty)$ to be the set of points where \eqref{eq:eps_no_concentration},
 \eqref{eq:required_eps1}, or \eqref{eq:required_eps2} does not hold, which is $\mathcal L^1$--null, and
 from \eqref{eq:trace_equals_disint},
 \begin{equation}
 \label{labuena}
\langle \FF \cdot \nu, \,\phi \rangle_{\partial U^{\eps}}
= \int_{\partial U^{\eps}} \phi \, \d (\FF \cdot \nu)_{\partial U^{\eps}} \qquad\text{for any }\eps \notin \mathcal{N}.
 \end{equation}
 Finally, taking any $\eps_k \to 0$ with $\eps_k \notin \mathcal{N}$ for all $k$, we have
 \begin{equation}
     \langle \FF \cdot \nu, \,\phi\rangle_{\partial U}
     = \lim_{k \to \infty} \langle \FF \cdot \nu, \phi\rangle_{\partial U^{\eps_k}}
     = \lim_{k \to \infty} \int_{\partial U^{\eps_{k}}}  \phi \,\dr (\FF \cdot \nu)_{\partial U^{\eps_{k}}}
 \end{equation}
 as required.
\end{proof}

\begin{remark}
Similar results were previously obtained by Frid \cite{Frid2} for domains with Lipschitz deformable boundaries.
\end{remark}

Later in \S \ref{sec:flux_property_general}, we will use a similar decomposition to $\overline{\nabla d \cdot \FF}$
from the above proof applied to a non-negative measure, which we record here.

\begin{lemma}\label{lem:disint_tau_decomp}
    Let $\mu$ be a non-negative Radon measure on an open set $U \subset \mathbb R^n$.
    Then there exists a decomposition
    \begin{equation}
      \mu = \mathcal L^1 \res [0,\infty) \otimes_{\partial U^t} \mu_t + \tau_{\mathrm{sing}} \otimes_{\partial U^t} \tilde\mu_t,
    \end{equation}
where $D_{\mathcal L^1}\tau_{\mathrm{sing}} = 0$ $\,\mathcal L^1$--\textit{a.e.}\,\,on $[0,\infty)$,
and we have used notation \eqref{eq:disint_notation} from {\rm Theorem \ref{prop:disintegration}}.
\end{lemma}

\begin{proof}
We apply Theorem \ref{prop:disintegration} to $\mu$ to decompose
  \begin{equation}
      \mu = \tau \otimes_{\partial U^t} \tilde\mu_t,
  \end{equation}
and then use the Lebesgue-Besicovitch Theorem to decompose
  \begin{equation}
      \tau = T \mathcal L^1 \res [0,\infty) + \tau_{\mathrm{sing}},
  \end{equation}
where $\mathcal L^1$ and $\tau_{\mathrm{sing}}$ are mutually singular,
and $T = D_{\mathcal L^1}\tau \in \Leb^1((0,\infty))$.
Taking the precise representative of $T$, we set $\mu_t = T(t)\tilde{\mu}_t$,
which is defined $\mathcal L^1$--\textit{a.e.}\,\,on $(0,\infty)$.
This gives the claimed decomposition by noting that the last part follows from Lemma \ref{elchiste}.
\end{proof}

\section{Properties of the Disintegration}\label{sec:prop_disintegration}

In \S \ref{sec:disintegration}, we have established Theorem \ref{principal} by considering a suitable disintegration of
$\overline{\nabla d \cdot \FF}$ along $\{\partial U^{\eps}\}_{\eps>0}$
and showing that the obtained measures $(\FF \cdot \nu)_{\partial U^{\eps}}$ coincided with the normal trace.
We now investigate this decomposition more closely
and show that
measure $\overline{\nabla d \cdot \FF}$ can actually be recovered from these traces.

Informally, if we take $d(x) = x_i$, which corresponds to taking $U$ to be
a half-plane $\{ x \in \mathbb R^n : x_i \geq 0\}$ for some $1 \leq i \leq n$,
then it follows from $\nabla d = e_i$ that $\overline{\nabla d \cdot \FF} = F_i$
for each $1 \leq i \leq n$,
where $\FF = (F_1,F_2,\cdots,F_n)$.
Thus, such a result allows us to recover the underlying field $\FF$ from the associated normal trace over
the sets $\{\partial U^{\eps}\}_{\eps>0}$; while Theorem \ref{thm:reconstruct_normaltrace} cannot
directly be applied to the half-space which is unbounded, we adapt this later in Lemma \ref{lem:hyperplane_reconstruct}.
This observation serves as a fundamental starting point for the theory of Cauchy fluxes in the extended setting,
which will be developed from \S \ref{sec:flux_definition} onwards.

\begin{theorem}\label{thm:reconstruct_normaltrace}
Let $\FF \in \mathcal{DM}^{\textnormal{ext}}(\Omega),$ $U \Subset \Omega$, and let $d(x) = \dist(x,\partial U)$.
Then
  \begin{equation}\label{eq:disint_equality}
    \int_0^{\infty} \langle \FF \cdot \nu, \,\phi\rangle_{\partial U^{t}} \,\d t = \int_{U } \phi \,\d\overline{\nabla d \cdot \FF}
    \qquad\,\, \mbox{for any $\phi \in \Sob^{1,\infty}(\Omega)$},
  \end{equation}
where $\overline{\nabla d \cdot \FF}$ is given as in {\rm Theorem \ref{productrule}}.
Moreover, \eqref{eq:disint_equality} extends to any bounded Borel function $\phi$ on $\Omega$,
understanding that $t \mapsto \langle \FF \cdot \nu, \,\phi\rangle_{\partial U^t}$ is only defined for $\mathcal L^1$--\textit{a.e.}\,\,$t>0$.
\end{theorem}

Equivalently, consider the disintegration given by Theorem \ref{prop:disintegration}, which allows us to write
\begin{equation}\label{eq:disintegration_2}
  \int_U \phi \,\d \overline{\nabla d \cdot \FF}
  = \int_0^{\infty} \int_{d^{-1}(t)} \phi(x) \,\d\mu_t(x)\,\d \tau(t)
\end{equation}
for any bounded Borel function $\phi$ on $U.$
Now, in Theorem \ref{principal}, we have seen that
\begin{equation}\label{eq:disintegration_slices}
  \langle \FF \cdot \nu, \cdot \rangle_{\partial U^t}
  =( \FF \cdot \nu)_{\partial U^\eps}= D_{\mathcal L^1} \tau(t)\, \mu_t \qquad\text{for } \mathcal L^1\textit{--a.e.\,\,} t>0,
\end{equation}
so Theorem \ref{thm:reconstruct_normaltrace} is equivalent to the assertion that $\tau_{\mathrm{sing}} = 0.$
To prove this, we use the following elementary lemma.
\begin{lemma}\label{lem:1d-bv}
  Let $\tau$ and $\sigma$ be finite Radon measures on $(0,\infty)$ such that
  \begin{equation}
    \tau' = \sigma\qquad  \text{ in $\mathcal D'((0,\infty))$},
  \end{equation}
  that is,
  \begin{equation}\label{eq:weakform_taudiff}
    -\int_0^{\infty} \psi'(t) \,\d\tau(t) = \int_0^{\infty} \psi(t) \,\d\sigma(t)
    \qquad\,\,\text{for any } \psi \in C^1_{\mathrm{c}}((0,\infty)).
  \end{equation}
  Then $\tau \ll \mathcal L^1$ and there exists $T \in \BV((0,\infty))$ such that $\tau = T \mathcal L^1 \res (0,\infty).$
\end{lemma}

\begin{proof}
Define $v(x) = \sigma((0,x))$ for $x>0$, which lies in $BV((0,\infty))$.
Then $v$ satisfies $v' = \sigma$ and is bounded pointwise by $\lvert\sigma\rvert((0,\infty))$.
Define the measure: $\tilde \tau = \tau - v \mathcal L^1$.
Since $\tilde\tau' = 0$ in $\mathcal D'((0,\infty))$,
we claim $\tilde\tau$ is a constant multiple of the Lebesgue measure.
Indeed, mollifying $\tilde\tau$, we obtain a smooth function $\tilde\tau_{\delta}$
that satisfies $\tilde\tau_{\delta}' \equiv 0$ in $(0,\infty)$.
Thus, for each $\delta>0$,
there exists a constant $c_{\delta} \in\mathbb R$ such that $\tilde\tau_{\delta}(x) = c_{\delta}$ for all $x>\delta$.
Since $\tilde\tau_{\delta} \weakstarto \tilde\tau$ as $\delta\to0$ locally in $(0,\infty)$,
we know that $c_{\delta} \to c$ for some $c \in \mathbb R$.
Hence, $\tau = (v + c) \mathcal L^1$, which implies that $\tau$ is represented by a $BV$ function.
\end{proof}

\begin{proof}[Proof of Theorem \ref{thm:reconstruct_normaltrace}]  We divide the proof into four steps.

\smallskip
{\bf 1}. Let $\psi \in C^1_{\mathrm{c}}((0,\infty))$ and $g \in C^1_{\mathrm{c}}(\Omega)$,
and set $\phi(x) = \psi(d(x)) g(x)$.
Then $\phi \in \Lip_{\mathrm{c}}(U)$,
and the mollification $\phi_{\delta}$ is defined and vanishes on $\partial U$ for $\delta>0$ sufficiently small.
By Lemma \ref{muyutil}, we have
  \begin{equation}
    \int_U \nabla\phi \cdot \FF_{\delta} \,\d x = \int_U \nabla \phi_{\delta} \cdot \d \FF = - \int_U \phi_{\delta}\, \d(\div \FF) = - \int_U \phi \div \FF_{\delta} \,\d x.
  \end{equation}
Expanding $\nabla \phi(x) = \psi'(d(x)) g(x)\nabla d(x) + \psi (d(x))\nabla g(x)$ gives
  \begin{equation}
    -\int_U \psi'(d)\, g\, \nabla d \cdot\FF_{\delta} \,\d x
    = \int_U \big(\phi\, (\div \FF)_{\delta} + \psi(d) \nabla g \cdot \FF_{\delta}\big) \,\d x.
  \end{equation}
Sending $\delta \to 0$ and using Theorem \ref{productrule}, we obtain
  \begin{equation}\label{eq:tau_weak_1}
    -\int_U \psi'(d) \, g \,\d\overline{\nabla d \cdot \FF}
    =  \int_U \phi\, \d (\div \FF) + \int_U \psi(d)\, \nabla g \cdot \d \FF.
  \end{equation}
{\bf 2}. Now set $\lambda = g \, \div \FF + \nabla g \cdot \FF$ and use Proposition \ref{prop:disintegration} to write
 \begin{align}
 \overline{\nabla d \cdot \FF}= \tau \otimes_{\partial U^t} \mu_t, \qquad
 \lambda = \sigma \otimes_{\partial U^t} \lambda_t.
 \end{align}
Using this, we can rewrite \eqref{eq:tau_weak_1} as
  \begin{equation}
    -\int_0^{\infty} \psi'(t)\int_{\partial U^t} g(x) \,\d\mu_t(x) \,\d\tau(t)
    = \int_0^{\infty} \psi(t) \int_{\partial U^t} \,\d\lambda_t(x) \,\d\sigma(t).
  \end{equation}
Hence, as distributions on $(0,1),$ we have shown that
  \begin{equation}
    \frac{\d}{\d t} \Big(\int_{\partial U^t} g(x) \,\d\mu_t(x) \,\tau\Big)
    = \Big( \int_{\partial U^t} \,\d\lambda_t(x) \Big) \,\sigma.
  \end{equation}
By Lemma \ref{lem:1d-bv}, we infer that $\left(\int_{\partial U^t} g(x) \,\d\mu_t(x)\right)\tau$ can be represented by a $\BV$ function.
In particular, decomposing $\tau = T \mathcal L^1 \res [0,\infty) + \tau_{\mathrm{sing}}$,
we infer by uniqueness of the Lebesgue-Besicovitch Decomposition Theorem that
  \begin{equation}
    \Big(\int_{\partial U^t} g(x) \,\d\mu_t(x) \Big)\tau
    = \Big(\int_{\partial U^t} g(x) \,\d\mu_t(x)\Big) T(t) \, \mathcal L^1 \res [0,\infty)
  \end{equation}
as measures on $[0,\infty)$ so that, for any $g \in C^1_{\rm c}(\Omega)$,
\begin{equation}\label{eq:almost_tausing}
 \int_{\partial U^t} g(x) \,\d\mu_t(x) =0 \qquad\tau_{\mathrm{sing}}\textit{--a.e.}
\end{equation}

{\bf 3}. Set $\mu = \overline{\nabla d \cdot \FF}$ and put $\tilde g = D_{\lvert \mu\rvert}\mu$.
In order to apply \eqref{eq:almost_tausing} with $\tilde g$,
we approximate $\tilde g$ by mappings in $C^1_{\mathrm{c}}(U)$.
By the Lusin Theorem \cite[{\rm Theorem}~1.45, {\rm Remark}~1.46]{afp},
we can find a sequence $\tilde{g}_k$  of continuous functions in $U$ such that
\begin{equation}
\tilde g_k(x) \rightarrow \tilde g(x) \qquad\,\,\text{for $\lvert\mu\rvert$--\textit{a.e.}\,\,$x \in U\,\,$
as $k \to \infty$},
\end{equation}
and $\lvert \tilde g_k\rvert \leq 1$ holds $\lvert \mu\rvert$--\textit{a.e.}\,\,for each $k$.
Then, for any $\delta_k \to 0,$ $g_k =(\chi_{U^{2\delta_k}}\tilde g_k)_{\delta_k}$ is a sequence of $C^{\infty}_{\mathrm{c}}(U)$ functions converging pointwise $\lvert \mu\rvert$--\textit{a.e.}\,\,to $\tilde g$ in $U$, which is also uniformly bounded by the constant $1$.
Now applying \eqref{eq:almost_tausing} with each $g_k$ yields
  \begin{equation}
    \begin{split}
    \int_0^{\infty}\Big\lvert \int_{\partial U^t} \tilde g \,\d\mu_t(x)\Big\rvert \,\d \tau_{\mathrm{sing}}(t)
    &=\int_0^{\infty}\Big\lvert \int_{\partial U^t} g_k - \tilde g \,\d\mu_t(x)\Big\rvert
    \,\d \tau_{\mathrm{sing}}(t)\\
    &\leq \int_0^{\infty} \int_{\partial U^t} \lvert g_k - \tilde g\rvert \,\d \mu_t(x) \,\d \tau \\
    &= \int_U \lvert g_k - \tilde g\rvert \,\d \mu,
    \end{split}
  \end{equation}
which tends to zero as $k \to \infty$ by the Dominated Convergence Theorem.
Hence, it follows that
\begin{equation}
\int_{\partial U^t} D_{\lvert \mu\rvert}\mu(x) \,\d\mu_t(x) = 0 \qquad\text{$\tau_{\mathrm{sing}}$\textit{--a.e.}}
\end{equation}
By Lemma \ref{lem:disint_polar}, we have
  \begin{equation}
    D_{\lvert \mu\rvert}\mu(x) = D_{\lvert \mu_t\rvert}\mu_t(x) \qquad\text{for } \tau\textit{--a.e.\,\,}t \text{ and } \mu_t\textit{--a.e.\,\,} x,
  \end{equation}
so it follows that
  \begin{equation}
    \lvert \mu_t\rvert(\partial U^t) = \int_{\partial U^t} D_{\lvert \mu\rvert}\mu(x) \,\d\mu_t(x) = 0
    \qquad\text{ $\tau_{\mathrm{sing}}$\textit{--a.e.}}
  \end{equation}
However, by Theorem \ref{prop:disintegration},
$\lvert \mu_t\rvert$ is a probability measure supported
on $\partial U^t$ for $\tau$--\textit{a.e.}\,\,$t \in (0,\infty),$ so the above can only hold if $\tau_{\mathrm{sing}} = 0$.

\smallskip
{\bf 4}. By \eqref{eq:disintegration_2}--\eqref{eq:disintegration_slices}, we have
  \begin{equation}
      \int_{U} \phi \,\d\overline{\nabla d \cdot \FF}
      = \int_0^{\infty} \int_{\partial U^t} \phi \,\d(\FF\cdot\nu)_{\eps} \,\d\eps
      + \int_0^{\infty} \int_{\partial U_t} \phi \,\d\mu_t\,\d\tau_{\mathrm{sing}}
  \end{equation}
  for any bounded Borel function $\phi$ in $\Omega$.
  Since we have shown that $\tau_{\mathrm{sing}} = 0$, the last integral is zero, thus establishing the result.
\end{proof}

\begin{remark}
    A similar coarea-type formula was recently obtained in \cite[{\rm Theorem}~6.1]{comiextended} in a more general context.
    Moreover, while it is not explicitly stated, a careful inspection of their argument
    reveals an alternative proof of Theorem \ref{thm:reconstruct_normaltrace}.
\end{remark}

As a consequence, we obtain an alternative proof of the following result from \cite{Silhavy2}, which will be used later.
Similar results for Lipschitz-deformable boundaries were proved earlier in \cite{CF2}.

\begin{theorem}[Theorem 2.4 in {\cite{Silhavy2}}] \label{thm:averaged_trace}
 Let $\FF \in \mathcal{DM}^{\ext}(\Omega)$ and $U \Subset \Omega$, and let $d(x) = \dist(x,\partial U)$.
 Then
\begin{equation}\label{eq:normaltrace_averaged_limit}
 \langle \FF \cdot \nu, \,\phi \rangle_{\partial U}
 = \lim_{\eps \to 0} \frac1{\eps} \int_{U \setminus \overline{U^{\eps}}} \phi \,\d\overline{\nabla d \cdot \FF}
 \qquad\,\,\text{for any } \phi \in \Sob^{1,\infty}(\Omega).
\end{equation}
Moreover, if
   \begin{equation}
       \liminf_{\eps \to 0} \frac1{\eps}\, \lvert \overline{\nabla d \cdot \FF}\rvert(U \setminus \overline{U^{\eps}}) < \infty,
   \end{equation}
   then the normal trace $\langle \FF \cdot \nu, \,\cdot\,\rangle_{\partial U}$ is represented by a measure on $\partial U$.
\end{theorem}

\begin{proof}
By Theorem \ref{thm:reconstruct_normaltrace}, we see that
\begin{equation}
\int_0^{\infty} \int_{\partial U^t} \psi \,\d(\FF\cdot\nu)_{\partial U^t} \,\d t
= \int_{U } \psi \,\d\overline{\nabla d \cdot \FF}
\end{equation}
holds for all bounded Borel functions $\psi$ on $\Omega$.
Then, taking $\psi = \mathbbm{1}_{U \setminus \overline{U^{\eps}}} \,\phi$ with $\eps>0$ and $\phi \in \Sob^{1,\infty}(\Omega)$,
we have
\begin{equation}\label{eq:diff_disint_prelimit}
\frac1{\eps} \int_0^{\eps} \phi \,\d\overline{\nabla d \cdot \FF}
= \frac1{\eps} \int_0^{\eps} \langle \FF \cdot \nu, \,\phi \rangle_{\partial U^t} \,\d t.
\end{equation}
Since the mapping:
    \begin{equation}
        t \mapsto \langle \FF \cdot \nu, \,\phi \rangle_{\partial U^t}
        = \int_{U^t} \d\overline{\nabla \phi \cdot \FF} + \int_{U^t} \phi \,\ \d(\div \FF)
    \end{equation}
    is right-continuous on $[0,\infty)$ for $\phi \in \Sob^{1,\infty}(\Omega)$, sending $\eps \to 0$ in \eqref{eq:diff_disint_prelimit},
    we deduce \eqref{eq:normaltrace_averaged_limit}.
Now, for any $\eps_k \to 0$, the above implies that
    \begin{equation}
        \frac1{\eps_k} \overline{\nabla d \cdot \FF} \res \left(U \setminus \overline{U^\eps_k}\right)
        \xrightharpoonup{\,\mathcal D'\,} \langle \FF \cdot \nu, \,\cdot\,\rangle_{\partial U} \qquad\text{as } k \to \infty
    \end{equation}
    as distributions in $\Omega$.
    If this sequence of measures is uniformly bounded in $\mathcal M(\Omega)$,
    by the weak${}^{\ast}$--compactness and uniqueness of the limit,
    we infer that the limiting distribution is also a measure.
\end{proof}

\section{Localization of the Normal Trace}\label{sec:localisation}

In this section, we analyze
how the normal trace relative to different boundaries can differ.
In particular, we seek to understand whether the relation
\begin{equation}\label{eq:question_localisation}
    \langle \FF \cdot \nu, \,\phi \rangle_{\partial U} = \langle \FF \cdot \nu, \,\phi \rangle_{\partial V}
\end{equation}
holds for $U, V \subset \Omega$ with overlapping boundary and for any $\phi$ supported
in a suitable neighborhood of $\partial U \cap \partial V$.
Moreover, if the normal traces of $\FF$ on the boundaries of $U$ and $V$ are represented by measures,
one may ask if the equality holds as measures:
\begin{equation}\label{eq:question_localisation2}
    (\FF \cdot \nu)_{\partial U} \res (\partial U \cap \partial V)
    = (\FF \cdot \nu)_{\partial V} \res (\partial U \cap \partial V).
\end{equation}

This question is motivated by the study of Cauchy fluxes in the sequel.
Indeed, it is easily answered whenever the normal trace admits a representation of the form
\begin{equation}
    \langle \FF \cdot \nu, \,\phi \rangle_{\partial U}
    = \int_{\partial U} \phi(x)\, \FF(x) \cdot \nu_{\partial U} \,\d \mathcal{H}^{n-1},
\end{equation}
which is valid by the classical Gauss-Green formula for a Lipschitz $\FF$ and  a bounded Lipschitz domain $U$;
we see that it is necessary for the associated normals $\nu_{\partial U}$ and $\nu_{\partial V}$
to coincide $\mathcal{H}^{n-1}$\textit{--a.e.}\,\,on the common intersection.
For bounded divergence-measure fields and sets of finite perimeter,
this can be affirmatively answered by using \cite[{\rm Theorem} 5.3]{ctz} after reformulation
to consider suitable notions of measure-theoretic boundaries and normals.
However, for a general field $\FF$ and open sets $U$ and $V$,
the normal trace may \emph{fail} to be equal as measures on the common intersection,
which is illustrated by the following examples.

\begin{example}\label{eq:localisation_counterexample}
Following {\rm \cite[Example~1.1]{CF2}} which is itself
based on {\rm \cite[\S III.14, Example 1]{whitney2012geometric}},
consider the vector field
  \begin{equation}
    \FF(x) = \frac{(x_1,x_2)}{x_1^2+x_2^2} \qquad\,\,\text{for }x=(x_1,x_2) \in \mathbb R^2,
  \end{equation}
which lies in $\mathcal{DM}^1_{\loc}(\mathbb R^2)$ with $\div \FF = 2\pi \delta_0$.
Indeed, $\FF(x) = 2 \pi\, \nabla \Gamma(x)$, where $\Gamma$ is the fundamental solution
for the Laplacian in $\mathbb R^2$ (see \textit{e.g.},\,\cite[{\S 2.2.1}]{eg2}).
Let $H = (0,\infty) \times \mathbb R$ be the half-space where $x_1>0$, and let $Q = (0,1) \times (0,1)$ be the unit cube.
Then we claim the normal traces of $\FF$ on $\partial H$ and $\partial Q$ can be represented by measures that take the form:
 \begin{align}
     (\FF \cdot \nu)_{\partial H} &= - \pi \delta_0, \label{eq:whitney_halfspace}\\
     (\FF \cdot \nu)_{\partial Q} &= -\frac{\pi}2 \delta_0 + (\FF \cdot \nu)  \mathcal H^1 \res \partial Q. \label{eq:whitney_cube}
 \end{align}
Indeed, for \eqref{eq:whitney_halfspace}, using an unbounded version of {\rm \cite[Theorem 7.1]{ChenComiTorres}},
for $\phi \in C^{\infty}_{\mathrm{c}}(\mathbb R^2)$, we can show that
\begin{equation}
      \langle \FF \cdot \nu, \,\phi \rangle_{\partial H}
      = -\lim_{\eps \to 0} \int_{\mathbb R} \frac{\eps}{\eps^2 + x_2^2}\, \phi(\eps,x_2) \,\d x_2 = -\pi \phi(0,0),
\end{equation}
which implies \eqref{eq:whitney_halfspace}.
To prove this limit, observe that
  \begin{equation}
      \int_{-\infty}^{\infty} \frac{\eps}{\eps^2+x_2^2} \phi(0,0) \,\d x_2 = \pi \phi(0,0)
      \qquad \mbox{for any $\eps>0$}
  \end{equation}
and, if $\phi$ is supported in $B_R(0)$,
the remainder can be estimated by
  \begin{equation}
      \begin{split}
      &\Big\lvert \int_{\mathbb R} \frac{\eps}{\eps^2+x_2^2} ( \phi(\eps,x_2)-\phi(0,0)) \,\d x_2\Big\rvert \\
      &\leq \int_{-R}^R \frac{\eps}{\eps^2+x_2^2} \lvert \phi(\eps,x_2) - \phi(0,0) \rvert\, \d x_2
      + 2 \lvert\phi(0,0)\rvert\int_R^{\infty} \frac{\eps}{\eps^2 + x_2^2} \,\d x_2 \\
      &\leq \lVert \nabla\varphi\rVert_{\Leb^{\infty}(\mathbb R^2)} \int_{-R}^R \frac{\eps}{\sqrt{\eps^2 + x_2^2}} \,\d x_2 + \lvert\phi(0,0)\rvert \,\frac{2\eps}R,
      \end{split}
  \end{equation}
  which vanishes as $\eps \searrow 0$.
  Thus, \eqref{eq:whitney_halfspace} follows.

 For the flux on $\partial Q$, we apply {\rm Theorem \ref{thm:averaged_trace}} near the origin to check
 that $(\FF \cdot \nu)_{\partial Q}$ is a measure on $\partial Q$;
  indeed, we now show that
  \begin{equation}\label{eq:whitney_example_measure}
      \limsup_{\eps \to 0} \frac1{\eps} \int_{B_{1/2}(0) \cap Q \setminus \overline{Q}^{\eps}} \lvert \nabla d \cdot \FF \rvert \,\d x < \infty.
  \end{equation}
   For this, observe that, for $x \in Q$ satisfying $\lvert x \rvert < \frac12$, $d(x) = \dist(x,Q)$ satisfies
  \begin{equation}
      \nabla d(x) = \begin{cases}
           e_1 &\text{if } x_1 < x_2, \\
           e_2 &\text{if } x_2 < x_1.
      \end{cases}
  \end{equation}
  Then, for each $\eps>0$, we split the integral in \eqref{eq:whitney_example_measure} into four pieces:
  \begin{align}
      A_{1,\eps} &= \big\{ (x_1,x_2) \in B_{1/2}(0)\, :\, 0< x_1 < x_2 < \eps \big\}, \\
      A_{2,\eps} &= \big\{ (x_1,x_2) \in B_{1/2}(0)\, :\, 0< x_1 < \eps < x_2 < \frac12 \big\}, \\
      A_{3,\eps} &= \big\{ (x_1,x_2) \in B_{1/2}(0)\, :\, 0< x_2 < x_1 < \eps \big\}, \\
      A_{4,\eps} &= \big\{ (x_1,x_2) \in B_{1/2}(0)\, :\, 0<  x_2 < \eps < x_1 < \frac12 \big\},
  \end{align}
  so that
  \begin{equation}
      B_{1/2}(0) \cap (Q \setminus \overline{Q}^{\eps})
      = A_{1,\eps} \cup A_{2,\eps} \cup A_{3,\eps} \cup A_{4,\eps}.
  \end{equation}
  Since $\nabla d = e_1$ on $A_{1,\eps}$ and $A_{2,\eps}$, we can compute
  \begin{equation}\label{eq:integral_A1eps}
    \begin{split}
        \frac1{\eps} \int_{A_{1,\eps}} \lvert \nabla d \cdot \FF \rvert \,\d x
        &\leq \frac1{\eps} \int_0^{\eps} \int_{0}^{x_2} \frac{x_1}{x_1^2+x_2^2} \,\d x_1 \,\d x_2
        = \frac12 \log 2,
    \end{split}
  \end{equation}
  and
  \begin{equation}\label{eq:integral_A2eps}
    \begin{split}
        \frac1{\eps} \int_{A_{2,\eps}} \lvert \nabla d \cdot \FF \rvert \,\d x
        &\leq \frac1{\eps} \int_{\eps}^{\frac12} \int_0^{\eps} \frac{x_1}{x_1^2+x_2^2} \,\d x_1 \,\d x_2 \\
        &= \frac1{4\eps} \log(1+4\eps^2) - \frac12 \log 2 + \arctan(\frac1{2\eps}) - \frac{\pi}4,
    \end{split}
  \end{equation}
  by a direct calculation. Both terms remain bounded as $\eps \to 0$,
  since $\arctan(x) \leq \frac{\pi}2$ and $\log(1+x) \leq x$ for all $x \geq 0$.
Moreover, the integrals over $A_{3,\eps}$ and $A_{4,\eps}$ coincide with the integrals in \eqref{eq:integral_A1eps}
and \eqref{eq:integral_A2eps}, respectively, by swapping the coordinates. Then the claim follows.

Therefore, $(\FF \cdot \nu)_{\partial Q}$ is a measure on $\partial Q$,
which agrees with $(\FF \cdot \nu) \mathcal H^1 \res (\partial Q \setminus \{0\})$
on $\partial Q \setminus \{0\}$,
by using Theorem \ref{thm:localisation} below and since $\FF$ is smooth there.
Then there is $c_0 \in \mathbb R$ such that $(\FF \cdot \nu)_{\partial Q} = c_0 \delta_0 + (\FF \cdot \nu) \mathcal H^1 \res \partial Q$.
This constant can be determined by noting that
  \begin{equation}
    (\FF \cdot \nu)_{\partial Q}(\partial Q) = c_0 + \frac{\pi}2 = \int_{ Q} \d(\div \FF) = 0.
  \end{equation}
 so that $c_0 = -\frac{\pi}2$, establishing \eqref{eq:whitney_cube}.
 Then we have
  \begin{equation}
      (\FF \cdot \nu)_{\partial Q} \res (\partial Q \cap \partial H)
      = - \frac{\pi}2 \delta_0 \neq - \pi \delta_0
      = (\FF \cdot \nu)_{\partial H} \res (\partial Q \cap \partial H).
  \end{equation}
Thus, the respective normal traces concentrate at the corner point $(0,0)$ and
take different values there.
This shows that \eqref{eq:question_localisation2} may fail in general.
\end{example}

The above example is for a field $\FF$ that is singular at the origin, and the discrepancy arose due to the concentration of the normal trace at a point.
One may also encounter obstructions arising from the regularity of the domains, which is illustrated in the following example.

\begin{example}\label{eq:localisation_counterexample2}
    In $\mathbb R^2$, we can consider the domains:
    \begin{equation}
        V = B_1(0) \setminus \{ (x_1,0)\,:\, x_1 \geq 0 \}, \qquad U = \{x \in V\,:\,x_2 > 0 \} = B_1(0)^+.
    \end{equation}
    Then we take $\FF(x)  = e_2 = (0,1)$, which is smooth and hence lies in $\DM^{\infty}_{\mathrm{loc}}(\mathbb R^2)$.
    Set $A = \{x \in B_1(0) : x_1 > 0 \}$.
    Then $A \cap \partial U = A \cap \partial V = \{ (x_1,0)\,:\,0 < x_1 < 1\}$.
    However, we have
    \begin{align}
        (\FF \cdot \nu)_{\partial U} \res (A \cap\partial U)
         &= \mathcal H^1 \res \{ (x_1,0)\,:\,0<x_1<1 \}, \\
        (\FF \cdot \nu)_{\partial V} \res (A \cap\partial V) &= 0,
    \end{align}
    so the associated normal traces do not coincide, even though $U \subset V$.
    Indeed, the normal trace on $\partial U \cap A$ is understood in the classical sense, since $U$ is a Lipschitz domain.
    For $\partial V$, using the definition, we can compute the normal trace for $\phi \in \Sob^{1,\infty}_{\mathrm{c}}(\mathbb R^2)$ as
    \begin{equation}
        \begin{split}
        \langle \FF \cdot \nu, \,\phi \rangle_{\partial V}
        &= -\int_V \phi \,\div \FF \,\d x - \int_V \nabla \phi \cdot \FF \,\d x \\
        &= -\int_{B_1(0)} \phi\, \div \FF \,\d x - \int_{B_1(0)} \nabla \phi \cdot \FF \,\d x \\
        &= \langle \FF \cdot \nu, \,\phi \rangle_{\partial B_1(0)},
        \end{split}
    \end{equation}
by noting that $\mathcal L^2(B_1(0) \setminus V) = 0$, so the normal trace indeed vanishes on $A$.
\end{example}

In general, one may still hope for \eqref{eq:question_localisation} to hold on {\it relatively open} portions
of the common boundary, which will be sufficient for our later purposes.
More precisely, we have

\begin{theorem}\label{thm:localisation}
    Let $\Omega \subset \mathbb R^n$ be open and $\FF \in \mathcal{DM}^{\ext}(\Omega)$.
    Given $U, V \Subset \Omega$, let $A \subset \mathbb R^n$ be open such that
    \begin{equation}\label{eq:UV_local_equality}
        U \cap A = V \cap A.
    \end{equation}
    Then $A \cap \partial U = A \cap \partial V$ and
    \begin{equation}
        \langle \FF \cdot \nu, \,\phi \rangle_{\partial U}
        = \langle \FF \cdot \nu, \,\phi \rangle_{\partial V} \qquad\text{for any } \phi \in \Sob^{1,\infty}_{\mathrm{c}}(A).
    \end{equation}
    In particular, if the normal traces on $\partial U$ and $\partial V$ are represented by measures, then    \begin{equation}\label{eq:localisation_meausres}
        (\FF \cdot \nu)_{\partial U} \res (\partial U \cap A)
         = (\FF \cdot \nu)_{\partial V} \res (\partial V \cap A).
    \end{equation}
\end{theorem}

\begin{proof}
  Replacing $A$ by $A \cap \Omega$ if necessary, we can assume that $A \subset \Omega$.
  We first show that $A \cap \partial U = A \cap \partial V$. Indeed, if $x \in A \cap \partial U$,
  since $U$ is open, then $x \notin U$.
  Moreover, since $A$ is open, there is a sequence $\{x_k\} \subset A \cap U$ such that $x_k \to x$.
  However, by \eqref{eq:UV_local_equality}, each $x_k \in V$ while $x \notin V$,
  so $x \in \partial V$ and the latter inclusion follows by symmetry.

  Now, for any $\delta>0$, when $\eps \in (0,\delta)$, we claim that
  \begin{equation}\label{eq:adelta_common}
    A^{\delta} \cap (U \setminus \overline{U^{\eps}})
    = A^{\delta} \cap (V \setminus \overline{V^{\eps}}),
  \end{equation}
  and that $d_U = d_V$ on this common intersection.
  Indeed, if $x \in A^{\delta} \cap (U \setminus \overline{U^{\eps}})$,
  there is $y \in \partial U$ such that $\lvert x - y \rvert < \eps$.
  Since $\eps<\delta$, we see that $y \in A$ and so $y \in \partial V$, giving
  \begin{equation}
      d_V(x) \leq \lvert x - y \rvert < \eps.
  \end{equation}
  Hence, $x \in A^{\delta} \cap (V \setminus \overline{V^{\eps}})$,
  and taking the infimum over all such $y$ yields that $d_V(x) \leq d_U(x)$.
  By symmetry of $U$ and $V$, the claim follows.

  In particular,
  \begin{equation}
      \overline{\nabla d_U \cdot \FF}\, \res (A^{\delta} \cap U \setminus \overline{U^{\eps}})
      \overline{\nabla d_V \cdot \FF}\, \res (A^{\delta} \cap V \setminus \overline{V^{\eps}}),
  \end{equation}
  by definition of the pairing from Theorem \ref{productrule}.
  Hence, by Theorem \ref{thm:averaged_trace},
  it follows that,  for any $\phi \in \Sob^{1,\infty}_{\mathrm{c}}(A^{\delta})$,
  \begin{equation}
    \begin{split}
      \langle \FF \cdot \nu, \,\phi \rangle_{\partial U}
      &= \lim_{\eps \to 0} \frac1{\eps} \int_{U \setminus \overline{U^{\eps}}} \phi \,\d\overline{\nabla d_U \cdot \FF} \\
      &= \lim_{\eps \to 0} \frac1{\eps} \int_{V \setminus \overline{V^{\eps}}} \phi \,\d\overline{\nabla d_V \cdot \FF}
      = \langle \FF \cdot \nu, \,\phi \rangle_{\partial V}.
    \end{split}
  \end{equation}
 Since $\delta>0$ is arbitrary, it follows that the above holds for any $\phi \in \Sob^{1,\infty}_{\mathrm{c}}(A)$.
  If the normal traces are represented by measures, \eqref{eq:localisation_meausres} follows by a standard density argument.
\end{proof}

\begin{remark}
    One may wonder if
    the condition: $U \cap A = V \cap A$ can be modified to hold at the level of boundaries,
    that is, to impose $\partial U \cap A = \partial V \cap A$ instead.
    However,
    Example \ref{eq:localisation_counterexample2} shows that this is insufficient, even if the additional assumption that $U \subset V$
    is made.
\end{remark}

We conclude this section by recording that
a maximal portion of the common intersection $\partial U \cap \partial V$ can be defined,
where Theorem \ref{thm:localisation} applies; that is, we can define
\begin{equation}
    S(U,V) = \{ x \in \partial U \cap \partial V\,:\,B_{\eps}(x) \cap U = B_{\eps}(x) \cap V \text{ for some } \eps>0 \}.
\end{equation}
Indeed, if $A$ is any open set such that $A \cap U = A \cap V$, then,
for every $x \in A \cap \partial U \cap \partial V$,
there is $\eps>0$ such that $B_{\eps}(x) \subset A$. It follows that $x \in S(U,V)$ so that $A \cap \partial U \cap \partial V \subset S(U,V)$.
Conversely, we can take $\tilde A = \bigcup_{x \in S(U,V)} B_{\eps_x}(x)$, where $\eps_x>0$ is as in the definition
of set $S(U,V)$.

\section{Cauchy Flux I: Main Results and Connections}\label{sec:flux_definition}

The balance law postulates that the production of a quantity in any open set $U \Subset \Omega$  is balanced by the Cauchy flux of this quantity through
boundary $\partial U$ of $U$.
We assume
the production is represented by
a finite (signed) Radon measure $\sigma$ in $\Omega$
such that
 \begin{equation}
\label{balancelaw3}
\sigma (U) = \mathcal{F}(\partial U) \qquad \text{for any } U \Subset \Omega,
 \end{equation}
where $\mathcal{F}$ is the flux through the boundary of $U$.
In this section, we introduce the conditions on the Cauchy flux $\mathcal{F}$
that guarantee the existence of an extended divergence-measure vector field $\FF$ satisfying
$$
-\div {\FF}= \sigma
$$
and such that $\mathcal{F}$ can be recovered locally, on the boundary of {\it almost every open set}, through the measure normal trace of $\FF$.
We begin by presenting the existing developments in this direction
and emphasizing the remaining difficulties, before stating our main results.

\subsection{Connections to other formulations of the Cauchy flux}\label{sec:flux_history}

The origin of the study of Cauchy fluxes dates back to the fundamental paper by Cauchy \cite{Cauchy1} in 1823
who considered the balance law in a bounded domain $\Omega$ in Classical Physics:
\begin{equation}\label{CauchyBalanceLaw}
\int_{U} p(x) \, \d x = \int_{\partial U} f(x, \nu(x)) \,\d \mathcal{H}^{n-1}(x) \qquad
\mbox{for any $U \subset \Omega$},
\end{equation}
where the production $p(x)$ is a bounded function in $x \in \Omega$,
and the {\it density function} $f(x,\nu)$ of the flux depends on point $x\in\partial U$
and the corresponding interior unit normal $\nu$.
It was shown in \cite{Cauchy1}
that, if $f$ is continuous in $x$, then $f$ must be linear in $\nu$.
For self-containedness and subsequent development, we now present a brief description of
Cauchy's argument.

\begin{theorem}[Cauchy's tetrahedron argument]\label{Cau}
Consider the classical balance law \eqref{CauchyBalanceLaw} in $\mathbb R^n$,
where $f$ is continuous in $x$ and $p \in \Leb^{\infty}(\Omega)$.
Then there exists a continuous vector field $\FF$ such that
\begin{equation*}
 f(x, \nu(x)) = \FF (x) \cdot \nu(x)  .
\end{equation*}
\end{theorem}

\begin{proof}
Consider the standard orthonormal basis $\{e_i\}_{i=1}^n$ of $\mathbb R^n$.
Let $\nu= (\nu_1,\nu_2,\cdots,\nu_n)$ be any unit normal vector satisfying $\nu \cdot e_i \neq 0$ for each $1 \leq i \leq n$.
Fix $\epsilon>0$ and consider the tetrahedron $T_{\eps}$ with vertices $v_0 = 0$,
$v_i = \eps \frac{\nu_n}{\nu_i} e_i$ for each $1 \leq i \leq n-1$, and $v_n = \eps e_n$.
Then the boundary of $T_{\eps}$ is composed of $(n+1)$ faces, denoted by $S_{i,\eps}$ for $1 \leq i \leq n$
and $S_{\eps}$.
These are chosen so that $S_{i,\eps}$ are contained in the planes: $x_i=0$, $1 \leq i \leq n$,
and $S_{\eps}$ is contained in the plane perpendicular to $\nu$
with the equation: $\sum_{i=1}^n \nu_i x_i= \eps \nu_n$.
For $\tilde{x} \in \mathbb R^n$, we consider the translated tetrahedron
$T_{\tilde{x},\eps}:= \tilde{x} + T_{\eps}$, whose faces are denoted as $S_{\tilde x,i,\eps}$ and $S_{\tilde x,\eps}$.

Then the balance law gives
\begin{equation}\label{eq:tetrahedron_balancelaw}
    \int_{T_{\tilde x, \eps}} p(x) \,\d x = \sum_{i=1}^n \int_{S_{\tilde x,i,\eps}} f(x,e_i) \,\d \mathcal H^{n-1}
    + \int_{S_{\tilde x,\eps}} f(x,-\nu) \,\d\mathcal H^{n-1},
\end{equation}
where we split the boundary integral in terms of the contributions on the $(n+1)$ faces.
Since the production is bounded, we can estimate
\begin{equation}
    \Big\lvert \int_{T_{\tilde x,\eps}} p(x) \,\d x \Big\rvert \leq \lVert p \rVert_{\Leb^{\infty}(\Omega)} \mathcal L^n(T_{\tilde x,\eps}).
\end{equation}
We then divide both sides of \eqref{eq:tetrahedron_balancelaw} by $\mathcal H^{n-1}(S_{\tilde x,\eps})$
and observe that
\begin{equation*}
(\nu \cdot e_i) \mathcal{H}^{n-1} (S_{\tilde x,\eps})= \mathcal{H}^{n-1}(S_{\tilde x,i, \eps})
\end{equation*}
to deduce
\begin{equation}\label{laultima2}
    \begin{split}
         &\left\lvert\frac{\int_{S_{\tilde x,\eps}}f(x,-\nu) \d \mathcal{H}^{n-1}(x)}{\mathcal{H}^{n-1} (S_{\tilde x,\eps})} + \sum_{i=1}^n (\nu \cdot e_i) \frac{\int_{S_{\tilde x,i,\eps}}f(x, e_i) \d \mathcal{H}^{n-1}(x)}{\mathcal{H}^{n-1}(S_{\tilde x,i,\eps})} \right\rvert \\
         &\leq  \lVert p \rVert_{\Leb^{\infty}(\Omega)} \frac{\mathcal{L}^{n}(T_{\tilde{x},\eps})} {\mathcal{H}^{n-1} (S_{\tilde x,\eps})}  =C(n)\lVert p \rVert_{\Leb^{\infty}(\Omega)} \, \eps.
    \end{split}
\end{equation}

We let $\eps \to 0$ in \eqref{laultima2} and use the fact that $x \mapsto f(x,\nu)$ is continuous to obtain
\begin{equation}\label{6.3a}
    f(\tilde{x},-\nu) + \sum_{i=1}^{n} \nu_i f(\tilde{x}, e_i)=0.
\end{equation}
Define a vector field $\FF=(F_1,F_2,\cdots,F_n)$ by
\begin{equation}
    F_i(x): = f(x,e_i) \qquad \mbox{for $x \in \Omega$}.
\end{equation}
Since $\tilde{x} \in \Omega$ is arbitrary in \eqref{6.3a},
then $\FF$ satisfies
\begin{equation}\label{eq:flux_linearity}
   f(x,-\nu)= -\FF(x) \cdot  \nu
\end{equation}
for any $x \in \Omega$ and $\nu \in \mathbb R^n$ such that $\nu_i \neq 0$ for all $1 \leq i \leq n$.
To extend this to hold for any $\nu$, we can choose a different orthonormal basis $\{\tilde e_i\}_{i=1}^n$ and
work in the associated coordinate system
$(\tilde x_1, \cdots, \tilde x_n)$
to allow for $\nu \in \mathbb R^n$ satisfying $\nu \cdot \tilde e_i \neq 0$ for all $1 \leq i \leq n$ in \eqref{eq:flux_linearity}.
Since $f$ is continuous, it follows that $\FF$ is also continuous.
\end{proof}

The above derivation assumes the existence of a continuous density $f$.
However, in applications, we may naturally encounter solutions that are discontinuous or singular,
thereby violating this continuity hypothesis.
It was not until 1959 that Noll in \cite{Noll2}
considered the problem of removing the continuity assumption of $f$ and proposed an axiomatic scheme
for Continuum Mechanics, which reduces to the familiar balance law for contact forces in the stationary case.
In this setting, the body $\Omega$ is composed of parts that are sets with smooth boundary.
In \cite[{\rm Theorem IV}]{Noll2}, it was shown that the density function should depend only on
the position $x$ and the normal vector $\nu(x)$, independent of other properties of the surface such as the curvature.
However, additional conditions are necessary to derive the linearity of $f$, as it was later remarked
in \cite[{\rm p}.\,79] {Noll1} that
``{\it it is unfortunate that nobody has been able, so far, to {\rm[}establish the linearity of $f${\rm]} under {\rm[}the proposed axioms{\rm]}
without the ad-hoc continuity assumption\dots}''.

From Classical Mechanics,
it follows that the object of study should be the total flux across a surface $S$ contained in $\partial U$,
that is,
\begin{equation}\label{object}
\mathcal{F}(S)= \int_{S} f(x, \nu(x)) \,\d \mathcal{H}^{n-1}(x).
\end{equation}
This viewpoint of studying the \emph{Cauchy flux} $\mathcal F$ was proposed by Gurtin \& Martin in \cite{gm}.
Since they assumed the production $p$ is bounded, the balance law \eqref{CauchyBalanceLaw} immediately implies the inequality:
\begin{equation}\label{acotar}
  \lvert\mathcal{F} (\partial U)\rvert \leq K\mathcal{L}^{n}(U) \qquad \mbox{for any $U \subset \Omega$}.
\end{equation}
Moreover, from \eqref{object}, it also follows that the flux should be an additive function in the sense that
\begin{equation}\label{sumar}
  \mathcal{F} (S_1 \cup S_2) = \mathcal{F}(S_1) + \mathcal{F} (S_2)  \qquad \mbox{for disjoint $S_1, S_2 \subset \partial U$}.
\end{equation}
Therefore, Gurtin \& Martins in \cite{gm}  considered an additive set function $\mathcal F$
defined on (sufficiently regular)
oriented surfaces $(S,\nu)$  which satisfies the properties of being \emph{area-bounded}:
\begin{equation}\label{eq:area_bounded}
  \lvert\mathcal F(S)\rvert \leq C \mathcal H^{n-1}(S),
\end{equation}
and \emph{weakly volume-bounded} in the sense that \eqref{acotar} holds.
Under these assumptions, they deduced the existence of a density function $f(x,\nu)$
defined for $\mathcal L^n$--\textit{a.e.}\,\,$x \in \Omega$ and all unit vectors $\nu$
for which \eqref{object} holds for every admissible surface.
Moreover, the following linearity of $f$ was proved: There exists a field $\FF \in \Leb^{\infty}$ such that
\begin{equation}
\label{linear}
 f(x,\nu)= \FF(x) \cdot \nu \qquad\text{for $\mathcal L^n$--\textit{a.e.\,}}\, x \in \Omega.
\end{equation}
Furthermore,
the following uniform average density condition on $\mathcal{F}$ was introduced:
\begin{equation}
\label{extra}
 \{f_r(x,\nu)\}_{r >0} \qquad\,
 \mbox{with $f_r(x,\nu):= \frac{ \mathcal{F} (D_r(x,\nu))}{\mathcal{H}^{n-1}(D_r(x,\nu))}$},
\end{equation}
which converges uniformly as $r \to 0$ on any compact set $K \subset \Omega$,
where $D_r(x,\nu)$ is the $(n-1)$-dimensional
disk of radius $r$ centered at $x$ with normal $\nu$.
It was proved that $\eqref{extra}$ is a necessary and sufficient condition for the Cauchy flux $\mathcal{F}$ to
have a continuous density $f$ that is linear for every $x$.

The class of admissible surfaces was extended by Ziemer in \cite{Ziemer1} to consider boundaries of sets of finite perimeter.
We refer the reader to \cite{Maggi} for a detailed exposition on the theory of sets of finite perimeter,
which gives a natural class of admissible surfaces that are closed under the set operations
and provides a suitable measure-theoretic notion of the normal.
Under the same assumptions \eqref{acotar}--\eqref{sumar},
considered on sets of finite perimeter,
the existence of a bounded density $f(x,\nu)$ satisfying \eqref{object} on every oriented surface $S$
was established in \cite[{\rm Theorem} 3.3]{Ziemer1}.
Once the existence of a density was established, the linearity of the flux followed from the results in \cite{gm}.
Indeed, in both \cite{gm} and \cite{Ziemer1}, the existence of the vector field is ultimately
based (after suitable mollifications) on the {\it Cauchy's tetrahedron argument}, as sketched above in Theorem \ref{Cau}.
Moreover, it was proved in \cite{Ziemer1} that the constructed vector field $\FF \in \Leb^{\infty}$
satisfies
$\div \FF \in \Leb^{\infty}$ and that the Gauss-Green formula
\begin{equation}
\int_{Q} \d(\div \FF) = -\int_{\partial Q} \FF (x) \cdot \nu (x)\, \d \mathcal{H}^{n-1}(x)
\end{equation}
holds for almost all cubes $Q = (a_1,b_1) \times \cdots\times (a_n,b_n)$.

The boundedness conditions \eqref{acotar} and \eqref{eq:area_bounded} were relaxed by \v{S}ilhav\'y
in \cite{S2, S3} to allow the Cauchy fluxes of the form:
\begin{align}
  &\lvert \mathcal F(S) \rvert \leq \int_S h \,\d\mathcal{H}^{n-1},\label{eq:lp_area_bound}\\
  &\lvert \mathcal F(\partial_{\ast} M)\rvert \leq \int_{M} k \,\d x,\label{eq:lp_volume_bound}
\end{align}
for $h, k \in \Leb^p$ and any Borel set $S \subset \partial_*M$,
where $M$ is assumed to be a normalized set of finite perimeter in $\Omega$
and  $\partial_{\ast}M$ is the measure-theoretic normal.
We denote the collection of such $M$ by $\mathscr P$.
We note that \eqref{eq:lp_area_bound} holds only for a particular representative $h$ in the $\Leb^p$ class, which encodes fine properties of the underlying field, and the right-hand side may fail to be finite in general.
As such the class of admissible surfaces must be relaxed to ``almost all'' surfaces, which is made precise by the class
\begin{equation}\label{eq:Ph_class}
    \mathscr{P}_h = \left\{ M \in \mathscr P\,:\, \mbox{$M\Subset \Omega\,$ and $\,\int_{\partial_{\ast}M} h \,\d\mathcal{H}^{n-1} < \infty$}\right\}.
\end{equation}
This leads to the establishment of a one-to-one correspondence between Cauchy fluxes
and $\Leb^p$--integrable fields $\FF$ with $\Leb^p$ divergence in \cite[Theorem 5.1]{S3}.
It is shown that conditions \eqref{eq:lp_area_bound}--\eqref{eq:lp_volume_bound} imply the existence of a vector field in $\Leb^p$.

However, the $\Leb^p$--regularity of the divergence rules out the important physical behavior where the Cauchy flux may \emph{jump} across the boundary.
One may naturally wonder whether the case of divergence-measure fields can be treated in this framework.
This was answered by Degiovanni, Marzocchi \& Musesti \cite{degiovanni1999cauchy},
providing a definition of Cauchy fluxes representable by $\mathcal{DM}^p$--fields.
For this, the upper bound \eqref{eq:lp_area_bound} for the flux remains the same, while \eqref{eq:lp_volume_bound} is replaced by
the measure bound:
\begin{equation}
    \lvert \mathcal F(\partial_{\ast}M) \rvert \leq \lambda(M)
\end{equation}
for some non-negative Radon measure $\lambda$ on $\Omega$.
In this setting, the flux is recovered on the class of surfaces
\begin{equation}
    \mathscr{P}_{h,\lambda} = \Big\{ M \in \mathscr{P}\,:\,M \Subset \Omega,\,\, \int_{\partial_{\ast}M} h \,\d \mathcal H^{n-1} < \infty,\,\,
    \lambda(\partial_*M) = 0\Big\}.
\end{equation}
However, in general, the condition $\lambda(\partial_{\ast}M)=0$ rules out surfaces along which
$\mathcal F$ has a jump,
say, $\mathcal F(-\partial_{\ast}M) \neq -\mathcal F(\partial_*M)$.
This requirement was removed by Schuricht \cite{Sch}, where the flux is defined for all subsets in $\mathscr P_h$,
where a formulation in terms of contact interactions between two bodies is proposed.

We also refer to \cite{ctz} for bounded divergence-measure fields
in which \emph{every} set of finite perimeter can be treated
and to \cite{ChenComiTorres} in which the limit formula for any general open subset is obtained.

Nevertheless, a characterization of Cauchy fluxes in the extended case has still been missing.
Note that a formulation was proposed in \cite{S4} in which the flux is defined only on planar polyhedra
and involves the estimates valid for {\it almost all} translates, making the hypotheses difficult to verify.
Instead, we now seek a construction in the spirit of \cite{gm,degiovanni1999cauchy,Sch,S3}
that is also sufficiently robust to treat jumps across boundaries and one-sided singularities.

\subsection{Formulation of the Cauchy flux and statement of the main theorem for the Cauchy flux }\label{sec:flux_mainresults}

We begin by precisely defining the admissible class of surfaces under consideration.
This is because the class $\mathscr{P}_h$ from \eqref{eq:Ph_class} has seemingly
no natural replacement when $h$ is replaced by a Radon measure $\mu$ and, moreover,
we wish to depart from sets of finite perimeter and consider general open sets.

\begin{definition}\label{defn:adapted_sop}
For a non-negative and finite Radon measure $\mu$ in $\Omega$, denote
\begin{equation}\label{eq:good_surface_criterion}
\mathcal O_{\mu}:=\Big\{ U \Subset \Omega \,\,\,\mbox{open}\,
:\, \liminf_{\eps \to 0} \frac1{\eps}\,\mu\left(\{ x \in U\,:\,\dist(x,\partial U) < \eps \}\right) < \infty\Big\}.
  \end{equation}
  \end{definition}

This condition is motivated by Theorem \ref{thm:averaged_trace}.
We see later as a consequence of Lemma \ref{lem:disint_flux} that,
for any open set $U \Subset \Omega$,
$U^{\eps} \in \mathcal O_{\mu}$ for $\mathcal L^1$--\textit{a.e.}\,\,$\eps>0$.
Thus, in this sense, $\mathcal O_{\mu}$ contains \emph{almost all} open sets.

\begin{remark}
Recall that, for $U \subset \Omega$ and $\eps>0,$ we write
\begin{equation}
\label{adentro}
  U^{\eps} = \{ x \in U\,:\,\dist(x,\partial U) > \eps \}.
\end{equation}
We also set $U^0 = U$ and
\begin{equation}
\label{afuera}
  U^{-\eps} = \{ x \in \mathbb{R}^n \,:\,\dist(x,U) < \eps \}.
\end{equation}
Using this notation, we can equivalently formulate \eqref{eq:good_surface_criterion} as
  \begin{equation}\label{6.15a}
    \liminf_{\eps \to 0}\, \frac{\mu(U) - \mu(\overline{U}^{\eps})}{\eps}  = \liminf_{\eps \to 0}\, \frac1{\eps}\,\mu(U \setminus \overline{U^{\eps}}) < \infty.
  \end{equation}
Notice that $\mu(\partial U^{\eps}) = 0$ for $\mathcal L^1$--\textit{a.e.}\,\,$\eps>0$
and $\eps \mapsto \mu(\overline U^{\eps})$ is non-increasing and left-continuous.
If $\eps_k \to 0$ is a sequence realizing this limit inferior in \eqref{6.15a},
then, for each $k$, we can find $0<\tilde\eps_k \leq \eps_k$
such that
$\mu(\overline U^{\tilde \eps_k}) - \frac{\eps_k}k
\leq \mu(\overline U^{\eps_k}) \leq \mu(\overline U^{\tilde \eps_k})$
and $\mu(\partial U^{\tilde \eps_k})=0$ for each $k$.
Thus, we always have
  \begin{equation}
      \liminf_{\eps \to 0}\, \frac{\mu(U) - \mu(\overline{U}^{\eps})}{\eps}  =\liminf_{\eps \to 0}\, \frac{\mu(U) - \mu(U^{\eps})}{\eps},
  \end{equation}
regardless of whether both sides are finite or not.
In addition, we may also consider the $\mathcal O_{\mu}$--condition
for subsets $U \subset \mathbb R^n$ that are not necessarily compactly contained in $\Omega$; for this,
we simply extend
measure $\mu$ by zero to $\mathbb R^n$.
\end{remark}

\begin{definition}
\label{DefinitionCauchyFlux}
A \emph{Cauchy flux} in $\Omega$ is a mapping $\mathcal{F}$ defined on pairs $(S,U)$ with
$S \subset \partial U$ Borel and $U \Subset \Omega$ open, so that
the balance law is satisfied{\rm :}
\begin{equation}
\label{balancelaw2}
\mathcal{F}_{U}(\partial U)=\sigma (U)\qquad  \text{ for any $U \Subset \Omega$},
 \end{equation}
 and there is a non-negative finite Radon measure $\mu$ such that the following conditions hold{\rm :}

\begin{enumerate}[label=\rm(\roman*)]

\item\label{item:additivity}
Additivity property{\rm :} For any $U \in \mathcal O_{\mu}$ and disjoint Borel subsets $S,T \subset \partial U$,
\begin{equation}
\mathcal F_{ U}(S \cup T)
= \mathcal F_{U}(S) + \mathcal F_{U}(T).
\end{equation}

\item\label{item:localisation}
Localization property{\rm :} For any $U,V \in \mathcal O_{\mu}$
and any open set $A \subset\Omega$ such that $A \cap U = A \cap V$, then
    \begin{equation}
      \mathcal F_{U}(A \cap \partial U) = \mathcal F_{V}(A \cap \partial V).
    \end{equation}

\item\label{item:eta_bound}
Upper bound{\rm :}
For any $U \in \mathcal O_{\mu}$,
\begin{equation}
    \label{eq:eta_bound}
    \lvert\mathcal F_{U}(S)\rvert \leq  \mu^{n-1}_{U}\left( S \right)
\end{equation}
holds for any Borel set $S \subset \partial U$ and
any limiting measure
\begin{equation}\label{eq:eps_n-1_limit}
      \frac1{\eps_k} \, \mu \res \left(U \setminus \overline{U^{\eps_k}}\right) \,\xrightharpoonup{\,\, * \,\,}
     \mu^{n-1}_{U} \qquad\mbox{as $\eps_k \searrow 0$.}
\end{equation}
\end{enumerate}
\end{definition}

In Remark \ref{rem:eta_bound_lp},
we will draw some connections to Condition \ref{item:eta_bound} in relation to the existing formulations in the $\Leb^p$ setting. However, our condition appears to be fundamentally different
to the class $\mathscr{P}_h$ that appears in the previously existing formulations.

\begin{remark}\label{rem:eta_bound}
In Condition \ref{item:eta_bound}, since $U \in \mathcal O_{\mu}$,
then there is a sequence $\eps_k \searrow 0$ for which the non-negative measure sequence:
  \begin{equation}
      \nu_k := \frac1{\eps_k} \, \mu \res \left( U \setminus \overline{U^{\eps_k}}\right)
  \end{equation}
  satisfies
  \begin{equation}
      \lim_{k \to \infty} \nu_k(\Omega) = \liminf_{\eps \to 0} \frac1{\eps} \, \mu\left( U \setminus \overline{U^{\eps}}\right) < \infty.
  \end{equation}
  Thus, $\{\nu_k\}$ is uniformly bounded in $\mathcal M(\Omega)$ and, up to a subsequence,
  converges weakly${}^*$ to a measure $\nu$.
  Moreover, since each $\nu_k$ vanishes on $(\Omega \setminus \overline{U}) \cup  U^{\eps_k}$ for each $k$, the limit measure $\nu$ is supported on $\partial U$.
  Therefore, such a limiting measure $\mu^{n-1}_{U}$ always exists, even though it is non-unique in general,
  and we impose \eqref{eq:eta_bound} for every limiting measure obtained in this way.
  However, it turns out that this limiting measure is unique on {\it most}  surfaces,
  which will be made precise in Lemma \ref{lem:disint_flux}.
\end{remark}

The main result of this section is that Definition \ref{DefinitionCauchyFlux} implies the existence of a vector-valued Radon measure $\FF=(F_1,\cdots,F_n)$ in $\Omega$,
whose normal trace corresponds to $\mathcal F$.
More precisely, we prove the following main theorem for the Cauchy flux.

\begin{theorem}\label{maintheorem}
Let $\mathcal F$ be a Cauchy flux in $\Omega$.
    Then there exists a unique divergence-measure field $\FF \in \mathcal{DM}^{\ext}(\Omega)$ such that
    \begin{equation}\label{eq:divF_sigma}
        -\div \FF = \sigma
    \end{equation}
    and
    \begin{equation}\label{eq:flux_localrecovery}
        \mathcal F_{U}(S) = (\FF \cdot \nu)_{\partial U}(S)
    \end{equation}
for any $U \in \mathcal O_{\mu}$ and any Borel set $S \subset \partial U$,
where the measure $(\FF \cdot \nu)_{\partial U}$ is the normal trace of $\FF $ on $\partial U$.
Conversely, every $\FF \in \mathcal{DM}^{\ext}(\Omega)$ defines a Cauchy flux,
by taking $\mu = \lvert \FF \rvert$ and $\sigma = -\div \FF$.
\end{theorem}

This makes precise the one-to-one correspondence
\begin{equation}
    \left\{\text{Cauchy fluxes } \mathcal F \text{ in } \Omega \right\}\,\longleftrightarrow\, \mathcal{DM}^{\ext}(\Omega)
\end{equation}
as we claimed in the introduction,
and the local recovery is now precisely formulated via the class $\mathcal O_{\mu}$.
We point out that there is no hope to obtain the recovery on all open sets $U$,
as the normal trace appearing in \eqref{eq:flux_localrecovery} is not generally a measure on arbitrary open sets.

The converse statement of Theorem \ref{maintheorem} is a consequence of the properties of the normal trace established in earlier sections,
which we record here.

\begin{theorem}\label{prop:flux_sufficiency}
Let $\FF \in \mathcal{DM}^{\ext}(\Omega)$. Define
\begin{equation}\label{eq:sufficiency_local}
      \mathcal F_{U}(S) := (\FF \cdot \nu)_{\partial U}(S)
  \end{equation}
for any open set $U \Subset \Omega$ and any Borel set $S \subset \partial U$,
whenever the normal trace $(\FF \cdot \nu)_{\partial U}$ restricts to a measure on $\partial U$.
Furthermore, for $U \Subset \Omega$, define globally
\begin{equation}\label{eq:sufficiency_global}
    \mathcal F_{U}(\partial U) := \langle \FF \cdot \nu, \,\mathbbm{1}_{\Omega} \rangle_{\partial U}.
  \end{equation}
  Then $\mathcal F$ defines a Cauchy flux in the sense of {\rm Definition \ref{DefinitionCauchyFlux}},
  with $\sigma = -\div \FF$ and any non-negative Radon measure $\mu$ such that $\lvert \FF\rvert \leq \mu$.
\end{theorem}

\begin{proof}
By definition of the normal trace, it follows that $\mathcal F_{U}(\partial U) = -\div \FF(U)$.
Using \eqref{eq:trace_balance_law}, we see that the balance law holds for all $U \Subset \Omega$.
Also, if $(\F\cdot \nu)_{\partial U}$ is represented by a measure on $\partial U$, then \eqref{eq:sufficiency_local}
and \eqref{eq:sufficiency_global} coincide when $S=\partial U$.

By Theorem \ref{thm:averaged_trace}, the normal trace $(\FF \cdot \nu)_{\partial U}$ is represented
by a measure whenever $U \in \mathcal O_{\lvert \FF\rvert}$, property \ref{item:additivity} follows
from the fact that $\mathcal O_{\mu} \subset \mathcal O_{\lvert \FF\rvert}$.
The localization property \ref{item:localisation} follows from Theorem \ref{thm:localisation}.
Finally, for \ref{item:eta_bound}, for any $U \in \mathcal O_{\mu}$, it follows from Remark \ref{rem:eta_bound}
that there exists both $\eps_k \to 0$ and a limiting measure $\mu_U^{n-1}$
such that
  \begin{equation}
      \frac1{\eps_k} \mu \res (U \setminus \overline{U^{\eps_k}}) \xrightharpoonup{\,\, * \,\,}
     \mu_U^{n-1}.
  \end{equation}
 Since $\lvert \overline{\nabla d \cdot \FF} \rvert \leq \lvert \FF \rvert \leq \mu$ by \eqref{eq:product_measure_bound}
 from the product rule, we see from Theorem \ref{thm:averaged_trace}
 that
   \begin{equation}
      \frac1{\eps_k}\,\overline{\nabla d \cdot \FF} \res (U \setminus \overline{U}^{\eps_k})
      \xrightharpoonup{\,\, * \,\,}
      (\FF \cdot \nu)_{\partial U},
  \end{equation}
 which implies
  that $\lvert (\FF \cdot \nu)_{\partial U}\rvert \leq \mu_{U}^{n-1}$ as measures.
 Therefore, $\mathcal F$ is a Cauchy flux.
\end{proof}

The main implication of  Theorem \ref{maintheorem} will be proven in several steps,
which will be broken up into separate theorems.
After establishing several useful properties of the flux in \S \ref{sec:flux_properties}, we will first construct a candidate field $\FF$ by integrating the Cauchy flux along hyperplanes.
This will be done in \S \ref{sec:flux_field_construct}, where
the constructed field is shown to satisfy $-\div \FF = \sigma$,
thereby establishing {\it global recovery} of the flux.
Although the construction of $\FF$ involves a choice of measure to integrate the flux,
in \S \ref{sec:flux_field_unique}, we will justify our choice by showing that the said flux is uniquely determined; this will rely on a version of Theorem \ref{thm:reconstruct_normaltrace} valid for half-spaces.
Finally, we will show that the flux and the normal trace coincide in the sense of \eqref{eq:flux_localrecovery}, which will be a {\it local recovery} result, in \S \ref{sec:flux_localrecovery}.
This will first be done on generic cubes, after which an approximation argument
will extend this to general surfaces.
Once the main theorem is proven, we will also collect a few consequences in \S \ref{sec:flux_consequences}.

\section{Cauchy Flux II: Properties of the Cauchy Flux}\label{sec:flux_properties}

\subsection{General properties of the Cauchy flux}\label{sec:flux_property_general}
We now present some useful properties of the Cauchy flux,
starting with some properties of sets in $\mathcal O_{\mu}$ and consequences of property \ref{item:eta_bound}.
Recall that we have defined the class of open subsets:
\begin{equation}\label{eq:oeta_rn}
    \mathcal O_{\mu}
    = \Big\{ U \Subset \Omega \text{ open}\, :\, \liminf_{\eps \to 0} \frac1{\eps} \,\mu(U \setminus \overline{U}^{\eps}) < \infty \Big\}.
\end{equation}
In what follows, we also consider the more restrictive class:
\begin{equation}\label{eq:tilde_oeta}
    \widetilde{\mathcal O}_{\mu}
    = \Big\{ U \Subset \Omega \text{ open}\,:\,\limsup_{\eps \to 0} \frac1{\eps} \,\mu(U \setminus \overline{U}^{\eps}) < \infty \Big\}.
\end{equation}

\begin{lemma}\label{lem:flux_measure}
If $U \in \mathcal O_{\mu},$ then $\mathcal F_{U}(\cdot)$ is represented by a measure supported on $\partial U$.
\end{lemma}

\begin{proof}
 By Condition \ref{item:additivity}, the flux is finitely additive.
 By \ref{item:eta_bound}, there exists some measure $\mu^{n-1}_{U}$ such that
 the following upper bound holds:
  \begin{equation}
    \lvert \mathcal F_{U}(S)\rvert \leq \mu^{n-1}_{U}(S)
    \qquad\mbox{for any Borel set $S \subset \partial U$}.
  \end{equation}
Note that setting $S = \varnothing$ gives $\mathcal F_{U}(\varnothing) = 0$.
For countable additivity, let $\{S_k\}_{k=1}^{\infty}$ be disjoint Borel subsets of $\partial U$,
and set $S = \bigcup_{k=1}^{\infty} S_k$.
Then, for each $N \geq 1$, we can estimate
  \begin{equation}
    \begin{split}
      \Big\lvert \mathcal F_{U}(S) - \sum_{k=1}^N \mathcal F_{U}(S_k)\Big\rvert
      &= \Big\lvert \mathcal F_{U}( \bigcup_{k=N+1}^{\infty} S_k ) \Big\rvert \\
      &\leq \mu^{n-1}_{U}( \bigcup_{k=N+1}^{\infty} S_k)
      = \sum_{k=N+1}^{\infty} \mu^{n-1}_{U}(S_k) \to 0 \qquad\,\mbox{as $N \to \infty$},
    \end{split}
  \end{equation}
where we have repeatedly used the additivity property \ref{item:additivity},
in addition to \ref{item:eta_bound} and the fact that $\mu_U^{n-1}$ is a measure.
Thus, it follows that
  \begin{equation}
    \mathcal F_{U}(S) = \sum_{k=1}^{\infty} \mathcal F_{U}(S_k),
  \end{equation}
  from which the result follows.
\end{proof}

\begin{lemma}\label{lem:sop_algebra}
Suppose that $U_1 \in \mathcal O_{\mu}$ and $U_2 \in \widetilde{\mathcal O}_{\mu}$.
Then $U_1 \cap U_2$ and $U_1 \cup U_2$ lie in $\mathcal O_{\mu}$.
In addition, if $U_1 \in \widetilde{\mathcal O}_{\mu}$, then $U_1 \cap U_2,\, U_1 \cup U_2 \in \widetilde{\mathcal O}_{\mu}$.
\end{lemma}

\begin{proof}
  For any $U_1, U_2 \subset \Omega$, set $U = U_1 \cap U_2$ and  $V = U_1 \cup U_2.$ Then we claim:
  \begin{align}
    U\setminus \overline{U^{\eps}} &\subset (U_1 \setminus \overline{U_1^{\eps}}) \cup (U_2 \setminus \overline{U_2^{\eps}}), \label{eq:intersection_inclusion}\\
    V\setminus \overline{V^{\eps}} &\subset (U_1 \setminus \overline{U_1^{\eps}}) \cup (U_2 \setminus \overline{U_2^{\eps}}).\label{eq:union_inclusion}
  \end{align}
  Indeed, if $x \in U \setminus \overline{U^{\eps}}$, then there exists $y \in \mathbb R^n \setminus U$
  such that $\lvert x - y \rvert < \eps$.
  Since $\mathbb R^n \setminus U = (\mathbb R^n \setminus U_1) \cup (\mathbb R^n \setminus U_2)$,
  then $x \in U_1 \setminus \overline{U_1^{\eps}}$ if $y \in \mathbb R^n \setminus U_1$,
  and $x \in U_2 \setminus \overline{U_2^{\eps}}$ if $y \in \mathbb R^n \setminus U_2$.
  Similarly, if $x \in V \setminus V^{\eps}$, there is $y \in \mathbb R^n \setminus V$
  with $\lvert x-  y\rvert < \eps$.
  If $x \in U_1$, then $y \in \mathbb R^n \setminus U_1$ so that $x \in U_1 \setminus \overline{U_1^{\eps}}$;
  similarly, if $x \in U_2$, then $x \in U_2 \setminus \overline{U_2^{\eps}}$.

  Thus, taking $\mu$ on both sides of \eqref{eq:intersection_inclusion}, we obtain
  \begin{equation}\label{eq:eta_Uintersection}
      \frac1{\eps} \, \mu(U \setminus \overline{U^{\eps}})
      \leq \frac1{\eps} \, \mu(U_1 \setminus \overline{U_1^{\eps}})
      + \frac1{\eps} \, \mu(U_2 \setminus \overline{U_2^{\eps}})\qquad \mbox{ for any $\eps >0$},
  \end{equation}
  so that
  \begin{equation}
      \liminf_{\eps \to 0} \frac1{\eps} \, \mu(U \setminus \overline{U^{\eps}})
      \leq \liminf_{\eps \to 0} \frac1{\eps} \, \mu(U_1 \setminus \overline{U_1^{\eps}})
      + \limsup_{\eps \to 0} \frac1{\eps} \, \mu(U_2 \setminus \overline{U_2^{\eps}}) < \infty,
  \end{equation}
  which implies that $U \in \mathcal O_{\mu}$.
  Moreover, if $U_1 \in \widetilde{\mathcal O}_{\mu}$, we have
  \begin{equation}
      \limsup_{\eps \to 0} \frac1{\eps} \, \mu(U \setminus \overline{U^{\eps}})
      \leq \limsup_{\eps \to 0} \frac1{\eps} \, \mu(U_1 \setminus \overline{U_1^{\eps}})
      + \limsup_{\eps \to 0} \frac1{\eps} \, \mu(U_2 \setminus \overline{U_2^{\eps}}) < \infty,
  \end{equation}
  so $U \in \widetilde{\mathcal O}_{\mu}$.
  The same result holds for $V$ by using \eqref{eq:union_inclusion} instead to obtain an estimate analogous to \eqref{eq:eta_Uintersection}.
\end{proof}

Adapting the arguments of the above proof allows us to establish the following upper bound
for the limit measure $\mu_{\partial U}^{n-1}$.

\begin{lemma}\label{lem:eta_intersection}
Let $U_1, \cdots, U_N \in \mathcal O_{\mu}$ such that there exists a common subsequence $\eps_k \to 0$ so that
  \begin{equation}
    \frac1{\eps_k} \, \mu \res (U_i \setminus \overline{U_i^{\eps_k}}) \xrightharpoonup{\,\, * \,\,}
    \mu^{n-1}_{U_i} \qquad
    \mbox{for each $1 \leq i \leq N$},
  \end{equation}
 and set
  \begin{equation}
      U = \bigcap_{i=1}^N U_i, \quad V = \bigcup_{i=1}^N U_i.
  \end{equation}
  \begin{enumerate}[label=\rm(\roman*)]
  \item\label{item:upperbound_intersection}
  If
  $\frac1{\eps_k} \, \mu \res (U \setminus \overline {U^{\eps_k}}) \xrightharpoonup{\,\, * \,\,}
  \mu^{n-1}_{U},$ then
      \begin{equation}
  \mu^{n-1}_{U}(S) \leq \sum_{i=1}^N \mu^{n-1}_{U_i}(S \cap \partial U_i \cap \bigcap_{j \neq i} \overline{U_j})
  \qquad\mbox{ for any Borel set $S \subset \partial U$}.
      \end{equation}

\smallskip
      \item\label{item:upperbound_union}
      If
      $\frac1{\eps_k} \, \mu \res (V \setminus \overline {V^{\eps_k}})\xrightharpoonup{\,\, * \,\,}
      \mu^{n-1}_{V}$, then
      \begin{equation}
          \mu_V^{n-1}(S) \leq \sum_{i=1}^N \mu_{U_i}^{n-1}( S \cap (\partial U_i \setminus \bigcup_{j \neq i} U_j))
       \qquad\mbox{ for any Borel set $S \subset \partial V$}.
       \end{equation}
 \end{enumerate}
\end{lemma}

\begin{proof}
  For \ref{item:upperbound_intersection},
  arguing as in \eqref{eq:intersection_inclusion} from the proof of Lemma \ref{lem:sop_algebra}, we have
  \begin{equation}
    U \setminus \overline{U^\eps} \subset \bigcup_{i=1}^N \left( U_i \setminus \overline{U_i^{\eps}} \right)
    \qquad\mbox{ for $\eps>0$}.
  \end{equation}
  Given any non-negative $\phi \in C_{\mathrm{c}}(\Omega)$,
  we can integrate $\phi$ over these levels sets with respect to $\mu$ and take the limit as $\eps_k \to 0$ to obtain
  \begin{equation}
    \int_{\partial U} \phi \,\d\mu^{n-1}_{U} \leq \sum_{i=1}^N \int_{\partial U_i} \phi \,\d\mu^{n-1}_{U_i}.
  \end{equation}
Then we can argue by density to infer that
  \begin{equation}
      \mu_U^{n-1}(S) \leq \sum_{i=1}^N \mu_{U_i}^{n-1}(S \cap \partial U_i) \qquad\mbox{for any Borel set $S\subset \Omega$}.
  \end{equation}
Now, since $\mu^{n-1}_{\partial U}$ is supported on $\partial U$, we can replace $S$ by $S \cap \partial U$ on the right-hand side.
Moreover, since
  \begin{equation}
      \partial U_i \cap \partial U \subset \partial U_i \cap \bigcap_{j \neq i} \overline{U_j},
  \end{equation}
  the result follows.
  The argument for
   \ref{item:upperbound_union}
  is analogous, based on the inclusions:
  \begin{equation}
      V \setminus \overline{V^{\eps}} \subset \bigcup_{i=1}^N (U_1 \setminus \overline{U_1^{\eps}})
  \end{equation}
  for each $\eps>0$ which is established analogously to \eqref{eq:union_inclusion} and
  \begin{equation}
      \partial V \subset \bigcup_{i=1}^N \partial U_i \setminus \bigcup_{j \neq i} U_j.
  \end{equation}
\end{proof}

\begin{lemma}\label{lem:disint_flux}
Let $\mathcal F$ be a Cauchy flux, and let $U \subset \mathbb R^n$ be an open set.
  Consider the disintegration of $\mu$ along $\partial U^t$ given by
  \begin{equation}
    \mu = \mathcal L^1 \res [0,\infty) \otimes_{\partial U^t} \mu_t + \tau_{\mathrm{sing}} \otimes_{\partial U^t} \tilde\mu_t
  \end{equation}
  from {\rm Lemma \ref{lem:disint_tau_decomp}}, extending $\mu$ by zero to a measure on $\mathbb R^n$.
  Then, for $\mathcal L^1$--\textit{a.e.}\,\,$t>0$,
  \begin{align}\label{eq:disint_mut_conv}
    \frac1{\eps} \, \mu \res  \Omega \cap U^{t} \setminus \overline{U^{t+\eps}} ) \xrightharpoonup{\,\, * \,\,}
    \mu_t,
    \quad
    \frac1{\eps} \, \mu \res ( \Omega \cap U^{t-\eps} \setminus \overline{U^{t}} )  \xrightharpoonup{\,\, * \,\,}
     \mu_t
\qquad\,\,\,\mbox{ as $\eps \to 0$},
\end{align}
and $\mu(\Omega \cap \partial U^t) = 0$.
In particular, if $U \Subset \Omega$, then $U^t, (\overline{U}^t)^{\mathrm{c}} \in \mathcal O_{\mu}$,
and $\mu^{n-1}_{U^t}$ and $\mu^{n-1}_{(\overline{U}^t)^{\mathrm{c}}}$ are uniquely determined as $\mu_t$
for $\mathcal{L}^{1}$--\textit{a.e.}\,\,$t>0$.
\end{lemma}

\begin{proof}
We apply Lemma \ref{lem:disint_tau_decomp}, which asserts that, for any $\phi \in C_{\mathrm{c}}(\Omega)$ and any $0<t<s$,
  \begin{equation}
    \frac1{s-t} \int_{U^t \setminus \overline{U^{s}}} \phi\,\d \mu
    = \frac1{s-t} \int_t^{s} \int_{\partial U^r} \phi \,\d\mu_r \,\d r
    + \frac1{s-t}\int_{(t,s)} \int_{\partial U^r} \phi \,\d\tilde{\mu}_r \,\d \tau_{\mathrm{sing}}(r).
  \end{equation}
For $\eps > 0$, we apply the above with $s=t+\eps$.
For any countable dense subset $\{\phi_j\}$ of $C_{\mathrm{c}}(\Omega)$ (with respect to the uniform convergence),
consider the partial mappings:
  \begin{align}
    \Phi_j(t) &:= \int_{\partial U^t} \phi_j \,\d \mu_t
    \qquad\text{for each } j.\label{eq:partialmap_leb}
  \end{align}
Then
we choose $t>0$ to be a Lebesgue point for each $\Phi_j$ and to satisfy $D_{\mathcal L^1}(\tau_{\mathrm{sing}})(t) = 0$ (using Lemma \ref{elchiste}),
which forms a set whose complement is $\mathcal L^1$--null.
For such $t$, sending $\eps \to 0$ gives that, for each $j$,
\begin{equation}
\lim_{\eps \to 0} \frac1{\eps} \int_t^{t+\eps} \int_{\partial U^s} \phi_j \,\d \mu_s \,\d s = \Phi(t) = \int_{\partial U^t} \phi_j \,\d \mu_t,
\end{equation}
and
  \begin{equation}
    \limsup_{\eps \to 0} \frac1{\eps} \Big\lvert \int_{(t,t+\eps)} \int_{\partial U^t} \phi_j \,\d \tilde{\mu}_s\,\d\tau_{\mathrm{sing}}(s)\Big\rvert
     \leq \lVert \phi_j \rVert_{\Leb^{\infty}(\Omega)} \limsup_{\eps \to 0} \frac1{\eps} \, \lvert \tau_{\mathrm{sing}}((t,t+\eps))\rvert = 0.
  \end{equation}
Combining the above, we deduce
  \begin{equation}\label{eq:limit_denseclass}
    \lim_{\eps \to 0} \frac1{\eps} \int_{U^t \setminus \overline U^{t+\eps}} \phi_j\,\d \mu
    = \int_{\partial U^t} \phi_j \,\d\mu_t.
  \end{equation}
Applying the same to $(t,s) = (t-\eps,t)$ with $0<\eps<t$, we have
  \begin{equation}
      \lim_{\eps \to 0} \frac1{\eps} \int_{U^{t-\eps} \setminus \overline{U^t}} \phi_j \,\d \mu
      = \int_{\partial U^t} \phi_j \,\d \mu_t
     \qquad \mbox{ for the same $t$}.
  \end{equation}
Discarding a null set if necessary, we can moreover assume that
  \begin{equation}
    \limsup_{\eps \to 0}\frac1{2\eps} \, \mu(U^{t-\eps} \setminus \overline U^{t+\eps}) < \infty,
  \end{equation}
so, in particular, $U^t, (\overline{U}^t)^{\mathrm{c}} \in \widetilde{\mathcal O}_{\mu}$
and $\mu(\Omega \cap \partial U^t) = 0$.
Then, for any sequence $\eps_k \to 0$ giving rise to a limiting measure
\begin{equation}
\frac1{\eps} \, \mu \res U^t \setminus \overline U^{t+\eps_k} \xrightharpoonup{\,\, * \,\,} \nu,
\end{equation}
testing against $\phi_j$ and using \eqref{eq:limit_denseclass} yields
  \begin{equation}
    \int_{\Omega} \phi_j \,\d \nu = \int_{\partial U^t} \phi_j \,\d \mu_t \qquad \mbox{for any $j$.}
  \end{equation}
By density of $\{\phi_j\}$, it follows that $\nu = \mu^t \res \partial U^t$.
This uniquely determines the limit. We can similarly show that
\begin{equation}
\frac1{\eps} \mu \res U^{t-\eps} \setminus \overline{U^{t}}\xrightharpoonup{\,\, * \,\,} \mu^t,
\end{equation}
as required.
\end{proof}

\begin{remark}\label{rem:eta_bound_lp}
In the theory of $\Leb^p$--integrable fields, one may consider
$\mu = h \,\mathcal L^n$ with a non-negative function $h \in \Leb^p(\Omega)$.
In this case, the disintegration reads as
    \begin{equation}
        \mu = \mathcal L^1 \res [0,\infty) \otimes_{\partial U^t} \mu_t,
        \qquad \mu_t = \left.h\right\rvert_{\partial U^t} \mathcal H^{n-1} \res \partial U^t
    \end{equation}
by using the coarea formula, where a representative of $h$ is fixed.
In this case, using Lemma \ref{lem:disint_flux}, condition \ref{item:eta_bound} implies that
\begin{equation}
\lvert \mathcal F_{ U^t}(S) \rvert \leq \int_{S} h \,\d \mathcal H^{n-1}
\qquad\mbox{ for $\mathcal L^1$--\textit{a.e.}\,\,$t>0$ and any Borel set $S \subset \partial U^t$.}
\end{equation}
This is reminiscent of the flux bound appearing in
\cite{S3,degiovanni1999cauchy, Sch, ChenComiTorres},
even though the notions need not coincide for a general open set $U$.
\end{remark}

For the next result, we introduce the sets:
\begin{equation}\label{eq:tilde_omega_delta}
    \widetilde\Omega^{\delta} = B_{1/\delta} \cap \Omega^{\delta}
    = \Big\{ x \in \Omega\,:\,\dist(x,\partial\Omega)>\delta, \ \lvert x \rvert < \frac1{\delta} \Big\}
    \qquad\mbox{for each $\delta>0$.}
\end{equation}
Note that each $\widetilde\Omega^{\delta} \Subset \Omega$ and $\bigcup_{\delta>0} \widetilde\Omega^{\delta} = \Omega$.
Using Lemmas \ref{lem:sop_algebra} and \ref{lem:disint_flux},
we see that $\widetilde\Omega^{\delta} \in \widetilde{\mathcal O}_{\mu}$ for $\mathcal L^1$--\textit{a.e.}\,\,$\delta>0$.

\begin{lemma}\label{lem:flux_extension}
Suppose that $U \subset \mathbb R^n$ is open such that
  \begin{equation}\label{eq:generalopen_etalimit}
    \frac1{\eps} \, \mu \res (\Omega \cap (U \setminus \overline{U^{\eps}}))
   \xrightharpoonup{\,\, * \,\,}
    \mu^{n-1}_{U} \qquad\text{as } \eps\to0.
  \end{equation}
Then an extension of the Cauchy flux can be defined by
\begin{equation}\label{eq:flux_extension}
    \mathcal F_{U}(S) := \mathcal F_{U \cap \widetilde\Omega^{\delta}}(S)
\end{equation}
for $\mathcal L^1$--\textit{a.e.}\,\,$\delta>0$ such that $\widetilde\Omega^{\delta} \in \widetilde{\mathcal O}_{\mu}$
and for any Borel set $S \subset\partial U$ such that $S \subset \widetilde\Omega^{\delta} \cap \partial U$.
Moreover, this extends uniquely to a measure on $\Omega \cap \partial U$ satisfying
  \begin{equation}\label{eq:extension_eta_bound}
    \lvert \mathcal F_{U}(S)\rvert \leq \mu^{n-1}_{U}(S)
    \qquad\text{for any Borel set } S \subset \Omega \cap \partial U.
    \end{equation}
\end{lemma}

\begin{proof}
First, the weak${}^{\ast}$--convergence in \eqref{eq:generalopen_etalimit} implies that
  \begin{equation}
    \limsup_{\eps \to 0} \frac1{\eps} \, \mu(\Omega \cap (U \setminus \overline{U^{\eps}})) < \infty.
  \end{equation}
By Lemma \ref{lem:sop_algebra} applied in $\mathbb R^n$ via extending $\mu$ by zero outside $\Omega$,
if $\delta>0$ such that $\widetilde\Omega^{\delta} \in \widetilde{\mathcal O}_{\mu}$,
then $\widetilde\Omega^{\delta} \cap U \in \widetilde{\mathcal O}$.
By property \ref{item:localisation}, we see that
\begin{equation}
\mathcal F_{U \cap \widetilde\Omega^{\delta_1}}(A \cap \partial U)
= \mathcal F_{U \cap \widetilde\Omega^{\delta_2}}(A \cap \partial U)
\end{equation}
for any $\delta_1$ and $\delta_2$ such that $0 < \delta_1 < \delta_2$, as above,
and $A \subset \widetilde\Omega_{\delta_2}$ open.
Since the flux is a measure on $\mathcal O_{\mu}$ by Lemma \ref{lem:flux_measure},
by approximation (Lemma \ref{lem:measure_coincidence}), the above extends to general Borel sets $S \subset \widetilde\Omega^{\delta_2} \cap \partial U$.
Hence, \eqref{eq:flux_extension} is well-defined.
To show that it extends to a measure, using Lemma \ref{lem:disint_flux},
let $\delta>0$ such that
  \begin{equation}\label{eq:omega_delta_limit}
  \frac1{\eps} \, \mu \res (\Omega^{\delta} \setminus \overline{\Omega^{\delta+\eps}})
  \xrightharpoonup{\,\, * \,\,}
   \mu^{n-1}_{\Omega^{\delta}},
  \quad \frac1{\eps} \, \mu \res (B_{1/\delta} \setminus \overline{B_{1/\delta}^\eps})
  \xrightharpoonup{\,\, * \,\,}
   \mu^{n-1}_{B_{1/\delta}}.
  \end{equation}
Then, since $U \cap \widetilde\Omega^{\delta} \in \mathcal O_{\mu}$ by Lemma \ref{lem:eta_intersection}, we can estimate
  \begin{equation}
    \mu^{n-1}_{U \cap \widetilde\Omega^{\delta}}
    \leq \mu^{n-1}_{U} \res \overline{ \widetilde\Omega^{\delta}}
    + (\mu^{n-1}_{\Omega^{\delta}} + \mu^{n-1}_{B^{1/\delta}})\res (\overline{ \widetilde\Omega^{\delta}}\cap  \overline{U})
  \end{equation}
  as measures.
Then, if $S \subset \Omega \cap \partial U$ such that $S \Subset \widetilde\Omega^{\delta}$, we obtain the claimed bound
  \begin{equation}
    \lvert \mathcal F_{U}(S) \rvert \leq \mu^{n-1}_{U}(S).
  \end{equation}
Since this estimate is independent of $\delta>0$, we claim that
$\mathcal F_U$ extends uniquely to a measure on $\Omega \cap \partial U$.
Indeed, take $\delta_j \searrow 0$ such that each $\widetilde\Omega^{\delta_j} \in \widetilde{\mathcal O}_{\mu}$
and, given $S \subset \Omega \cap \partial U$, define $S_1 = S \cap \widetilde\Omega^{\delta_1}$ and
$S_j = S \cap (\widetilde\Omega^{\delta_j} \setminus \overline{\widetilde\Omega^{\delta_{j-1}}})$ for $j \geq 2$.
Then, by additivity, we have
  \begin{equation}
      \sum_{j=1}^N \mathcal F_U(S_j) = \mathcal F_U( \bigcup_{j=1}^N S_j),
  \end{equation}
  which is uniformly bounded in $N$ by noting that
  \begin{equation}
      \sum_{j=1}^{\infty}\lvert \mathcal F_U(S_j)\rvert \leq \mu_U^{n-1}(S) < \infty.
  \end{equation}
 Thus, there is a unique limit
  \begin{equation}
      \mathcal F_U(S) := \sum_{j=1}^{\infty} \mathcal F_{U}(S_j),
  \end{equation}
which also satisfies \eqref{eq:extension_eta_bound}.
Arguing analogously with disjoint sets $S_1, S_2\subset \partial U$ and considering
$S_{i,j} = S_i \cap (\widetilde\Omega^{\delta_j} \setminus \overline{\widetilde\Omega^{\delta_{j-1}}})$,
we see that this extension is also additive.
Since this extension is additive and satisfies \eqref{eq:extension_eta_bound},
  arguing as in the proof of Lemma \ref{lem:flux_measure}, we conclude that
  this extension $\mathcal F_U$ is a measure on $\Omega \cap \partial U$.
\end{proof}

\begin{remark}\label{rem:flux_complement}
In the case that $U = \overline{V}^{\mathrm{c}} = \mathbb R^n \setminus {\overline V}$ for some $V \Subset \Omega$,
it suffices to assume that
    \begin{equation}
        \liminf_{\eps \to 0} \frac1{\eps} \,\mu(\Omega \cap (V^{-\eps} \setminus \overline{V})) < \infty,
    \end{equation}
which is written as $\overline{V}^{\mathrm{c}} \in \mathcal O_{\mu}$.
In this case, we can set
 \begin{equation}\label{eq:flux_complement}
 \mathcal F_{\overline{V}^{\mathrm{c}}} := \mathcal F_{\widetilde\Omega^{\delta} \setminus \overline{V}} \res \tilde\Omega^{\delta}
\end{equation}
for any $\delta>0$ such that $\delta < \dist(V,\partial\Omega)$, $V \subset B_{1/\delta}$,
and $\widetilde\Omega^{\delta} \in \widetilde{\mathcal O}_{\mu}$.
Indeed, by Lemma \ref{lem:sop_algebra} (applied in $\mathbb R^n$ extending $\mu$ by zero),
$\widetilde\Omega^{\delta} \setminus \overline V \in \mathcal O_{\mu}$, so it follows from
Lemma \ref{lem:flux_measure} that $\mathcal F_{\widetilde\Omega^{\delta} \setminus \overline V}$ is a measure.
By the localization property \ref{item:localisation}, its restriction to $\partial V$ is independent of the choice of $\delta>0$,
so \eqref{eq:flux_complement} is well-defined.
Also, for any such $\delta$, we can set
$\mu_{\overline{V}^{\mathrm{c}}}^{n-1} := \mu_{\widetilde\Omega^{\delta} \setminus \overline{V}}^{n-1} \res \partial V$,
which satisfies
    \begin{equation}
        \lvert \mathcal F_{\overline{V}^{\mathrm{c}}}(S) \rvert \leq \mu_{\overline{V}^{\mathrm{c}}}^{n-1}(S)
        \qquad\text{for any Borel set } S \subset \partial V.
          \end{equation}
\end{remark}

\subsection{ Properties of the Cauchy flux on cubes and half-spaces}\label{sec:flux_property_cubes}
Our proof of Theorem \ref{maintheorem} relies on considering the Cauchy flux $\mathcal F$
on the boundaries of half-spaces and cubes, and integrating along suitable hyperplanes to construct the desired field $\FF$.
Since a general half-space $H = \{ x \in \mathbb R^n\, :\, a \cdot x > t\}$ need not be contained in $\Omega$,
we use Lemma \ref{lem:flux_extension} to consider the flux on $\Omega \cap \partial H$ for {\it almost all} $H$.
In doing so, it is useful to understand the behavior of the flux on cubes and, in particular, write $\mathcal F_{Q}(\cdot)$
as a sum of fluxes going through the respective faces.
It turns out that this is not possible for arbitrary cubes,
since the vector field may potentially concentrate
at the corners, as illustrated for the normal traces
in Example \ref{eq:localisation_counterexample}.

However, localization to faces is possible for {\it almost all} cubes, which we will make precise below.
For this, it is useful to introduce some notation.
For simplicity, we consider only cubes and hyperplanes subordinate to the coordinate axes $x_1, \cdots, x_n$,
even though more general frames in $\mathbb R^n$ can also be considered
as done in \cite{degiovanni1999cauchy} with minimal modifications.

\begin{definition}
For a cube $Q = Q(a,b) := (a_1,b_1) \times \cdots \times (a_n,b_n)$, define the
$2$-skeleton of $Q$ as
\begin{equation}
\partial^2Q = \big\{ x \in \overline Q\,:\,x_i \in \{a_i,b_i\}, x_j \in \{ a_j,b_j\} \text{ for some } i \neq j \big\}.
\end{equation}
Note that, in two dimensions, $\partial^2Q$ consists of the corner points,
whereas it is the union of all of the edges of $Q$ $($including the vertices$)$ in three dimensions.

Additionally, for $1 \leq i \leq n$, define the half-spaces
  \begin{equation}
    H_{i, +}^t = \big\{ x \in \mathbb R^n\,:\,x_i > t\big\},
    \quad H_{i, -}^t = \big\{ x \in \mathbb R^n\,:\,x_i < t\big\}.
  \end{equation}
Note that $(H_{i,+}^t)^{\eps} = H_{i,+}^{t+\eps}$ and $(H_{i,-}^t)^{\eps} = H_{i,-}^{t-\eps}$ by recalling definition \eqref{adentro}.
In particular, $\partial H_{i,+}^{a_i}$ and $\partial H_{i,-}^{b_i}$ contain the faces of $Q$ for each $1 \leq i \leq n$.
\end{definition}

We also define suitable null sets to make precise our notions of {\it almost all} cubes and half-spaces.

\begin{definition}\label{defn:good_hyperplanes}
  Given the Radon measures $\mu$ and $\sigma$ from {\rm Definition \ref{DefinitionCauchyFlux}},
  for $1 \leq i \leq n$,  define $\mathcal N_{i,\mu} \subset \mathbb R$
  to be the complement of the set of points $t \in \mathbb R$ for which
  $\mu(\partial H_{i,+}^t )=0$, and  the limits{\rm :}
\begin{equation}
\frac1{\eps}\,\mu \res (H_{i,+}^{t} \setminus \overline{H_{i,+}^{t+\eps})}
\xrightharpoonup{\,\, * \,\,}
\mu_{H_{i,+}^{t}}^{n-1}
\quad\mbox{and}\quad \frac1{\eps}\,\mu \res (H_{i,-}^{t-\eps} \setminus \overline{H_{i,-}^{t})}
\xrightharpoonup{\,\, * \,\,}
\mu_{H_{i,-}^{t}}^{n-1}
\end{equation}
exist as $\eps \to 0$ with $\mu_{H_{i,+}^{t}}^{n-1} = \mu_{H_{i,-}^{t}}^{n-1}$,
where $\mu$  is viewed as a measure on $\mathbb R^n$ by setting $\lvert \mu\rvert(\mathbb R^n \setminus \Omega) = 0$,
and the set $\mathcal N_{i,\mu} \subset \mathbb R$ is
$\mathcal L^1$-null by {\rm Lemma \ref{lem:disint_flux}}.
Furthermore, define $\mathcal M_{i,\sigma} \subset \mathbb R$
as the complement of the set of points $t \in \mathbb R$ such that
  \begin{equation}
    \lvert\sigma\rvert( \partial H_{i,+}^t) = 0,
  \end{equation}
  which is at most countable and hence $\mathcal L^1$-null.
\end{definition}

Observe that Lemma \ref{lem:flux_extension} applies to the half-spaces $H_{i,\pm}^t$, provided
that $t \notin \mathcal N_{i,\mu}$, allowing us to make sense of $\mathcal{F}_{ H_{i,\pm}^t}$,
which will frequently be used in what follows.

\begin{lemma}\label{eq:cube_eta_bound}
  Let $Q = Q(a,b) \Subset \Omega$ be a cube such that $a_i,b_i \notin \mathcal N_{i,\mu}$ and
    \begin{equation}\label{eq:corner_small}
    \mu_{H_{i,+}^{a_i}}^{n-1}(\partial^2Q) = \mu_{H_{i,-}^{b_i}}^{n-1}(\partial^2Q) = 0 \qquad
    \mbox{for any $1 \leq i \leq n$}.
  \end{equation}
 Then $Q, \overline Q^{\mathrm{c}} \in \widetilde{\mathcal O}_{\mu}$,
 and there is a unique measure $\mu_Q^{n-1}$ such that
  \begin{equation}\label{eq:cube_eta_unique}
    \frac1{\eps} \,\mu(Q \setminus \overline{Q^{\eps}})
    \xrightharpoonup{\,\, * \,\,}
     \mu_Q^{n-1}, \quad
    \frac1{\eps} \,\mu(Q^{-\eps} \setminus \overline{Q})
    \xrightharpoonup{\,\, * \,\,}
     \mu_Q^{n-1}\qquad\,\, \mbox{  as $\eps \to 0$}.
  \end{equation}
Furthermore,  for each $1 \leq i \leq n$,
  \begin{align}
    \mu^{n-1}_Q \res (\partial H_{i,+}^{a_i} \cap \partial Q)
      &= \mu^{n-1}_{H_{i,+}^{a_i}} \res(\partial H_{i,+}^{a_i} \cap \partial Q), \\
    \mu^{n-1}_Q \res (\partial H_{i,-}^{b_i} \cap \partial Q)
      &= \mu^{n-1}_{H_{i,-}^{b_i}} \res(\partial H_{i,-}^{b_i} \cap \partial Q),
  \end{align}
which uniquely determines $\mu_Q^{n-1}$.
  In particular, $\mu_Q^{n-1}(\partial^2Q) = \mu_{\,\overline Q^{\mathrm c}}^{n-1}(\partial^2Q) = 0$.
\end{lemma}

\begin{proof}
Viewing $\mu$ as a measure on $\mathbb R^n$, we see that
each $H_{i,+}^{a_i}, H_{i,-}^{b_i} \in \widetilde{\mathcal O}_{\mu}$.
Then it follows from Lemma \ref{lem:sop_algebra}
that $Q \in  \widetilde{\mathcal O}_{\mu}$.
Now, for any $\eps_k \searrow 0$, there exists a limit measure $\lambda$ such that
  \begin{equation}
    \frac1{\eps_k} \, \mu \res (Q \setminus \overline{Q^{\eps_k}}) \xrightharpoonup{\,\, * \,\,}
       \lambda
    \qquad\,\,\text{as } k \to \infty.
  \end{equation}
We now show this limit measure is uniquely determined,
from which the full convergence \eqref{eq:cube_eta_unique} follows.

Since $Q = \bigcap_{i=1}^n (H_{i,+}^{a_i} \cap H_{i,-}^{b_i})$
by using Lemma \ref{lem:eta_intersection}\ref{item:upperbound_intersection} together
with \eqref{eq:corner_small}, we have
\begin{equation}\label{eq:cube_corner_small}
\lambda(\partial^2 Q) = 0.
\end{equation}
Moreover, if $\phi \in C_{\mathrm{c}}(\Omega \setminus \partial^2Q)$,
then $\phi$ is supported in some open subset $A \Subset \Omega \setminus \partial^2Q$.
Choosing $\eps>0$ such that $\eps < \dist(A,\partial^2Q)$ and $\eps < \frac12(b_j - a_j)$ for each $1 \leq j \leq n$, we have
\begin{equation}
A \cap (Q \setminus \overline{Q^{\eps}})
= \bigcup_{j=1}^n \Big(A \cap (H_{j,+}^{a_j} \setminus \overline{H_{j,+}^{a_j+\eps}})\Big)
\cup \bigcup_{j=1}^n \Big(A \cap (H_{j,-}^{b_j} \setminus \overline{H_{j,-}^{b_j-\eps}})\Big),
\end{equation}
and the sets of the right-hand side are disjoint.
Hence, integrating over $\phi$ with respect to $\frac1{\eps} \,\mu$
and sending $\eps_k \searrow 0$ give
\begin{equation}
\int_{\partial Q} \phi \,\d \lambda
= \sum_{j=1}^n \int_{\partial H_{j,+}^{a_j}} \phi \,\d\mu_{H_{j,+}^{a_j}}^{n-1}
  +\sum_{j=1}^n \int_{\partial H_{j,-}^{b_j}} \phi \,\d\mu_{H_{j,-}^{b_j}}^{n-1}.
\end{equation}
This, combined with \eqref{eq:cube_corner_small}, uniquely determines
$\lambda = \mu_{ Q}^{n-1}$.
Similarly, for the complement, we obtain that
$\overline Q^{\mathrm{c}} \in \widetilde{\mathcal O}_{\mu}$
by using Lemma \ref{lem:sop_algebra} and noting that it is the union
of $H_{i,-}^{a_i}$ and $H_{i,+}^{b_i}$ for each $1 \leq i \leq n$.
Then, if $\eps_k \searrow 0$ for which
\begin{equation}
\frac1{\eps_k} \mu \res (Q^{-\eps_k} \setminus \overline{Q}) \xrightharpoonup{\,\, * \,\,}
\tilde\lambda,
\end{equation}
it follows from Lemma \ref{lem:eta_intersection}(b) that $\tilde\lambda(\partial^2Q) = 0$.
Moreover, we argue that,  for each $A \Subset \Omega \setminus \partial^2Q$,
\begin{equation}
\tilde\lambda \res  A
      = \sum_{j=1}^n \Big( \mu_{\partial H_{j,-}^{a_j}}^{n-1} \res (A \cap \partial Q)
      +  \mu_{\partial H_{j,+}^{b_j}}^{n-1} \res (A \cap \partial Q) \Big),
\end{equation}
from which we infer that $\tilde\lambda = \mu^{n-1}_Q$,
since $\mu_{\partial H_{j,-}^{a_j}}^{n-1} = \mu_{\partial H_{j,+}^{a_j}}^{n-1}$ and $\mu_{\partial H_{j,+}^{b_j}}^{n-1} = \mu_{\partial H_{j,-}^{b_j}}^{n-1}$ for each $1 \leq j \leq n$.
\end{proof}

\begin{lemma}\label{lem:cube_localisation}
Let $Q = Q(a,b) \Subset \Omega$ be a cube such that $a_i,b_i \notin \mathcal N_{i,\mu}$ for each $1 \leq i \leq n$ and that \eqref{eq:corner_small} holds.
Then, for any Borel set $S \subset \partial Q$,
\begin{align}
&\mathcal F_{Q}(S)
= \sum_{i=1}^n \Big( \mathcal F_{H_{i,+}^{a_i}}(S \cap \partial H_{i,+}^{a_i})
+ \mathcal F_{H_{i,-}^{b_i}}(S \cap \partial H_{i,-}^{b_i}) \Big),\label{eq:cube_faces}\\
&\mathcal F_{\,\overline Q^{\mathrm{c}}}(S)
 = \sum_{i=1}^n \Big( \mathcal F_{H_{i,-}^{a_i}}(S \cap \partial H_{i,-}^{a_i})
 + \mathcal F_{H_{i,+}^{b_i}}(S \cap \partial H_{i,+}^{b_i})\Big).\label{eq:cube_faces_complement}
\end{align}
\end{lemma}

\begin{proof}
By Lemma \ref{lem:flux_extension},
the flux is defined on the hyperplanes associated with each face of $Q$.
Also, by Lemma \ref{eq:cube_eta_bound} and condition \ref{item:eta_bound},
$Q, \overline Q^{\mathrm{c}} \in \widetilde{\mathcal O}_{\mu}$ so that
$\mathcal F_Q$ and $\mathcal F_{\,\overline Q^{\mathrm{c}}}$ are measures
majorized by $\mu_Q^{n-1}$ (by using Remark \ref{rem:flux_complement} for the complement).

Since $\Omega \setminus \partial^2Q$ is open,
for any open set $A \subset \Omega \setminus \partial^2Q$,
the additivity property \ref{item:additivity} from Definition \ref{DefinitionCauchyFlux} gives
\begin{equation}
  \begin{split}
    \mathcal F_{Q}(A \cap \partial Q)
    &=  \sum_{i=1}^n \Big( \mathcal F_{ Q}(A \cap \partial Q\cap \partial H_{i,+}^{a_i})
      + \mathcal F_{Q}(A \cap \partial Q\cap \partial H_{i,-}^{b_i}) \Big),
    \end{split}
 \end{equation}
by noting that the collection $\{(\partial H_{i,+}^{a_i}, \partial H_{i,-}^{b_i}) : 1 \leq i \leq n\}$
is pairwise disjoint in $\mathbb R^n \setminus \partial^2Q$.

We now claim that,  for each $1 \leq i \leq n$,
  \begin{align}
      \mathcal F_{Q}(A \cap \partial Q \cap \partial H_{i,+}^{a_i}) &= \mathcal F_{ H_{i,+}^{a_i}}(A \cap \partial Q\cap \partial H_{i,+}^{a_i}), \\
      \mathcal F_{Q}(A \cap \partial Q\cap \partial H_{i,-}^{b_i}) &= \mathcal F_{H_{i,-}^{b_i}}(A \cap \partial Q\cap \partial H_{i,-}^{b_i}).
  \end{align}
Indeed, when $i=1$, writing $Q = (a_1,b_1) \times Q'$ with $Q' \subset \mathbb R^{n-1}$,
for $\delta>0$ sufficiently small,
we see that $Q_{1,\delta} = (a_1-\delta,a_1+\delta) \times Q'$ is an open cube
such that
\begin{equation}
A \cap Q_{1,\delta} \cap Q = A \cap Q_{1,\delta} \cap  H_{1,+}^{a_1}.
\end{equation}
Thus, applying property \ref{item:localisation} with the open set $A \cap Q_{1,\delta}$, we have
\begin{equation}
\begin{split}
\mathcal F_{Q}(A \cap \partial Q \cap \partial H_{1,+}^{a_1})
&= \mathcal F_{Q}(A \cap Q_{1,\delta} \cap \partial Q) \\
&= \mathcal F_{ H_{1,+}^{a_1}}(A \cap Q_{1,\delta} \cap \partial H_{1,+}^{a_1})
= \mathcal F_{ H_{1,+}^{a_1}}(A \cap \partial Q \cap \partial H_{1,+}^{a_1}).
\end{split}
\end{equation}
Arguing similarly for the remaining terms, the claim follows.

Therefore, we have
\begin{equation}
\mathcal F_{Q}(A \cap \partial Q)
 =  \sum_{i=1}^n \Big( \mathcal F_{ H_{i,+}^{a_i}}(A \cap \partial Q\cap \partial H_{i,+}^{a_i})
 + \mathcal F_{H_{i,-}^{b_i}}(A\cap \partial Q \cap \partial H_{i,-}^{b_i}) \Big)
\end{equation}
for any open set $A \subset \Omega \setminus \partial^2Q$.

By Lemmas \ref{lem:flux_measure} and \ref{lem:flux_extension}, we know that
both sides of \eqref{eq:cube_faces} are Radon measures on $\partial Q$,
and the equality holds for any relatively open set $S \subset \partial Q \setminus \partial^2Q$,
so that this extends to all Borel subsets $S \subset \partial Q \setminus \partial^2Q$ by Lemma \ref{lem:measure_coincidence}.
Moreover, since both sides of \eqref{eq:cube_faces} vanish on $\partial^2Q$
by property \ref{item:eta_bound}, \eqref{eq:corner_small}, and Lemma \ref{eq:cube_eta_bound},
this extends to hold for any Borel set $S \subset \partial Q$.
The argument for the flux on the complement, namely \eqref{eq:cube_faces_complement}, is analogous.
\end{proof}

The previous lemma relies crucially on condition \eqref{eq:corner_small},
which ensures that the flux does not concentrate on the corners of the cube.
The following result asserts that this condition holds for {\it almost all} cubes
in a suitable sense, which is used frequently in the sequel.
We postpone the technical proof until the end of this section,
as it is disjoint from the discussion which follows.

\begin{lemma}\label{lem:good_cubes}
Given a cube $Q=Q(a,b) \Subset \Omega$,
for $\mathcal L^n$--\textit{a.e.}\,\,$x \in \mathbb R^n$,
\eqref{eq:corner_small} holds for $x+Q${\rm :}
\begin{equation}
\mu_{H_{i,+}^{x_i+a_i}}^{n-1}(\partial^2(x+Q)) = \mu_{H_{i,-}^{x_i+b_i}}^{n-1}(\partial^2(x+Q)) = 0
      \qquad\text{for any } 1 \leq i \leq n.
\end{equation}
\end{lemma}

\begin{lemma}\label{lem:flux_cube_continuity}
Consider a cube $Q = Q(a,b) \Subset \Omega$ with $a_i,b_i \notin \mathcal N_{i,\mu}$ which satisfies \eqref{eq:corner_small} for each $1 \leq i \leq n$.
Then, for each $1 \leq j \leq n$, there exists an $\mathcal L^1$-null set $\mathcal N_{j,Q} \subset (a_j,b_j)$ such that
\begin{equation}\label{eq:Q_flux_map}
t \mapsto \mathcal F_{H_{j,+}^t}(Q \cap \partial H_{j,+}^t)
\end{equation}
is continuous when restricted to $\mathbb R \setminus \mathcal N_{j,Q}$
and hence is $\mathcal L^1$-measurable on $\mathbb R$, and
\begin{equation}\label{eq:hyperplane_cube_twosided}
\mathcal F_{H_{j,+}^t}(Q \cap \partial H_{j,+}^t)
= - \mathcal F_{H_{j,-}^{t}}(Q \cap \partial H_{j,-}^{t}) \qquad\text{for any }t \in \mathbb R \setminus \mathcal N_{j,Q}.
\end{equation}
\end{lemma}

\begin{proof} We divide the proof into two steps.

\smallskip
{\bf 1}. By permuting the coordinates, we can assume that $j = 1$.
We write $Q = (a_1,b_1) \times Q'$ and set $Q_{s,t} := (s,t) \times Q' \subset Q$ for $a_1 < s < t < b_1$.
This satisfies
  \begin{equation}
      \partial^2 Q_{s,t} \subset \{s,t\} \times \partial Q' \cup \partial^2Q,
  \end{equation}
viewing $\partial Q'$ as the boundary of $Q' \subset \mathbb R^{n-1}$.
Furthermore, notice that
\begin{equation}
\partial^2 Q_{s,t} \cap (\partial H_{i,+}^{a_i} \cup \partial H_{i,-}^{b_i})
\subset \partial Q' \cup \partial^2Q \qquad\text{for any } 2 \leq i \leq n.
\end{equation}
combining this with \eqref{eq:corner_small}, we obtain
  \begin{equation}\label{eq:Qst_corner_small}
 \mu_{H_{i,+}^{a_i}}^{n-1}(\partial^2 Q_{s,t}) = \mu_{H_{i,-}^{b_i}}^{n-1}(\partial^2Q_{s,t}) = 0
 \qquad\,\,\text{for any } 2\leq i \leq n.
 \end{equation}
 Observe that, since $a_i,b_i \not\in \mathcal N_{i,\mu}$,
 we have
 \begin{equation}
 \mu((a_1,b_1) \times \partial Q') = 0
 \end{equation}
 owing to $\mu(\partial H_{i,+}^{a_i}) = \mu(\partial H_{i,-}^{b_i}) = 0$ for all $2 \leq i \leq n$
 by Definition \ref{defn:good_hyperplanes}.
 Thus, for $\mathcal L^1$--\textit{a.e.}\,\, $t \in (a_1,b_1)$,
   \begin{equation}
    \mu_{H_{i,+}^{t}}^{n-1}(((a_1,b_1) \times \partial Q') \cap \partial H_{i,+}^{t})
    = \mu_{H_{i,-}^{t}}^{n-1}(((a_1,b_1) \times \partial Q') \cap \partial H_{i,-}^{t}) = 0.
  \end{equation}
 We denote by $\mathcal T$ the set of points $t \in (a_1,b_1) \setminus (\mathcal N_{1,\mu} \cup \mathcal M_{1,\sigma})$
 for which the above holds.
 Then, for $s,t \in \mathcal T$ with $s<t$, we have
  \begin{equation}
    \mu_{H_{1,+}^{s}}^{n-1}(\partial^2Q_{s,t}) = \mu_{H_{1,-}^{t}}^{n-1}(\partial^2Q_{s,t}) = 0.
  \end{equation}
Then, combined with \eqref{eq:Qst_corner_small},
Lemmas \ref{eq:cube_eta_bound}--\ref{lem:cube_localisation} apply to $Q_{s,t}$.

\smallskip
{\bf 2}. We now show the result holds with $\mathcal N_{1,Q} = (a_1,b_1) \setminus \mathcal T$.

Let $\eps_k,\delta_k>0$ such that $\eps_k,\delta_k \searrow 0$ and $t - \eps_k, t + \delta_k \in \mathcal T$ for each $k$.
Then, applying Lemma \ref{lem:cube_localisation} and the balance law \eqref{balancelaw2} to cube $Q_{t-\eps_k,t}$,
we deduce
\begin{equation}
    \begin{split}
      \sigma(Q_{t-\eps_k,t})
      &= \mathcal F_{Q_{t-\eps_k,t}}(\partial Q_{t-\eps_k,t}) \\
      &= \mathcal F_{H_{1,+}^{t-\eps_k}}(Q \cap \partial H_{1,+}^{t-\eps_k}) + \mathcal F_{H_{1,-}^{t}}(Q \cap \partial H_{1,-}^{t}) \\
      &\quad + \sum_{i=2}^n \Big( \mathcal F_{H_{i,+}^{a_i}}(Q_{t-\eps_k,t} \cap \partial H_{i,+}^{a_i})
         + \mathcal F_{H_{i,-}^{b_i}}(\partial Q_{t-\eps_k,t} \cap \partial H_{i,-}^{b_i})\Big)  \\
      &=: I_{1,+}^k + I_{1,-} + \sum_{i=2}^n \big( I_{i,+}^k + I_{i,-}^k \big).
    \end{split}
\end{equation}
By Lemma \ref{lem:flux_measure}, the fluxes on the respective hyperplanes are measures.
It follows that $I_{i,\pm}^k \to 0$ as $k \to \infty$ for each $2 \leq i \leq n$.
Similarly, $\sigma(Q_{t-\eps_k,t}) \to 0$.
Then we see that
  \begin{equation}
    \lim_{k \to \infty} \mathcal F_{H_{1,+}^{t-\eps_k}}(Q \cap \partial H_{1,+}^{t-\eps_k})
    = \lim_{k \to \infty} I_{1,+}^k = - I_{1,-} =- \mathcal F_{H_{1,-}^{t}}(Q \cap \partial H_{1,-}^{t}).
  \end{equation}
Analogously, for $Q_{t,t+\delta_k}$, it follows that
  \begin{equation}
    \lim_{k \to \infty} \mathcal F_{H_{1,-}^{t+\delta_k}}(Q \cap \partial H_{1,-}^{t+\delta_k})=- \mathcal F_{H_{1,+}^{t}}(Q \cap \partial H_{1,+}^{t}),
\end{equation}
and, for $Q_{t-\eps_k,t+\delta_k}$, we apply that $\sigma(\partial H_{1,+})=0$ to obtain
  \begin{equation}
    \lim_{k \to \infty} \mathcal F_{H_{1,-}^{t+\delta_k}}(Q \cap \partial H_{1,-}^{t+\delta_k})
    + \lim_{k \to \infty} \mathcal F_{H_{1,+}^{t-\eps_k}}(Q \cap \partial H_{1,+}^{t-\eps_k}) = 0.
  \end{equation}
Thus, it follows that
  \begin{equation}
    \mathcal F_{H_{1,+}^{t}}(Q \cap \partial H_{1,+}^{t}) = - \mathcal F_{H_{1,-}^{t}}(Q \cap \partial H_{1,-}^{t})
  \end{equation}
establishing \eqref{eq:hyperplane_cube_twosided}, which implies
  \begin{align}
    \lim_{k \to \infty} \mathcal F_{H_{1,-}^{t+\delta_k}}(Q \cap \partial H_{1,-}^{t+\delta_k})
    & = \mathcal F_{H_{1,-}^{t}}(Q \cap \partial H_{1,-}^{t}), \\
    \lim_{k \to \infty} \mathcal F_{H_{1,+}^{t-\eps_k}}(Q \cap \partial H_{1,+}^{t-\eps_k})
     &= \mathcal F_{H_{1,+}^{t}}(Q \cap \partial H_{1,+}^{t}),
  \end{align}
leading to the continuity of \eqref{eq:Q_flux_map} restricted to $\mathbb R \setminus \mathcal N_{j,Q}$.
Finally, the $\mathcal L^1$--measurability follows from a general principle;
indeed, denoting this mapping by $g$, if $A \subset \mathbb R$ is open, then
\begin{equation}
      g^{-1}(A) = \big(g\rvert_{\mathbb R \setminus \mathcal N_{j,Q}}\big)^{-1}(A)
                   \cup \big(g^{-1}(A) \cap \mathcal N_{j,Q}\big).
\end{equation}
This is the union of a relatively open subset of the $\mathcal L^1$-measurable set $\mathbb R \setminus N_{j,Q}$
and an $\mathcal L^1$-null set, so both are $\mathcal L^1$-measurable by using the completeness of the Lebesgue measure.
Therefore, $g$ is $\mathcal L^1$-measurable, as required.
\end{proof}

Combining Lemma \ref{lem:good_cubes} with measure-theoretic argument allows us to extend
Lemma \ref{lem:flux_cube_continuity} to hold for general Borel subsets.
This provides a key step in the construction of an associated field $\FF$
in \S \ref{sec:flux_field_construct}.

\begin{lemma}\label{lem:hyperplane_integrable}
  For each $1 \leq j \leq n$, there is a null set $\mathcal K_j \subset \mathbb R$ containing $\mathcal N_{j,\mu}\cup \mathcal M_{j,\sigma}$ such that, for $t \notin \mathcal K_j$,
   \begin{equation}\label{eq:hyperplane_twosided}
    \mathcal F_{H_{j,+}^t}(S \cap \partial H_{j,+}^t) = -\mathcal F_{H_{j,-}^{t}}(S \cap \partial H_{j,-}^{t})
   \qquad\text{for any Borel set $S \subset \Omega$}.
    \end{equation}
Moreover, for any such $S$, the mapping{\rm :}
  \begin{equation}\label{eq:hyperplane_flux_map}
    t \mapsto \begin{cases}
      \mathcal F_{H_{j,+}^t}(S \cap \partial H_{j,+}^t) &\text{if } t \notin \mathcal K_j, \\
      0 &\text{if } t \in \mathcal K_j,
    \end{cases}
  \end{equation}
  is $\mathcal L^1$-integrable satisfying
  \begin{equation}\label{eq:hyperplane_integral_bound}
    \int_{\mathbb R} \lvert \mathcal F_{H_{j,+}^t}(S \cap \partial H_{j,+}^t)\rvert \,\d t \leq \mu(S).
  \end{equation}
\end{lemma}

\begin{proof}
Consider the collection $\{ Q(a,b) \,:\, a,b \in \mathbb Q^n\}$ of cubes with rational endpoints
in $\mathbb R^n$.
Then, for $\mathcal L^n$--\textit{a.e.}\,\,$x \in \mathbb R^n$,
the conclusion of Lemma \ref{lem:good_cubes} is satisfied for all $x+Q(a,b)$ with $a,b$ rational,
and $x_i+a_i, x_i+b_i \notin \mathcal N_{i,\mu} \cup \mathcal M_{i,\sigma}$ for all $1 \leq i \leq n$.
We then let $\mathcal Q_{\Omega}$ be the collection of such cubes $x+Q \Subset \Omega$ that are compactly contained in $\Omega$.
Observe that $\mathcal Q_{\Omega}$ is a countable collection of open cubes generating the Borel $\sigma$-algebra of $\Omega$,
which is also closed under finite intersections.

Writing $\mathcal Q_{\Omega} = \{Q_k\}_{k \in \mathbb N}$ and applying Lemma \ref{lem:flux_cube_continuity} to each $Q_k$,
we obtain an associated null set $\mathcal N_{j,Q_k}$ for each $1 \leq j \leq n$, which
contains $\mathcal N_{j,\mu} \cup \mathcal M_{j,\sigma}$ by construction.
We set $\mathcal K_j := \bigcup_k \mathcal N_{j,Q_k}$, which is a null set for each $j$.
Then, for all $t \notin \mathcal K_j$ and $k \in \mathbb N$,
\eqref{eq:hyperplane_twosided} holds with $S = Q_k$,
and the mapping in \eqref{eq:hyperplane_flux_map} with $S = Q_k$ is $\mathcal L^1$-measurable.

Now, for any $t \notin \mathcal K_j$, Lemma \ref{lem:flux_extension} ensures that
$\mathcal F_{H_{j,\pm}^t}(\,\cdot \cap \partial H_{j,\pm}^t)$ are measures on $\Omega$,
so we can extend \eqref{eq:hyperplane_twosided}
that holds for any $S = Q_k$ to all Borel sets $S \subset \Omega$ by using Lemma \ref{lem:measure_coincidence}.
For the second part, let $\mathcal D_j$ denote the set of Borel subsets $S \subset \Omega$ for which the result holds;
that is, \eqref{eq:hyperplane_flux_map} is $\mathcal L^1$-integrable satisfying \eqref{eq:hyperplane_integral_bound}.

We now show that $\mathcal D_j$ contains all Borel subsets of $\Omega$,
by a similar argument to that in \cite[Remark 1.9]{afp}, thereby establishing the result.
For this, observe first that, if $S \subset \Omega$ is a Borel set,
then, by property \ref{item:eta_bound} and Lemma \ref{lem:disint_flux},
  \begin{equation}\label{eq:cube_pointwise_est}
    \lvert \mathcal F_{H_{j,+}^t}(S \cap \partial H_{j,+}^t) \rvert\leq \mu_t(S \cap \partial H_{j,+}^t)
    \qquad\mbox{ for $\mathcal L^1$--\textit{a.e.}\,\,$t$},
  \end{equation}
where $\mu_t$ is given by the disintegration
\begin{equation}
\mu = \mathcal L^1 \res [0,\infty)\otimes_{\partial H_{j,+}^t} \mu_t + \tau_{\mathrm{sing}} \otimes_{\partial H_{j,+}^t} \tilde\mu_t.
\end{equation}
Hence, if mapping \eqref{eq:hyperplane_flux_map} is measurable,
we can integrate in $t$ to deduce \eqref{eq:hyperplane_integral_bound} and thus
infer that $S \in \mathcal D_j$.
In particular, $\mathcal Q_{\Omega} \subset \mathcal D_j$, since the associated mapping is measurable.

We now claim that $\mathcal D_j$ is closed under the set differences and countable increasing unions.
Indeed, if $A, B \in \mathcal D_j$ with $A \subset B$, then, for $t \notin \mathcal K_j$, we can write
\begin{equation}
\mathcal F_{H^t_{j,+}}((B \setminus A) \cap \partial H_{j,+}^t) = \mathcal F_{H^t_{j,+}}(B \cap \partial H_{j,+}^t)
- \mathcal F_{H^t_{j,+}}(A \cap \partial H_{j,+}^t),
\end{equation}
which is evidently measurable in $t$, so that $B \setminus A \in \mathcal D_j$.
Also, if $\{S_j\} \subset \mathcal D_j$ is any countable collection of pairwise disjoint Borel subsets in $\Omega$,
then, setting $S = \bigcup_j S_j$, we have
\begin{equation}
\sum_{j}\mathcal F_{H^t_{j,+}}(S_j \cap \partial H_{j,+}^t) = \mathcal F_{H^t_{j,+}}(S \cap \partial H_{j,+}^t)
\qquad\mbox{ for any $t \notin \mathcal K_j$},
\end{equation}
so this is also measurable in $t$.
Now, if $\{A_j\}\subset \mathcal D_j$ is a countable collection of Borel subsets which is not necessarily disjoint,
define $\{S_k\}$ by setting $S_1 = A_1$ and $S_k = A_k \setminus \big( \bigcup_{j=1}^{k-1} S_j \big)$ for each $k \geq 2$.
Then $\{S_k\}$ are pairwise disjoint such that $\bigcup_k S_k = \bigcup_k A_k =:A$ and, by induction,
each $S_k \in \mathcal D_j$, so that $A \in \mathcal D_j$.
Using this, we also obtain $\Omega \in \mathcal D_j$, so $\mathcal D_j$ is a $\sigma$-algebra
containing $\mathcal Q_{\Omega}$ and hence contains all Borel subsets $S \subset \Omega$, thereby establishing the result.
\end{proof}

\begin{lemma}\label{cor:cube_localisation2}
Suppose that $Q = Q(a,b) \Subset \Omega$ such that $a_i,b_i \notin \mathcal K_{i}$ for all $1 \leq i \leq n$
and \eqref{eq:corner_small} holds.
Then
\begin{equation}\label{eq:cube_faces2}
\mathcal F_{Q}(S)
= \sum_{i=1}^n \Big( \mathcal F_{H_{i,+}^{a_i}}(S \cap \partial H_{i,+}^{a_i})
  - \mathcal F_{H_{i,+}^{b_i}}(S \cap \partial H_{i,+}^{b_i}) \Big)
  = - \mathcal F_{\overline Q^{\mathrm{c}}}(S)
\end{equation}
for any Borel set $S \subset \partial Q$.
In particular, for any cube $Q$, this holds for $x+Q$ for $\mathcal L^n$--\textit{a.e.}\,$\,x \in \mathbb R^n$.
\end{lemma}

\begin{proof}
The representations of $\mathcal F_Q(S)$ and $\mathcal F_{\overline Q^{\mathrm{c}}}(S)$ follow
by combining Lemma \ref{lem:cube_localisation} with \eqref{eq:hyperplane_twosided}
from Lemma \ref{lem:hyperplane_integrable} and recalling that
$\mathcal N_{j,\mu} \cup \mathcal M_{j,\sigma} \subset \mathcal K_j$ for each $1 \leq j \leq n$.
For the last part, we apply Lemma \ref{lem:good_cubes} to see that \eqref{eq:corner_small} is satisfied
for $\mathcal L^n$--\text{a.e.}\,\,$x \in \mathbb R^n$.
By increasing the null set if necessary, we also require that $x_j+a_j, x_j+b_j \notin \mathcal K_j$
for all $1 \leq j \leq n$, where $Q = Q(a,b)$.
\end{proof}

We now conclude this section with the proof of Lemma \ref{lem:good_cubes}.

\begin{proof}[Proof of Lemma {\ref{lem:good_cubes}}]
Extending $\mu$ by zero, we view $\mu$ as a measure on $\mathbb R^n$ in what follows:
Let $v, w \in \mathbb R^n$ be unit vectors with non-zero components such that,
for all $1 \leq i < j \leq n$, the vectors $(v_i,v_j)$ and $(w_i,w_j)$ are linearly independent.
For instance, we can define $\tilde v, \tilde w \in \mathbb R^n$
by setting $\tilde v_j = j$ and $\tilde w_j = j^2$ for each $1 \leq j \leq n$,
which can then be normalized by letting $v = \frac{\tilde v}{\lvert \tilde v \rvert}$
and $w = \frac{\tilde w}{\lvert \tilde w \rvert}$.
For this choice, we can readily verify that $v_iw_j \neq v_jw_i$ for all $i \neq j$, which implies linear independence.
Now we define
\begin{equation}\label{eq:corner_intersection}
\mathrm C:= \bigcup_{\beta \in \mathbb R} \big(\beta w +\partial^2Q\big).
\end{equation}

Then we divide the rest of the proof into three steps.

\smallskip
\textbf{1}. We first prove that, for
the
set $\mathrm C$ as defined in \eqref{eq:corner_intersection} above,
 \begin{equation}\label{eq:good_cube_key_claim}
        \mu(\alpha v + \mathrm{C}) = 0 \qquad\,\,\mbox{for $\mathcal L^1$--\textit{a.e.}\,\,$\alpha \in \mathbb R$}.
    \end{equation}

We establish this in two substeps, starting with the following algebraic result.

\textbf{1.1}. We show that there is $M \in \mathbb N$ sufficiently large such that, for any $\alpha_1,\alpha_2,\cdots,\alpha_M \in \mathbb R$ distinct,
\begin{equation}\label{eq:Mintersection_empty}
\bigcap_{k=1}^{M} ( \alpha_k v + \mathrm{C}) = \varnothing.
\end{equation}
In fact, we can take $M = 2n(n-1) + 1$.

To see this, observe that, for each $x \in \alpha v + \mathrm{C}$, there is
some $\beta \in \mathbb R$ such that $y := x - \alpha v - \beta w \in \partial^2Q$.
Then there exist $i$ and $j$ with $1 \leq i < j \leq n$ such that
$y_i \in \{ a_i, b_i\}$ and $y_j \in \{ a_j, b_j\}$.
We can associate $y$ as lying in the {\it corner} associated to $((i,j),(y_i,y_j))$, for which there are $2^2 \binom{n}{2} = 2n(n-1) = M-1$ distinct choices.

Now, for this choice of $M$, suppose that there are distinct
$\alpha_1,\cdots,\alpha_M \in \mathbb R$ for which the claim fails to hold, so that
there exists some $x \in \bigcap_{k=1}^M (\alpha_kv + \mathrm{C})$.
For each $k$, we can find $\beta_k$ such that
\begin{equation}
y_k := x - \alpha_{k} v - \beta_k w \in \partial^2Q.
\end{equation}
By the pigeonhole principle, there exists two points that lie in the same corner,
that is, there exist $k_1, k_2 \in \{1,2,\cdots,M\}$ with $k_1 \neq k_2$ and $1 \leq i < j \leq n$
such that $(y_1)_i = (y_2)_i \in\{a_i,b_i\}$ and $(y_1)_j = (y_2)_j \in\{a_j,b_j\}$.
We denote these common values by $c_i$ and $c_j$, respectively:
\begin{align}
(y_q)_i &= x_i - \alpha_{k_q}v_i - \beta_{k_q} w_i= c_i, \\
(y_q)_j &= x_j - \alpha_{k_q}v_j - \beta_{k_q} w_j = c_j,
\end{align}
for $q=1,2$.
Rearranging for $x_i-c_i$ and $x_j-c_j$, we obtain the following equations:
\begin{align}
\alpha_{k_1}v_i + \beta_{k_1} w_i &= \alpha_{k_2}v_i + \beta_{k_2} w_i,  \\
\alpha_{k_1}v_j + \beta_{k_1} w_j &= \alpha_{k_2}v_j + \beta_{k_2} w_j.
\end{align}
Since $v$ and $w$ have been chosen so that $(v_i,v_j)$ and $(w_i,w_j)$ are linearly independent, it follows that $\alpha_{k_1} = \alpha_{k_2}$ and $\beta_{k_1}=\beta_{k_2}$.
However, this contradicts the fact that $\alpha_k$ are distinct, which implies that  \eqref{eq:Mintersection_empty} holds as claimed.

\smallskip
\textbf{1.2}. For each $1 \leq k \leq M$, we will show there exists a countable set $\mathcal N_k$
with $\mathcal N_M \subset \mathcal N_{M-1} \subset \cdots \subset \mathcal N_1$
such that
  \begin{equation}\label{eq:kint_mu_null}
    \mu\big(\bigcap_{i=1}^k (\alpha_i v + \mathrm{C})\big) = 0
  \end{equation}
for all $\alpha_1, \cdots, \alpha_k \in \mathbb R \setminus \mathcal N_k$ distinct.
Note that claim \eqref{eq:good_cube_key_claim} is precisely the case $k=1$.

We show this by induction descending in $k$. First, we notice that this holds for $k = M$ with $\mathcal N_M = \varnothing$,
since the previous step ensures that the intersection is empty.

For general $k < M$, suppose that there is a $\mu$-null set $\mathcal N_{k+1}$ as claimed, and consider any finite collection $F$ of
subsets $\Lambda = \{\alpha_1,\cdots,\alpha_k\}$ containing $k$ distinct elements
taking values in $\mathbb R \setminus \mathcal N_{k+1}$.
Observe that, if $\Lambda_1,\Lambda_2 \in F$ are distinct,
then the union $\Lambda_1 \cup \Lambda_2$ contains at least $k+1$ elements and,  by assumption on $\mathcal N_{k+1}$,
\begin{equation}
\mu\big(\bigcap_{\alpha \in \Lambda_1 \cup \Lambda_2} (\alpha v + \mathrm{C})\big) = 0.
\end{equation}
Now the inclusion-exclusion principle gives
\begin{align}
&\mu\big(\bigcup_{\Lambda \in F} \bigcap_{\alpha \in \Lambda} (\alpha v + \mathrm{C})\big)
\nonumber\\
 &= \sum_{\Lambda \in F} \mu\big(\bigcap_{\alpha \in \Lambda} (\alpha v + \mathrm{C})\big)
 + \sum_{\substack{S \subset F \\ \lvert S \rvert > 1}} (-1)^{1+\lvert F \rvert}\sum_{\Lambda \in F} \mu\big(\bigcap_{\alpha \in \bigcup S} (\alpha v + \mathrm{C})\big).\label{eq:inc_exc}
\end{align}
We see that all terms in the second sum are zero, since $\bigcup S$ always contains at least $k+1$ distinct elements of $\mathbb R \setminus \mathcal N_{k+1}$.
Thus,  we deduce from \eqref{eq:inc_exc} that
\begin{equation}
\sum_{\Lambda \in F} \mu\big(\bigcap_{\alpha \in \Lambda}(\alpha v + \mathrm{C})  \big)
=\mu\big( \bigcup_{\Lambda \in F} \bigcap_{\alpha\in\Lambda}(\alpha_i^m v + \mathrm{C}) \big)
\leq \mu(\Omega) < \infty.
\end{equation}
Since $F$ is arbitrary, it follows that there are at most countably many
collections $\Lambda_m = \{\alpha_1^m,\cdots,\alpha_k^m\}$ of subsets of $k$ distinct elements in $\mathbb R \setminus \mathcal N_{k+1}$ such that
$$\mu\big(\bigcup_{\alpha \in \Lambda_m} (\alpha v + \mathrm{C})\big) \neq 0.
$$
Thus, we take $\mathcal N_k = \mathcal N_{k+1} \cup \bigcup_m  \Lambda_m$
which is $\mathcal L^1$-null, and observe that
any collection $\alpha_1, \dots, \alpha_k \in \mathbb R\setminus \mathcal N_k$
necessarily does not coincide with any of  $\Lambda_m$, which implies
that
$$
\mu\big(\sum_{i=1}^k (\alpha_iv+\mathrm{C})\big)=0.
$$
Therefore, the claim follows by induction.

\medskip
\textbf{2}: We now show that $(\alpha v + \beta w) +Q$ satisfies \eqref{eq:corner_small}
for $\mathcal L^2$--\emph{a.e.}\,\,$(\alpha,\beta) \in \mathbb R^2$,

\smallskip
Let $\alpha \in \mathbb R \setminus \mathcal N_1$.
For $\beta \in \mathbb R$ to be determined, let $x = \alpha v + \beta w$.
Then $x+Q$ has faces intersecting
with $\partial H^{\alpha v_i + \beta w_i + a_i}_{i,+}$
and $\partial H^{\alpha v_i + \beta w_i + b_i}_{i,-}$ for each $1 \leq i \leq n$.
Moreover, by Lemma \ref{lem:disint_flux}, we can estimate
  \begin{equation}
      \int_{\mathbb R} \mu^{n-1}_{H^{\alpha v_i + t + a_i}_{i,+}}
      \big( (\alpha v + \mathrm{C}) \cap \partial H^{\alpha v_i + t + a_i}_{i,+}\big) \,\d t \leq \mu(\alpha v + \mathrm{C}) = 0,
  \end{equation}
and similarly for $\partial H^{\alpha v_i + t + b_i}_{i,-}$. Thus considering $t = \beta w_i$ in the respective cases, by noting that $w_i \neq 0$,
 \begin{equation}
      \mu_{H^{\alpha v_i +\beta w_i + a_i}_{i,+}}^{n-1}(\alpha v+\mathrm{C}) =\mu_{H^{\alpha v_i + \beta w_i + b_i}_{i,-}}^{n-1}(\alpha v+\mathrm{C}) = 0
      \qquad\,\mbox{for $\mathcal L^1$--\textit{a.e.}\,\,$\beta\in\mathbb R$}.
  \end{equation}

For such $\beta$, we see that $x+Q$ with $x= \alpha v + \beta w$ satisfies \eqref{eq:corner_small}, provided that $x+Q \Subset \Omega$.
That is, for $\mathcal L^2$--\textit{a.e.}\,\,$(\alpha,\beta)\in \mathbb R^2$, the result holds for $x = \alpha v + \beta w$ .

\medskip
\textbf{3}: We now show that $x+Q$ satisfies \eqref{eq:corner_small}
for $\mathcal L^n$--\textit{a.e.}\,\,$x \in \mathbb R^n$.

For each $y \in P := (\operatorname{span}\{v,w\})^{\perp}$, applying the above to $y+Q$,
we obtain an $\mathcal L^2$-null set $\mathcal M_y \subset \mathbb R^2$
such that the result holds for $(\alpha v + \beta w) + y + Q$,
whenever $(\alpha,\beta) \notin \mathcal M_y$.
Then,
for $\widetilde{\mathcal M}_y :=\{ y+ \alpha v + \beta w
\,:\, (\alpha,\beta) \in \mathcal M_y\}$,
$\mathcal H^2(\widetilde{\mathcal M}_y)=0$ by the area formula (\cite[Theorem 2.71]{afp}).
Then,
setting $\mathcal M = \bigcup_{y \in P} \widetilde M_y$,
by the Fubini Theorem,
we have
  \begin{equation}
      \mathcal L^n(\mathcal M) = \int_{P} \mathcal H^2(\widetilde{\mathcal M}_y) \,\d \mathcal H^{n-2}(y) = 0.
  \end{equation}
Therefore, \eqref{eq:corner_small} holds for all $x \notin \mathcal M$,
provided that $x+Q \Subset \Omega$.
\end{proof}

\smallskip
\section{Cauchy Flux III: Construction and Uniqueness of the Representing Divergence-Measure Field}
\label{sec:flux_existence}

We now employ the properties of the flux established in \S 7 to work towards the proof of Theorem \ref{maintheorem}.
In particular, we define a candidate vector field $\FF$ to represent the flux, recover the balance law, and show that any such field must take this form.

\subsection{Construction of a divergence-measure field}\label{sec:flux_field_construct}

\begin{theorem}\label{thm:flux_construct}
    Given a Cauchy flux, there exists a measure-valued vector field
    \begin{equation}
        \FF = (F_1,F_2,\cdots,F_n) \in \mathcal{M}(\Omega,\mathbb R^n)
    \end{equation}
    such that $\lvert F_j \rvert \leq \mu$ for each $1 \leq j \leq n$ and
        \begin{equation}\label{eq:flux_general_integral}
        F_j(S) = \int_{\mathbb R} \mathcal{F}_{H_{j,+}^t}(S \cap \partial H_{j,+}^t) \,\d t
   \qquad\,\,\mbox{for any Borel set $S \subset \Omega$}.
   \end{equation}
   \end{theorem}

\begin{proof}
By Lemma \ref{lem:hyperplane_integrable}, we know that integral \eqref{eq:flux_general_integral} is well-defined for any Borel
set $S \subset \Omega$ (understanding that it is zero on the null set $\mathcal K_j$) and satisfies the estimate:
\begin{equation}\label{eq:Fj_eta_bound}
\lvert F_j(S) \rvert \leq \mu(S).
\end{equation}
From this, it also follows that $F_j(\varnothing) = 0$.
By the additivity property \ref{item:additivity} of the flux and
the linearity of the integral, we know that $F_j$ is additive on Borel sets.
To show the countable additivity, let $\{S_k\}$ be a collection of pairwise disjoint Borel subsets of $\Omega$, and let $S = \bigcup_k S_k$.
Then, by finite additivity and \eqref{eq:Fj_eta_bound}, we have
  \begin{equation}
      \Big\lvert F_j(S) - \sum_{k=1}^M F_j(S_k) \Big\vert
      = \Big\lvert F_j\big(S \setminus \bigcup_{k=1}^M S_k \big)\Big\vert
      \leq \mu\big(S \setminus \bigcup_{k=1}^M S_k \big) \to 0
\qquad\mbox{as $M \to \infty$}.
\end{equation}
This shows that each $F_j$ is a signed measure on $\Omega$, as required.
\end{proof}

\begin{theorem}[Recovery of the balance law]\label{thm:flux_globalrecovery}
The field $\FF$ constructed in {\rm Theorem \ref{thm:flux_construct}} lies in $\mathcal{DM}^{\ext}(\Omega)$ and satisfies
    \begin{equation}
        -\div \FF = \sigma \qquad\text{in } \Omega.
    \end{equation}
In particular, for any $U \Subset \Omega$,
\begin{equation}\label{eq:global_recovery}
\mathcal F_U(\partial U) = \sigma(U)
= \langle \FF \cdot \nu, \,\mathbbm{1}_{\Omega} \rangle_{\partial U}.
\end{equation}
\end{theorem}

\begin{proof}
To simplify notation, we write $H_{j,+}^t = H_j^t$.
Let $\rho_{\delta}$ be a standard mollifier,
and define $\FF_{\delta}:= \FF \ast \rho_{\delta}$.
Then, for any cube $Q = Q(a,b) \Subset \Omega$ and
$0 < \delta < \dist(Q,\partial\Omega)$, since $\FF_{\delta}$ is smooth
on $\overline Q$, the classical Divergence Theorem gives
\begin{equation}
\begin{split}
 \int_{Q} \div \FF_{\delta} (x) \,\d x
 &= -\int_{\partial Q} \FF_{\delta}(x) \cdot \nu_{\partial Q} (x)\,\d \mathcal{H}^{n-1} (x)\\
 &= - \int_{\partial Q} \int_{\mathbb{R}^n} \rho_{\delta} (x-y)\,\nu_{\partial Q} (x) \cdot \d \FF (y)\,\d \mathcal{H}^{n-1} (x),
\end{split}
\end{equation}
where $\nu_{\partial Q}$ is the inwards facing normal for $\partial Q$.

We consider the face: $\partial Q \cap \partial H_1^{a_1}$
on which $\nu_{\partial Q}= e_1$.
Then, by construction of $F_j$ from Theorem \ref{thm:flux_construct}, we know that the disintegration
\begin{equation}
  F_j = \mathcal L^1 \otimes_{\partial H_{j}^t} F_{j,t} = \mathcal L^1 \otimes_{\partial H_{j}^t} \mathcal F_{H_{j}^t}(\,\cdot \cap \partial H_{j}^t)
\end{equation}
holds.
Using this and writing $Q = (a_1,b_1) \times Q'$, we can write
\begin{equation}
\begin{split}
&\int_{\partial Q \cap \partial H_1^{a_1}}
\int_{\mathbb{R}^n} \rho_{\delta} (x-y)\,\nu_{\partial Q}(x)\cdot \d \FF (y)\,\d \mathcal{H}^{n-1}(x)\\
&= \int_{\mathbb R^{n-1}} \int_{\mathbb R} \int_{\mathbb R^{n-1}} \chi_{Q'}(x')\rho_{\delta}(a_1-y_1,x'-y')\, \d F_{1,y_1}(y) \,\d y_1\,\d x' \\
&= \int_{\mathbb R^{n-1}} \int_{\mathbb R} \int_{\mathbb R^{n-1}} \chi_{Q'-z'}(y')\rho_{\delta}(z_1,z')\, \d F_{1,a_1-z_1}(a_1-z_1,y') \,\d z_1\,\d z' \\
&= \int_{\mathbb R^n} \rho_{\delta}(z) F_{1,a_1-z_1}(\{a_1 - z_1\} \times (Q'-z'))\,\d z \\
&= \int_{\mathbb R^n}  \rho_{\delta}(z) \mathcal F_{-z + H_1^{a_1}}(-z+(\partial Q \cap \partial H_1^{a_1})) \,\d z,
\end{split}
\end{equation}
where we have made the change of variables $y_1 \mapsto z_1 = a_1-y_1$
and $x' \mapsto z' = x'-y'$.
The remaining sides can be computed similarly.
Writing $Q_z = z+Q$ that has faces $H_{i}^{a_i-z_i}$ and $H_i^{b_i-z_i}$,
we deduce
\begin{equation}
  \begin{split}
  \int_{Q} \div \FF_{\delta} (x) \,\d x
  &= \sum_{i=1}^n \int_{\mathbb R^n} \rho_{\delta}(z) \mathcal F_{H_j^{a_j-z_j}}(\partial Q_{-z} \cap \partial H_j^{a_j-z_j}) \,\d z \\
  &\quad- \sum_{i=1}^n \int_{\mathbb R^n} \rho_{\delta}(z) \mathcal F_{H_j^{b_j-z_j}}(\partial Q_{-z} \cap \partial H_j^{b_j-z_j})\,\d z.
  \end{split}
\end{equation}
Now, by Lemmas \ref{cor:cube_localisation2} and \ref{lem:good_cubes},
for $\mathcal L^n$--\textit{a.e.}\,\,$z$ with $\lvert z\rvert \leq \dist(Q,\partial\Omega)$,
we have
\begin{equation}
\begin{split}
&\sum_{i=1}^n \int_{\mathbb R^n} \Big(\mathcal F_{H_j^{a_j-z_j}}(\partial Q_{-z} \cap \partial H_j^{a_j-z_j}) - \mathcal F_{H_j^{b_j-z_j}}(\partial Q_{-z} \cap \partial H_j^{b_j-z_j}) \Big)\\
& = \mathcal F_{Q_{-z}}(\partial Q_{-z}) = \sigma(Q_{-z}).
\end{split}
\end{equation}
Thus, we obtain
\begin{equation}
  \int_Q \div \FF_{\delta}(z) \,\d x = -\int_{\mathbb R^n} \rho_{\delta}(z) \sigma(Q+z) \,\d z = -\int_Q \sigma \ast \tilde\rho_{\delta}(x) \,\d x,
\end{equation}
where $\tilde\rho_{\delta}(x) = \rho_{\delta}(-x)$ is also a mollifier.
Since this holds for all cubes $Q \Subset \Omega$ whenever $0<\delta < \dist(Q,\partial\Omega)$, it follows that
\begin{equation}
  -\div \FF \ast \rho_{\delta} = \sigma \ast \tilde\rho_{\delta} \qquad\text{in } \Omega^{\delta}.
\end{equation}
Sending $\delta \to 0$, we deduce that
$\FF \in \mathcal{DM}^{\ext}(\Omega)$ with $-\div \FF = \sigma$.
Finally, it follows from the balance law \eqref{balancelaw2} that,
for any $U \Subset \Omega$,
\begin{equation}
  \mathcal F_U(\partial U) = \sigma(U) = -\int_U \d(\div \FF),
\end{equation}
and, by \eqref{eq:trace_balance_law},
this is equal to $\langle \FF \cdot \nu, \,\mathbbm{1}_{\Omega}\rangle_{\partial U}$, thereby establishing \eqref{eq:global_recovery}.
\end{proof}

\subsection{Uniqueness of the field}\label{sec:flux_field_unique}

We now show that the constructed field $\FF$ is unique.
Note that, in Theorem \ref{thm:flux_construct}, the divergence-measure field $\FF$
has been constructed componentwise by integrating the fluxes
on $\partial H_{j,+}^t$ with respect to the Lebesgue measure.
However, this choice of the measure is seemingly arbitrary,
and one may wonder if there is a {\it singular part} that is missed with this construction.
This possibility is essentially ruled out by
Theorem \ref{thm:reconstruct_normaltrace}; however,
this cannot be directly applied with $U = H_{j,+}^t$ since $U \not\Subset \Omega$.
Nevertheless, using the localization results from \S \ref{sec:localisation}
and \S \ref{sec:flux_property_general}, we can adapt the result as follows:

\begin{lemma}\label{lem:hyperplane_reconstruct}
Let $\FF = (F_1,F_2,\cdots,F_n) \in \mathcal{DM}^{\ext}(\Omega)$.
Consider the disintegration{\rm :}
\begin{equation}
F_j = \tau \otimes_{\partial H_{j,+}^t} \tilde F_{j,t}
\qquad \mbox{ for $1 \leq j \leq n$},
\end{equation}
where $H_{j,+}^t = \{ x \in \mathbb R^n : x_j > t\}$, extending $F_j$
by zero to a measure on $\mathbb R^n$.
Then $\tau \ll \mathcal L^1$ and, for $\mathcal L^1$--\textit{a.e.}\,\,$t \in \mathbb R$,  $F_{j,t} := D_{\mathcal L^1}\tau(t) \tilde F_{j,t}$ is well-defined and satisfies
\begin{equation}
F_{j,t} \res \widetilde\Omega^{\delta}= (\FF \cdot \nu)_{\partial (\widetilde\Omega^{\delta} \cap H_{j,+}^t)} \res  \widetilde\Omega^{\delta} \qquad\text{for } \mathcal L^1\textit{--a.e.\,\,} \delta>0
\end{equation}
as measures. In particular,
$F_j = \mathcal L^1 \otimes_{\partial H_{j,+}^t} F_{j,t}.$
\end{lemma}

\begin{proof}
By Theorem \ref{prop:flux_sufficiency}, the divergence-measure field
$\FF$ defines a Cauchy flux $\widetilde{\mathcal F}$ via its normal trace
with $\mu = \lvert \FF \rvert$.
Fix $t \in \mathbb R$ be such that
\begin{equation}
\frac1{\eps} \lvert \FF \rvert \res H_{j,+}^t \setminus \overline{H_{j,+}^{t+\eps}}
\xrightharpoonup{\,\, * \,\,}
\mu^{n-1}_t \qquad\text{ as } \eps\to0
\end{equation}
to some limiting measure $\mu_t^{n-1}$, which is satisfied
for $\mathcal L^1$--\textit{a.e.}\,\,$t\,$ by Lemma \ref{lem:disint_flux} applied
with $\mu = \lvert \FF \rvert$.
Then, by Lemma \ref{lem:flux_extension}, there is a well-defined measure
$\widetilde{\mathcal F}_{H_{j,+}^t}(\cdot)$ on $\Omega$ such that
\begin{equation}\label{eq:hyperplane_trace_equality}
\widetilde{\mathcal F}_{H_{j,+}^t}(S)
=(\FF \cdot \nu)_{\partial(\widetilde\Omega^{\delta} \cap H_{j,+}^t)}(S)
\end{equation}
for $\mathcal L^1$--\textit{a.e.}\,\,$\delta\in(0,1)$ with
$\widetilde\Omega^{\delta} \in \mathcal O_{\lvert \FF \rvert}$ and
for any Borel set $S \subset \widetilde\Omega^{\delta} \cap \partial H_{j,+}^t$,
where $\widetilde\Omega^{\delta}$ is defined in \eqref{eq:tilde_omega_delta}.
For any such $\delta>0$, we can apply Theorem \ref{thm:reconstruct_normaltrace}
to $U := \widetilde\Omega^{\delta} \cap H_{j,+}^t \Subset \Omega$, which gives
    \begin{equation}\label{eq:hyperplane_reconstruct}
    \int_0^{\delta} \langle \FF \cdot \nu, \,\phi \rangle_{\partial(\widetilde\Omega^{\delta} \cap H_{j,+}^t)^{\eps}} \,\d \eps
    = \int_{U \setminus \overline{U^{\delta}}} \,\phi \,\d\overline{\nabla d_U \cdot \FF}
    \qquad\text{for any } \phi \in \Lip_{\mathrm{c}}(\Omega).
  \end{equation}
 Now suppose further that $\phi$ is compactly supported
in $\widetilde\Omega^{2\delta}$, and set
\begin{equation}
V_{t,\delta} := \widetilde\Omega^{2\delta} \cap
(H_{j,+}^t \setminus \overline{H_{j,+}^{t+\delta}}).
\end{equation}
Then, for each $\eps \in (0,\delta)$,
observe that $d_U(x) = d_{H_{j,+}^t}(x) =  \eps$
for $x \in \widetilde\Omega^{2\delta} \cap \partial H_{j,+}^{t+\eps}$,
since $\dist(x,\partial\widetilde\Omega^{\delta}) \geq \min\{\delta,\delta^{-1}\}$
and $\delta<1$.
Hence, $V_{t,\delta} \subset U \setminus \overline{U}^{\delta}$
and $d_U = d_{H_{j,+}^t}$ on $V_{t,\delta}$ that is open.
Then, by the product rule (Theorem \ref{productrule}), we infer
\begin{equation}\label{eq:Fj_measure_equality}
    \overline{\nabla d_U \cdot \FF} \res V_{t,\delta} = \overline{\nabla d_{H_{j,+}^t} \cdot \FF} \res V_{t,\delta}  = \FF \cdot e_j \res V_{t,\delta}
 \end{equation}
as measures.
In addition, for each $\eps \in (0,\delta)$, the above implies that
  \begin{equation}
      (\widetilde\Omega^{\delta} \cap H_{j,+}^t)^{\eps} \cap V_{\delta} = \widetilde\Omega^{\delta} \cap  H_{j,+}^{t+\eps} \cap V_{\delta}.
  \end{equation}
By localization (Theorem \ref{thm:localisation}), we have
  \begin{equation}\label{eq:slice_fluxes_local}
      (\FF \cdot \nu)_{\partial(\widetilde\Omega^{\delta} \cap H_{j,+}^t)^{\eps}} \res V_{\delta} = (\FF \cdot \nu)_{\partial(\widetilde\Omega^{\delta} \cap H_{j,+}^{t+\eps})} \res V_{\delta} = \widetilde{\mathcal F}_{H_{j,+}^t} \res V_{\delta}
 \qquad\mbox{for $\mathcal L^1$--\textit{a.e.}\,\,$\eps$},
 \end{equation}
 by using \eqref{eq:hyperplane_trace_equality}.
Combining \eqref{eq:slice_fluxes_local} with \eqref{eq:Fj_measure_equality} in \eqref{eq:hyperplane_reconstruct}, we arrive at
  \begin{equation}
    \int_0^{\delta} \int_{\Omega \cap \partial H_{j,+}^{t+\eps}} \phi \,\d \widetilde{\mathcal F}_{H_{j,+}^{t+\eps}} \,\d \eps = \int_{H_{j,+}^t \setminus \overline{H_{j,+}^{t+\delta}}} \,\phi \,\d F_j, \qquad\text{for any } \phi  \in \Lip_{\mathrm{c}}(\Omega^{2\delta}).
  \end{equation}
  This shows that
\begin{equation}
    F_j \res V_{t,\delta} = \mathcal L^1 \otimes_{\partial H_{j,+}^t} \widetilde{\mathcal F}_{H_{j,+}^{t}}\res V_{t,\delta}
\end{equation}
as measures.
Since $\delta \in (0,1)$ is arbitrary and
$\mathcal L^1$--\textit{a.e.}\,\,$t \in \mathbb R$ can be taken, it follows that
\begin{equation}
F_j = \mathcal L^1 \otimes_{\partial H_{j,+}^t} \widetilde{\mathcal F}_{H_{j,+}^{t}}
\end{equation}
as measures in $\mathbb R^n$. Then the result follows by using \eqref{eq:hyperplane_trace_equality}.
\end{proof}

\begin{theorem}[Uniqueness of the representing field]\label{thm:flux_unique}
Given a Cauchy flux $\mathcal F$, suppose that there exists a field $\GG \in \mathcal{DM}^{\ext}(\Omega)$ such that
\begin{equation}\label{eq:normaltrace_unique}
    \mathcal F_{U}(S) = (\GG \cdot \nu)_{\partial U}(S)
\qquad\mbox{for any $U \in \mathcal O_{\mu}$ and
any Borel set $S \subset \partial U$}.
\end{equation}
Then $\FF = \GG$ as measures, where $\FF$ the field constructed in {\rm Theorem \ref{thm:flux_construct}}.
\end{theorem}

\begin{proof}
By Theorem \ref{prop:flux_sufficiency}, the field $\GG = (G_1,G_2,\cdots,G_n)$ defines
a Cauchy flux via the normal trace, which is denoted by $\mathcal G$.
Let $1 \leq j \leq n$, and consider $t \in \mathbb R$ for which both limits:
\begin{align}
\frac1{\eps} \mu \res (H_{j,+}^t \setminus \overline{H_{j,+}^{t+\eps}})
\xrightharpoonup{\,\, * \,\,}
\mu_t^{n-1}, \qquad
\frac1{\eps} \lvert \GG \rvert \res (H_{j,+}^t \setminus \overline{H_{j,+}^{t+\eps}})
\xrightharpoonup{\,\, * \,\,}
\tilde\mu_t^{n-1}
\end{align}
exist as $\eps \to 0$, which is satisfied
by $\mathcal L^1$--\textit{a.e.}\,\,$t \in \mathbb R$.
For any such $t$, applying Lemma \ref{lem:flux_extension}
to $\mathcal F$ and $\mathcal G$ with $U = H_{j,+}^t$, we see that
\begin{equation}
\mathcal F_{H_{j,+}^t}(\,\cdot \cap\Omega\cap \partial H_{j,+}^t) = \mathcal G_{H_{j,+}^t}(\,\cdot \cap\Omega\cap \partial H_{j,+}^t)
\end{equation}
are well-defined and agree as measures,
since they agree on $\widetilde\Omega^{\delta} \cap \partial H_{j,+}^t$
for $\mathcal L^1$--\textit{a.e.}\,\,$\delta>0$.
Then, applying Lemma \ref{lem:hyperplane_reconstruct} above, we conclude
 \begin{equation}
 G_j = \mathcal L^1 \otimes \mathcal F_{H_{j,+}^t}(\,\cdot\cap \partial H_{j,+}^t) \res \Omega \qquad\text{for each }1\leq j \leq n.
  \end{equation}
This is precisely how $\FF$ has been defined
in Theorem \ref{thm:flux_construct}. Therefore, $\FF = \GG$.
\end{proof}

\section{Cauchy Flux IV: Local Recovery and Applications}\label{sec:flux_localrecovery}

In \S \ref{sec:flux_existence}, we have constructed an associated field $\FF \in \mathcal{DM}^{\ext}(\Omega)$
and have shown that any field representing the flux $\mathcal F$ must take this form.
It remains to show that the field $\FF$ represents the flux in the sense of Theorem \ref{maintheorem},
which is the content of the following subsection.

\subsection{Local recovery of the flux}\label{sec:flux_localrecovery2}

The precise statement of the local recovery result is following:

\begin{theorem}[Local recovery of the flux]\label{thm:flux_localrecovery}
Let $\mathcal F$ be a Cauchy flux, and let $\FF \in \mathcal{DM}^{\ext}(\Omega)$
be the associated field from {\rm Theorem \ref{thm:flux_construct}}.
Then
    \begin{equation}\label{eq:flux_trace_equality}
        \mathcal F_{U}(S) = (\FF \cdot \nu)_{\partial U}(S)
    \end{equation}
    for any $U \in \mathcal O_{\mu}$ and any Borel set $S \subset \partial U$.
\end{theorem}

We first show this local recovery result holds on {\it almost all} cubes, before extending to general open sets in $\mathcal O_{\mu}$.

\begin{lemma}\label{lem:local_cube_recovery}
Let $\mathcal F$ be a Cauchy flux, and let $\FF$ be the associated field from {\rm Theorem \ref{thm:flux_construct}}.
Then, for any cube $Q$, there exists a null set $\mathcal N \subset \mathbb R^n$
such that, for any $x \notin \mathcal N$, if $Q_x = x+Q\Subset \Omega$, then
\begin{equation}\label{eq:cube_localrecovery}
    \mathcal F_{Q_x}(S) = (\FF \cdot \nu)_{\partial Q_x}(S) = - \mathcal F_{\,\overline{Q}_x^{\mathrm{c}}}(S) = - (\FF \cdot \nu)_{\partial\overline{Q}_x^{\mathrm{c}}}(S)
\end{equation}
for any Borel set $S \subset \partial Q_x$, where the flux
on $\overline{Q}^{\mathrm{c}}_x$ in \eqref{eq:cube_localrecovery}
is defined in {\rm Remark \ref{rem:flux_complement}}.
\end{lemma}

\begin{proof}
By Theorem \ref{prop:flux_sufficiency},  $\FF$ defines a Cauchy
flux $\widetilde{\mathcal F}$ via its normal trace.
Also, by construction of $\FF=(F_1,F_2,\cdots,F_n)$ from Theorem \ref{thm:flux_construct} and Lemma \ref{lem:hyperplane_reconstruct} applied to $\FF$,
we know that there is a null set $\mathcal N_j \subset \mathbb R$ such that
\begin{equation}\label{eq:flux_hyperplane_equality}
\widetilde{\mathcal F}_{ H_{j,+}^t} = F_j
= \mathcal F_{ H_{j,+}^t}
\end{equation}
as measures whenever $t \notin \mathcal N_j$.
Now, for any cube $Q = Q(a,b)$, consider $Q_x = x+Q$ for $x \in \mathbb R^n$ such that $Q_x \Subset \Omega$,
Lemma \ref{cor:cube_localisation2} holds for both $\mathcal F$ and $\widetilde{\mathcal F}$, and  Lemma \ref{lem:hyperplane_integrable} applies
at the endpoints $x_j+a_j$ and $x_j+b_j$ for each $1\leq j \leq n$
(so, in particular, they do not lie in $\mathcal N_j$).
By Lemma \ref{lem:good_cubes}, this holds for $\mathcal L^n$--\textit{a.e.}\,\,$x$
such that $Q_x \Subset \Omega$.
Then, for such $x$, using \eqref{eq:flux_hyperplane_equality}, we infer that,
for any Borel set $S \subset \partial Q_x$,
\begin{equation*}
\begin{split}
\mathcal F_{Q_x}(S)
&= \sum_{i=1}^n \Big( \mathcal F_{H_{i,+}^{x_i+a_i}}(S \cap \partial H_{i,+}^{x_i+a_i})
  - \mathcal F_{H_{i,+}^{x_i+b_i}}(S \cap \partial H_{i,+}^{x_j+ b_i}) \Big) \\
&= \sum_{i=1}^n \Big( \widetilde{\mathcal F}_{H_{i,+}^{x_i+a_i}}(S \cap \partial H_{i,+}^{x_i+a_i})
  - \widetilde{\mathcal F}_{H_{i,+}^{x_i+b_i}}(S \cap \partial H_{i,+}^{x_j+ b_i}) \Big) \\
&= \widetilde{\mathcal F}_{Q_x}(S) = (\FF \cdot \nu)_{\partial Q_x}(S).
\end{split}
\end{equation*}
Since Lemma \ref{cor:cube_localisation2} also gives $\mathcal F_{Q_x} = - \mathcal F_{\,\overline{Q}_x^{\mathrm{c}}}$ and similarly for $\widetilde{\mathcal F}$, \eqref{eq:cube_localrecovery} holds as required.
\end{proof}

We now extend Lemma \ref{lem:local_cube_recovery} by approximating a general
domain $U \in \mathcal O_{\mu}$ via finite unions of {\it good} cubes
for which the local recovery holds, and argue that we can pass to the limit.
To do this, we first show the existence of a family of the said good cubes.

\begin{lemma}\label{lem:good_cube_collection}
There exists a countable collection $\mathcal Q_{\Omega}$ of cubes such that
\begin{enumerate}[label=\rm(\roman*)]
\item
\label{item:cube_containment} For each $Q \in \mathcal Q_{\Omega}$,
$Q, \overline{Q}^{\mathrm{c}} \in \widetilde{\mathcal O}_{\mu}$ and $Q \Subset \Omega$.

\item
\label{item:cube_lattice} If $Q_1, Q_2 \in \mathcal Q_{\Omega}$ intersect non-trivially, then $Q_1 \cap Q_2 \in \mathcal Q_{\Omega}$.

\item
\label{item:cube_union} Every open set $U \Subset \Omega$ can be written as a union of cubes in $\mathcal Q_{\Omega}$.

\item
\label{item:cube_recovery} For each $Q \in \mathcal Q_{\Omega}$, the recovery of the flux on $Q$ and its complement can be achieved in the sense that
\begin{equation}
\mathcal F_{Q} = (\FF \cdot \nu)_{\partial Q} = - \mathcal F_{\,\overline{Q}^{\mathrm{c}}} = - (\FF \cdot \nu)_{\partial\overline{Q}^{\mathrm{c}}}
\end{equation}
holds as measures on $\partial Q$.

\item
\label{item:cube_two_sided} For each $Q \in \mathcal Q_{\Omega}$, the localization and two-sided properties of {\rm Lemmas \ref{lem:cube_localisation}} and {\rm \ref{cor:cube_localisation2}} hold for both the flux $\mathcal F$ and the normal trace $\FF \cdot \nu$.

\item
\label{item:cube_no_concentration}
$\mu(\partial Q) = 0$ and $\mu_{Q}^{n-1}(\partial^2Q)=\mu_{\overline{Q}^{\,\mathrm{c}}}^{n-1}(\partial^2Q)=0$ for each $Q \in \mathcal Q_{\Omega}$.
  \end{enumerate}
\end{lemma}

\begin{proof}
Consider the collection $\{ Q(a,b)\,:\,a,b \in \mathbb Q^n\}$
of open cubes with rational endpoints in $\mathbb R^n$,
which is countable and closed under finite intersections.

Then, for $\mathcal L^n$--\textit{a.e.}\,\,$x \in \mathbb R^n$,
Lemma \ref{lem:local_cube_recovery} applies to each $Q_x := x + Q(a,b)$ such that $Q_x \Subset \Omega$:
\begin{align}
 \mathcal F_{Q_x}(S)
 &= \sum_{i=1}^n \Big( \mathcal F_{H_{i,+}^{x_i+a_i}}(S \cap \partial H_{i,+}^{x+a_i})
    - \mathcal F_{H_{i,+}^{x_i+b_i}}(S \cap \partial H_{i,+}^{b_i})\Big) \nonumber\\
&= \sum_{i=1}^n \Big(- \mathcal F_{H_{i,-}^{x_i+a_i}}(S \cap \partial H_{i,-}^{x+a_i})
   + \mathcal F_{H_{i,-}^{x_i+b_i}}(S \cap \partial H_{i,-}^{b_i})\Big)\nonumber \\
&= - \mathcal F_{\,\overline{Q}_x^{\mathrm{c}}}(S) = (\FF \cdot \nu)_{\partial Q_x}(S)
   = - (\FF \cdot \nu)_{\partial\overline{Q}_x^{\mathrm{c}}}(S)
\label{eq:localisation_two_sided}
\end{align}
for each Borel set $S \subset \partial Q_x$.
Moreover, we can ensure that \eqref{eq:corner_small} holds
for each $Q_x$ by Lemma \ref{lem:good_cubes}.
We can also choose $x$ such that \eqref{eq:localisation_two_sided} holds
with the normal trace $\FF \cdot \nu$ in place of the flux $\mathcal F$,
by using Theorem \ref{prop:flux_sufficiency},
so that each $Q_x$ satisfies properties \ref{item:cube_recovery} and \ref{item:cube_two_sided}.
In addition, we can ensure that the endpoints $x_j+a_j$ and $x_j+b_j$
do not lie in $\mathcal N_{j,\mu}$ from Definition \ref{defn:good_hyperplanes}
for each $1 \leq j \leq n$ so that
$\mu(\partial Q_x) = 0$.
Also, Lemma \ref{eq:cube_eta_bound} holds for $Q_x$ as \eqref{eq:corner_small}
is satisfied, which ensures that
$\mu_{Q}^{n-1}(\partial^2Q)=\mu_{\,\overline{Q}^{\mathrm{c}}}^{n-1}(\partial^2Q)=0$.
Hence, $Q_x$ satisfies \ref{item:cube_no_concentration}.
Now, for such $x$, we let
\begin{equation}
\mathcal Q_{\Omega}
= \{ Q_x = x+Q(a,b)\,:\,Q_x \Subset \Omega, a, b \in \mathbb Q^n\}.
\end{equation}
By construction, \ref{item:cube_containment} and \ref{item:cube_lattice} hold.
Since every open set $U$ can be written as the union of cubes with rational endpoints, by considering $-x + U$, we see that \ref{item:cube_union} also holds.
Therefore, $\mathcal Q_{\Omega}$ satisfies \ref{item:cube_containment}--\ref{item:cube_no_concentration} as required.
\end{proof}

Next, we show that local recovery holds for unions of cubes from this good collection.

\begin{lemma}\label{lem:good_cube_union}
Given $\mathcal Q_{\Omega}$ as
in {\rm Lemma \ref{lem:good_cube_collection}},
denote $\mathcal V_{\Omega}$ as the set of all finite unions $V$ of cubes
in $\mathcal Q_{\Omega}$ given by
\begin{equation}\label{eq:closure_union}
    V = \interior\big( \bigcup_{i=1}^k \overline{Q_i} \big)
    \qquad\text{for } Q_1, \cdots, Q_k \in \mathcal Q_{\Omega}.
\end{equation}
Then, for any $V \in \mathcal V_{\Omega}$,
$\,V, \overline V^{\mathrm{c}} \in \widetilde{\mathcal O}_{\mu}$ and
\begin{equation}\label{eq:two_sided_recovery}
\mathcal F_{V} = - \mathcal F_{\,\overline{V}^{\mathrm{c}}}
= (\FF \cdot \nu)_{\partial V}
= - (\FF \cdot \nu)_{\partial\overline{V}^{\mathrm{c}}}
\end{equation}
as measures on $\Omega$.
\end{lemma}

\begin{proof}
Since $\lvert \FF\rvert \leq \mu$, by Theorem \ref{prop:flux_sufficiency},
we can view the normal trace $\FF \cdot \nu$ as a flux $\widetilde{\mathcal F}$ with the same $\mu$.
Now we divide the proof into four steps.

\smallskip
\textbf{1.} We first show that every $V \in \mathcal V_{\Omega}$ can be written
as a union of disjoint cubes.
It suffices to establish this claim in the case of two cubes, since
the general case follows by inductively replacing the cubes that overlap.
Suppose that $Q_1 = Q(a,b)$ and $Q_2 = Q(c,d)$ are contained in the union of $V$ and intersect non-trivially.
Denote $\mathcal S$ as the set of cubes $Q(x,y) \subset Q_1 \cup Q_2$ such that $x_j,y_j \in \{a_j,b_j,c_j,d_j\}$ and $x_j < y_j$ for all $1 \leq j \leq n$,
which defines a finite collection of cubes in $\mathcal Q_{\Omega}$.
We use  $\mathcal S$ to define $\widetilde S$ by discarding all
cubes $Q \in \mathcal S$ for which there is $\widetilde Q \in \mathcal S$
with $\widetilde Q \subset Q$ (this can be done inductively in any order).
Then $\widetilde S$ is a disjoint collection of cubes such that
  \begin{equation}\interior\big( \bigcup_{Q \in \widetilde S} \overline Q \big)
  = \interior\big(\overline Q_1 \cup \overline Q_2\big),
  \end{equation}
so that $Q_1$ and $Q_2$ can be replaced by this collection $\widetilde S$.

\medskip
\textbf{2}. We now show that, for any $V \in \mathcal V_{\Omega}$,
\begin{equation}\label{eq:V_limsup}
\limsup_{\eps \to 0} \frac1{2\eps} \, \mu(V^{-\eps} \setminus \overline{V^\eps}) < \infty.
\end{equation}
That is, $\mu(\partial V) = 0$ and $V, \overline V^{\mathrm{c}} \in \widetilde{\mathcal O}_{\mu}$.

We induct on the number of cubes in the union to show this.
If $V = Q$, this follows by properties \ref{item:cube_containment}
and \ref{item:cube_no_concentration} of $\mathcal Q_{\Omega}$.
Assuming that this holds for some $V$,
we consider a disjoint cube
$Q \in \mathcal Q_{\Omega}$ and put $\widetilde V
= \interior\big(\overline V \cup \overline Q\big)$.
Then, arguing analogously as in the proof of Lemma \ref{lem:sop_algebra},
we have
\begin{equation}
\widetilde V^{-\eps} \setminus \overline{\widetilde V^\eps} \subset (V^{-\eps} \setminus \overline{V^{\eps}}) \cup (Q^{-\eps} \setminus \overline{Q^{\eps}}) \qquad\text{for any } \eps>0.
\end{equation}
Integrating this over $\frac1{2\eps}\, \mu$ and sending $\eps \to 0$ yield
\eqref{eq:V_limsup}, by noting that both $Q$ and $V$ satisfy \eqref{eq:V_limsup}.

\medskip
\textbf{3}. We next show that,
for every $V \in \mathcal V_{\Omega}$ and any cube $Q \in \mathcal Q_{\Omega}$ contained in the union,
\begin{equation}
\mathcal F_V \res \partial V \cap \partial^2Q = \mathcal F_{\,\overline V^{\mathrm{c}}} \res \partial V \cap \partial^2Q = 0,
\end{equation}
and the same holds with $\widetilde{\mathcal F}$ in place of $\mathcal F$.

Writing $V \in \mathcal V_{\Omega}$ as a union of disjoint
cubes $\{Q_i\}_{i=1}^k$, by property \ref{item:cube_no_concentration}
of $\mathcal Q_{\Omega}$, we know that
  \begin{equation}
      \mu_{Q_i}^{n-1}(\partial^2Q_i)=\mu_{\overline{Q}_i^{\mathrm{c}}}^{n-1}(\partial^2Q_i)=0
      \qquad\text{ for each } 1 \leq i \leq k.
  \end{equation}
Then, by Lemma \ref{lem:eta_intersection}\ref{item:upperbound_union},
we see that
\begin{equation}
\mu_V^{n-1}(\partial V \cap \partial^2Q_i)
=\mu_{\,\overline{V}^{\mathrm{c}}}^{n-1}(\partial V \cap \partial^2Q_i)=0\qquad\text{ for each } 1 \leq i \leq k.
\end{equation}
By Definition \ref{DefinitionCauchyFlux}\ref{item:eta_bound} applied to both $\mathcal F$ and $\widetilde{\mathcal F}$, the result follows.

\medskip
\textbf{4}. We prove that, if $V \in \mathcal V_{\Omega}$ is the disjoint union
of cubes $\{Q_i\}_{i=1}^k$ in $\mathcal Q_{\Omega}$, then
\begin{equation}\label{eq:V_flux_add}
\mathcal F_{V} = \sum_{i=1}^k \mathcal F_{Q_k},
    \qquad\,\,
\mathcal F_{\,\overline V^{\mathrm{c}}}
    = \sum_{i=1}^k \mathcal F_{\,\overline Q_k^{\mathrm{c}}}
\end{equation}
as measures in $\Omega$, and the same holds with $\widetilde{\mathcal F}$ in place of $\mathcal F$.

We prove \eqref{eq:V_flux_add} by inducting on the number of cubes,
where
the case $k=1$ follows by
Lemma \ref{lem:good_cube_collection}\ref{item:cube_recovery}.
While we only consider $\mathcal F$, the argument for $\widetilde{\mathcal F}$ is analogous since it is a Cauchy flux with the same $\mu$ and $\sigma$,
and agrees with $\mathcal F$ on cubes $Q \in \mathcal Q_{\Omega}$.

Assuming that $V \in \mathcal V_{\Omega}$ satisfies \eqref{eq:V_flux_add} and $Q \in \mathcal Q_{\Omega}$ is disjoint from $V$,
define $\widetilde V := \interior(\overline V \cup \overline Q)$.
By the localization property \ref{item:localisation} applied with
$A=\Omega \setminus \overline Q$, we see that
$\mathcal F_{\widetilde V}$ and $\mathcal F_{V}$ agree on $\partial\widetilde V \setminus \partial Q$ and similarly for the complement by localizing on $Q$, so
that \eqref{eq:V_flux_add} holds on $\partial V\setminus\partial Q$.
Similarly, by localizing on $\Omega \setminus \overline V$ and $V$,
we deduce that \eqref{eq:V_flux_add} holds on $\partial Q \setminus \partial V$.

It remains to show that \eqref{eq:V_flux_add} holds on the intersection
$\partial Q \cap \partial V$.
Let $H$ be a half-space such that $\partial H$ intersects both $\partial V$
and $\partial Q$. We assume that $H = H_{j,+}^{a_j}$ for some $j$,
since the case $H = H_{j,-}^{b_j}$ is analogous.
Suppose that $\widetilde Q \subset V$ is one of the disjoint cubes
whose union gives $V$, and set
\begin{equation}
I := \partial Q \cap \partial \widetilde Q
\cap \partial H \setminus ( \partial^2Q \cup \partial^2\widetilde Q).
\end{equation}
Since the fluxes vanish on $\partial^2Q \cup \partial^2\widetilde Q$,
we know that \eqref{eq:V_flux_add} holds there so that, if $I = \varnothing$, there is nothing further to show.
Otherwise, if $I \neq \varnothing$,
observe the projections of $Q$ and $\widetilde Q$ to $\partial H_j$
intersect non-trivially. Then it follows that
$\widetilde Q \subset H_{j,-}^{a_j}$ since $Q$ and $\widetilde Q$ are disjoint.
Moreover, writing $Q = (a_1,b_1) \times Q'$ with $Q' \subset \mathbb R^{n-1}$ an open cube,
there is $\kappa>0$ such that
$(a_1-\kappa,a_1)\times Q' \subset \widetilde Q$.
Therefore, setting
  \begin{equation}
      A = (a_1-\kappa,b_1) \times Q'
  \end{equation}
which is open, we see that $A \cap H = A \cap Q$ and $A \cap \overline H^{\mathrm{c}} = A \cap \widetilde Q$.
Hence, by the localization property \ref{item:localisation} and the two-sided property of the flux on $\partial H$ from Lemma \ref{lem:good_cube_collection}\ref{item:cube_two_sided}, we deduce that
  \begin{equation}
    \mathcal F_{V} \res I = \mathcal F_{\widetilde Q} \res I
    = \mathcal F_{H_{j,-}^{a_j}} \res I = - \mathcal F_{H_{j,+}^{a_j}} \res I = - \mathcal F_{Q} \res I,
  \end{equation}
  and, for the complement, we have
    \begin{equation}
    \mathcal F_{\,\overline{V}^{\mathrm{c}}} \res I = \mathcal F_{\,\overline{\widetilde Q}^{\mathrm{c}}} \res I = \mathcal F_{H_{j,+}^{a_j}} \res I = - \mathcal F_{H_{j,-}^{a_j}} \res I = - \mathcal F_{\,\overline{Q}^{\mathrm{c}}} \res I.
  \end{equation}
  Since $I \cap \partial{\widetilde V} = \varnothing$, it follows that
  \begin{equation}
    \mathcal F_{\widetilde V} \res I = 0 = \mathcal F_V \res I + \mathcal F_{Q} \res I,
  \end{equation}
  and similarly for the complement.
Thus, by inductive hypotheses with $Q$ and $V$, \eqref{eq:V_flux_add} holds
for $\widetilde V$ on $I$, so the result follows by induction.

Finally, given $V \in \mathcal V_{\Omega}$, by the first step,
we can write $V$ as a disjoint union of cubes $\{Q_i\}_{i=1}^k$ in $\mathcal Q_{\Omega}$ and,
by Lemma \ref{lem:good_cube_collection}\ref{item:cube_recovery}, \eqref{eq:two_sided_recovery} holds for each $Q_i$. By summation, we have
  \begin{equation}
  \sum_{i=1}^k \mathcal F_{Q_i} = \sum_{i=1}^k (\FF \cdot \nu)_{\partial Q_i} = - \sum_{i=1}^k \mathcal F_{\,\overline{Q_i}^{\mathrm{c}}} = - \sum_{i=1}^k (\FF \cdot \nu)_{\partial\overline{Q_i}^{\mathrm{c}}}.
\end{equation}
Combining this with \eqref{eq:V_flux_add} as proved, we obtain \eqref{eq:two_sided_recovery} as required.
\end{proof}

\bigskip
\begin{proof}[Proof of Theorem {\ref{thm:flux_localrecovery}}]
Let $\mathcal U_{\Omega}$ be the set of $U \Subset \Omega$
such that conclusion \eqref{eq:flux_trace_equality} holds. Then
$\mathcal V_{\Omega} \subset \mathcal U_{\Omega}$
by Lemma \ref{lem:good_cube_union}.
We fix $U \in \mathcal O_{\mu}$ and then show that $U \in \mathcal U_{\Omega}$.

To achieve this, we approximate $U$ by elements in $\mathcal V_{\Omega}$ as follows:
For each $k \in \mathbb N$, take a finite covering of $\overline{U^{2^{-k}}}$
by cubes $Q \in \mathcal Q_{\Omega}$ such that each $Q \Subset U^{2^{-(k+1)}}$.
Let $V_k$ be the union of this covering in the sense of \eqref{eq:closure_union}
so that $U^{2^{-k}} \Subset V_k \Subset U^{2^{-(k+1)}}$ and  $V_k \in \mathcal V_{\Omega} \subset \mathcal U_{\Omega}$.
We now divide the remaining proof into three steps.

\medskip
\textbf{1}. We first show that, for any open cube $Q \Subset \Omega$ and $\mathcal L^1$--\textit{a.e.}\,\,$\eps>0$,
  \begin{equation}\label{eq:Qeps_limit}
    \mathcal F_{U}(Q^{\eps} \cap \partial U) = \lim_{k \to \infty} \mathcal F_{V_k}(Q^{\eps} \cap \partial V_k).
  \end{equation}
To achieve this, we set
\begin{equation}
A_{\eps,k} := Q^{\eps} \cap (U \setminus \overline{V_k}),
\end{equation}
and apply the balance law on this open set.
By discarding a null set, we can assume that $Q^{\eps} \in \widetilde{\mathcal O}_{\mu}$, so $A_{\eps,k} \in \mathcal O_{\mu}$ by Lemma \ref{lem:sop_algebra}.
Then, by the balance law \eqref{balancelaw2} and the additivity property \ref{item:additivity},
  \begin{align}
     \sigma(A_{\eps,k})
      &= \mathcal F_{A_{\eps,k}}(Q^{\eps} \cap \partial U)
      +\mathcal F_{A_{\eps,k}}(Q^{\eps} \cap \partial V_k)\nonumber\\
      &\quad\,+\mathcal F_{A_{\eps,k}}(\partial Q^{\eps} \cap (U \setminus\overline{V_k}))\nonumber \\
      &\quad\,+\mathcal F_{A_{\eps,k}}(\partial A_{\eps,k} \cap \partial Q^{\eps} \cap \partial (U \setminus\overline{V_k}))\nonumber \\
      &=: I_1 + I_2 + I_3 + I_4.\label{eq:main_splitting_k}
    \end{align}
Since sets $A_{\eps,k}$ are nested in $k$
(as $V_{k} \subset U^{2^{-(k+1)}} \subset V_{k+1}$), we have
\begin{equation}
\lim_{k \to \infty} \sigma(A_{\eps,k}) = \sigma\big(Q^{\eps} \cap (\cap_{k \in \mathbb N} (U \setminus \overline{V_k}))\big)
= \sigma(\varnothing)= 0 \qquad\mbox{for each $\eps$}.
\end{equation}
Moreover, by the localization property \ref{item:localisation} applied with $A=Q^{\eps}$ and since each $V_k \in \mathcal V_{\Omega}$, we have
 \begin{align}
    I_1 &= \mathcal F_{U}(Q^{\eps} \cap \partial U),\\
    I_2 &= \mathcal F_{(U \setminus \overline V_k)}(Q^{\eps} \cap \partial V_k) = - \mathcal F_{V_k}(Q^{\eps} \cap \partial V_k).
\end{align}
By property \ref{item:localisation} applied with $A =U \setminus \overline V_k$ and since $\mathcal F_{Q^{\eps}}$ is a measure on $\partial Q^{\eps}$ (by Lemma \ref{lem:flux_measure}), we see that
  \begin{equation}
    \lim_{k \to \infty} I_3 = \lim_{k \to \infty} \mathcal F_{Q^{\eps}}(\partial Q^{\eps} \cap (U \setminus V_k))
    = \mathcal F_{Q^{\eps}}(\partial Q^{\eps} \cap (\cap_{k \in \mathbb N} (U \setminus V_k))) = 0.
  \end{equation}
For the final term, we set
$B_{\eps,k}:= \partial A_{\eps,k} \cap \partial Q^{\eps} \cap \partial(U \setminus\overline{V_k})$ and claim that
\begin{equation}\label{eq:B_eps_k_vanishes}
\mu_{A_{\eps,k}}^{n-1}(B_{\eps,k}) = 0
\qquad \mbox{for any $k$ and $\mathcal L^1$--\textit{a.e.}\,\,$\eps>0$}.
\end{equation}
Indeed, let $\delta_j \searrow 0$ such that
\begin{equation}
    \frac1{\delta_j} \,\ \mu \res (U \setminus \overline{U^{\delta_j}})
    \xrightharpoonup{\,\, * \,\,}
      \mu_{U}^{n-1} \qquad\mbox{as $j \to \infty$}.
\end{equation}
Then, by arguing as in the proof of Lemma \ref{lem:sop_algebra},
we see that $\frac1{\delta_j} \mu(U \setminus \overline{U^{\delta_j}})$ is uniformly bounded in $j$
and hence, passing to a subsequence (depending on $\eps$ and $k$), we obtain a limiting measure
  \begin{equation}
      \frac1{\delta_j} \, \mu \res (A_{\eps,k} \setminus \overline{A_{\eps,k}^{\delta_j}})
      \xrightharpoonup{\,\, * \,\,}
          \mu_{A_{\eps,k}}^{n-1}
      \qquad\mbox{as $j \to \infty$}.
  \end{equation}
  Since $\overline{V}_k^{\mathrm{c}}$, $Q^{\eps} \in \widetilde{\mathcal O}_{\mu}$, passing to a further subsequence,
  we can also assume that there exist limit measures:
\begin{align}
\frac1{\delta_j} \, \mu \res (V_k^{-\delta_j} \setminus \overline{V_k})
\xrightharpoonup{\,\, * \,\,}
\mu_{\,\overline{V}_k^{\mathrm{c}}}^{n-1},
\quad
&\frac1{\delta_j} \,\mu \res (Q^{\eps} \setminus \overline{Q^{\eps+\delta_j}})
\xrightharpoonup{\,\, * \,\,}
\mu_{Q^{\eps}}^{n-1} \qquad\,\,\,\mbox{as $j \to \infty$}.
\end{align}
By Lemma \ref{lem:eta_intersection}\ref{item:upperbound_intersection},
  \begin{equation}\label{eq:A_eps_k_bound}
    \mu_{A_{\eps,k}}^{n-1}(B_{\eps,k}) \leq \mu_{U}^{n-1}(\partial Q^{\eps} \cap \partial U)
    + \mu_{\,\overline V_k^{\mathrm{c}}}^{n-1}(\partial Q^{\eps} \cap \partial V_k) + \mu_{Q^{\eps}}^{n-1}(B_{\eps,k})
  \end{equation}
for each $k$ and $\mathcal L^1$--\textit{a.e.}\,\,$\eps>0$.
We now show that \eqref{eq:A_eps_k_bound} vanishes for any $k$ and $\mathcal L^1$--\textit{a.e.}\,\,$\eps>0$.

Indeed, by considering
$\widetilde\mu = \mu \res (\Omega \setminus \partial U)$ and noting that
$\mu_{A_{\eps,k}}^{n-1} = \widetilde\mu_{A_{\eps,k}}^{n-1}$ for all $\eps$ and $k$,
we can assume that $\mu(\partial U) = 0$ in what follows.
Since $\{\partial Q^{\eps}\}_{\eps>0}$ are pairwise disjoint in $\eps>0$,
for all but countably many $\eps>0$, we have
  \begin{equation}
    \mu^{n-1}_{U}(\partial Q^{\eps} \cap \partial U)
    + \mu^{n-1}_{\overline{V}_k^{\mathrm{c}}}(\partial Q^{\eps} \cap \partial V_k)  = 0 \qquad\text{for any }k.
  \end{equation}
Also, by disintegration along $\partial Q^{\eps}$, we obtain
  \begin{equation}
    \int_0^{\infty}\mu^{n-1}_{Q^{\eps}}(\partial Q^{\eps} \cap (\partial U \cap \partial V_k)) \,\d \eps
    \leq \mu(Q \setminus Q^{\eps_0} \cap (\partial U \cap \partial V_k)) = 0,
  \end{equation}
by noting that $\mu(\partial U) = 0$ as assumed and $\mu(\partial V_k) = 0$
by Lemma \ref{lem:good_cube_collection}\ref{item:cube_no_concentration}.
Then, for $\mathcal L^1$--\textit{a.e.}\,\,$\eps >0$,
$\mu^{n-1}_{Q^{\eps}}(B_{\eps,k}) = 0$ for all $k$ which, combining with \eqref{eq:A_eps_k_bound}, leads to \eqref{eq:B_eps_k_vanishes} as claimed.
Hence, by property \ref{item:eta_bound}, we have
  \begin{equation}
    \lvert I_4\rvert = \lvert \mathcal F_{A_{\eps,k}}(B_{\eps,k})\rvert \leq \mu^{n-1}_{A_{\eps,k}}(B_{\eps,k})  = 0,
  \end{equation}
giving $I_4 \to 0$ for this choice of $\eps$.
Thus, combining everything in \eqref{eq:main_splitting_k}, we conclude
  \begin{equation}
    \begin{split}
      0 = \lim_{k \to \infty} \sigma(A_{\eps,k})
      &= I_1 + \lim_{k \to \infty} I_2 + \lim_{k \to \infty} (I_3+I_4) \\
     &= \mathcal F_{U}(Q^{\eps} \cap \partial U) - \lim_{k \to \infty} \mathcal F_{V_k}(Q^{\eps} \cap \partial V_k),
    \end{split}
  \end{equation}
  which rearranges to give \eqref{eq:Qeps_limit} as claimed.

\medskip
{\bf 2}. Notice that the normal trace $\FF \cdot \nu$ defines a Cauchy flux
$\widetilde{\mathcal F}$ with the same $\mu$ and $\sigma$ by Theorem \ref{prop:flux_sufficiency}.
By restricting the allowed parameter $\eps>0$ if necessary,
we can ensure that the above claim also holds for $\widetilde{\mathcal F}$,
that is,
\begin{equation}
(\FF \cdot \nu)_{\partial U}(Q^{\eps} \cap \partial U) = \lim_{k \to \infty} (\FF \cdot \nu)_{\partial V_k}(Q^{\eps} \cap \partial V_k).
\end{equation}
Since
$V_k \in \mathcal V_{\Omega} \subset \mathcal U_{\Omega}$,
then
\begin{equation}\label{eq:recovery_aecubes}
    \mathcal F_{U}(Q^{\eps} \cap \partial U) = (\FF \cdot \nu)_{\partial U}(Q^{\eps} \cap \partial U)
  \end{equation}
for all cubes $Q \Subset \Omega$ and for $\mathcal L^1$--\textit{a.e.}\,\,$\eps >0$, where the null set depends on $Q$.
We can further assume that
\begin{equation}\label{eq:recovery_aecubes_boundary}
    \mathcal F_U(\,\cdot \cap \partial U) \res\partial Q^{\eps} = (\FF \cdot \nu)_{\partial U}(\,\cdot \cap \partial U) \res\partial Q^{\eps} = 0
  \end{equation}
by restricting the allowed parameter $\eps>0$.

\medskip
{\bf 3}. To complete the proof, we need to extend \eqref{eq:recovery_aecubes}
to hold for all Borel subsets of $\partial U$.
To do this, for each $\eps>0$, define
\begin{equation}\label{eq:rational_eps_cubes}
\mathcal{D}_{\Omega}^{\eps} = \big\{ Q(a,b)^{\eps}\, :\, a,b \in \mathbb Q^n,\, Q(a,b) \Subset \Omega\big\},
\end{equation}
and let $\widetilde{\mathcal{D}}_{\Omega}^{\eps}$ be
the collection of all finite unions of cubes in $\mathcal D_{\Omega}^{\eps}$
in the sense of \eqref{eq:closure_union}.
Since this is a countable collection, for $\mathcal L^1$--\textit{a.e.}\,\,$\eps>0$,
we can ensure that \eqref{eq:recovery_aecubes}--\eqref{eq:recovery_aecubes_boundary} hold for
each $Q^{\eps} \in \mathcal{D}_{\Omega}^{\eps}$.

We first claim that $\mathcal D_{\Omega}^{\eps}$ is closed
under finite intersections. Indeed, observe that
\begin{equation}
 Q(a,b)^{\eps} = (a_1+\eps,b_1-\eps) \times \cdots \times (a_n + \eps, b_n - \eps)
 = Q(a+\eps \pmb{1},b-\eps \pmb{1}),
\end{equation}
where $\pmb{1} = (1,\cdots,1) \in \mathbb R^n$.
Using this, we can verify that $(Q_1 \cap Q_2)^{\eps} = Q_1^{\eps} \cap Q_2^{\eps}$
in general, from which the assertion follows.
Now, let $V \in \widetilde{\mathcal D}_{\Omega}^{\eps}$,
which can be written as a union $V = \interior(\bigcup_{i=1}^k \overline Q_i)$,
where
$Q_i \in \mathcal D_{\Omega}^{\eps}, i=1,\cdots, k$, are pairwise disjoint.
Then
  \begin{equation}
    \mathcal F_U(V \cap \partial U) = \sum_{i=1}^k \mathcal F_U(Q_i \cap \partial U)
    = \sum_{i=1}^k (\FF \cdot \nu)_{\partial U}(Q_i \cap \partial U) = (\FF \cdot \nu)_{\partial U}(V \cap \partial U)
  \end{equation}
by using \eqref{eq:recovery_aecubes}--\eqref{eq:recovery_aecubes_boundary}.
Thus, \eqref{eq:recovery_aecubes_boundary} holds if $Q^{\eps}$ is replaced
by the elements in $\widetilde{\mathcal{D}}_{\Omega}^{\eps}$.

Now, let $A \subset \Omega^{\eps}$ be open. We claim that,
for each $x \in A$, there is a cube in $Q^{\eps} \in \mathcal{D}_{\Omega}^{\eps}$ such that $x \in Q^{\eps} \Subset A$.
Indeed, there exists $\delta>0$ such that $Q_0 := Q(x-\delta\pmb{1},x+\delta\pmb{1}) \Subset A$. Then we can choose $y \in \mathbb Q^n$ so that
$\lvert x- y\rvert < \frac{\delta}4$ and
$r \in \mathbb Q \cap (\eps+\frac{\delta}4,\eps+\frac{\delta}2)$.
Letting $Q_1 = Q(y-r\pmb{1},y+r\pmb{1})$,
we see that $Q_1^{\eps} \in \mathcal D_{\Omega}^{\eps}$.
Since $\frac{\delta}4 < r - \eps < \frac{\delta}2$, we obtain the following inclusions:
  \begin{equation}
   x \in Q_1^{\eps} = Q(y-(r-\eps)\pmb{1},y+(r-\eps)\pmb{1}) \Subset Q_0 \Subset A,
  \end{equation}
as required.
Performing this about each point and passing to a countable subcover give a collection
$\{Q_j\} \subset \mathcal D_{\Omega}^{\eps}$ such that $\bigcup_j Q_j = A$.
Then, since \eqref{eq:recovery_aecubes} holds for elements
in $\widetilde{\mathcal{D}}_{\Omega}^{\eps}$, it follows that
\begin{equation}
\mathcal{F}_U\big(\bigcup_{j=1}^J Q_j \cap \partial U\big)
= (\FF\cdot\nu)_{\partial U}\big(\bigcup_{j=1}^J Q_j \cap \partial U\big)
\qquad\text{for any } J,
\end{equation}
   so passing to the limit gives
  \begin{equation}\label{eq:localrecovery_open}
    \mathcal{F}_U(A \cap \partial U) = (\FF \cdot \nu)_{\partial U}(A \cap \partial U)
  \end{equation}
for any $A \subset \Omega^{\eps}$ open and $\mathcal L^1$--\textit{a.e.}\,\,$\eps>0$.
Since both sides are Radon measures and $\eps>0$ can be chosen to be arbitrarily small, it follows from Lemma \ref{lem:measure_coincidence} that this holds
when $A$ is replaced by an arbitrary Borel set $B \subset \Omega$.
From this, we infer that $U \in \mathcal U_{\Omega}$, establishing the result.
\end{proof}

We can now collect the results established in the previous results to prove the main theorem.

\begin{proof}[Proof of {Theorem \ref{maintheorem}}]
$\,$Given a Cauchy flux $\mathcal F$, Theorem \ref{thm:flux_construct} gives the existence of a measure-valued field $\FF$,
which further satisfies the global and local recovery properties by Theorems \ref{thm:flux_globalrecovery} and \ref{thm:flux_localrecovery}, respectively.
Moreover, Theorem \ref{thm:flux_unique} shows that this flux is unique and the converse statement is precisely the content of Theorem \ref{prop:flux_sufficiency}.
\end{proof}

\subsection{Consequences of the main theorem}\label{sec:flux_consequences}

Using the equivalence given by Theorem \ref{maintheorem}, we can infer the
properties of the Cauchy flux based on the results about the normal traces
established in earlier sections.
We now list two of such consequences.

\begin{theorem}
Let $\mathcal F$ be a Cauchy flux, and let $U \Subset \Omega$ such that
$U, \overline{U}^{\mathrm{c}} \in \mathcal O_{\mu}$ and $\mu(\partial U) = \lvert\sigma\rvert(\partial U) = 0$.
Then
    \begin{equation}
        \mathcal F_U(S) = - \mathcal F_{\,\overline U^{\mathrm{c}}}(S) \qquad\text{for all Borel subsets } S \subset \partial U,
    \end{equation}
where
$\mathcal F_{\,\overline U^{\mathrm{c}}}$
is defined
through {\rm Remark \ref{rem:flux_complement}}.
\end{theorem}

\begin{proof}
Let $\FF$ be the corresponding $\mathcal{DM}^{\ext}$--field given by Theorem \ref{maintheorem}.
Then the corresponding result for the normal traces holds by Remark \ref{rem:inner_outer_trace}, which asserts that
\begin{equation}\label{eq:twosided_distribution}
    \langle \FF \cdot \nu,\, \cdot\,\rangle_{\partial U}
    = \langle \FF \cdot \nu,\, \cdot\, \rangle_{\partial\overline U}
\end{equation}
as distributions, whenever $\lvert \FF \rvert(\partial U) = \lvert \div \FF \rvert(\partial U) = 0$, and this is satisfied since $\lvert \FF \rvert \ll \mu$
and $\lvert \div \FF \rvert \ll \lvert\sigma\rvert$.
Since $U \in \mathcal O_{\mu}$, both sides
of \eqref{eq:twosided_distribution} can be represented as measures,
which remain equal.
By Theorem \ref{maintheorem},
$\mathcal F_{U}(\cdot) = \langle \FF \cdot \nu, \,\cdot\,\rangle_{\partial U}$ as measures.
On the other hand, since $\overline U^{\mathrm{c}} \in \mathcal{O}_{\mu}$, using the notation from Remark \ref{rem:flux_complement}, for $\mathcal L^1$--\textit{a.e.}\,\,$\delta>0$ such that $U \Subset \tilde\Omega^{\delta}$, we have
\begin{equation}
    \langle \FF \cdot \nu, \,\cdot\,\rangle_{\partial( \tilde\Omega^{\delta}\setminus \overline U)}
    = \langle \FF \cdot \nu,\,\cdot\,\rangle_{\partial \tilde\Omega^{\delta}}
    - \langle \FF \cdot \nu, \,\cdot\,\rangle_{\partial \overline U}.
\end{equation}
Restricting to $\tilde\Omega^{\delta}$ by using Theorem \ref{thm:localisation}
and Remark \ref{rem:flux_complement}, we obtain
\begin{equation}
    \mathcal F_{\overline U^{\mathrm{c}}}
    = \mathcal F_{\tilde\Omega^{\delta}\setminus V} \res \tilde \Omega^{\delta}
    = \langle \FF \cdot \nu, \,\cdot\,\rangle_{\partial( \tilde\Omega^{\delta}\setminus \overline U)} \res \tilde\Omega^{\delta}
    = -\langle \FF \cdot \nu, \,\cdot\,\rangle_{\partial \overline U}.
\end{equation}
We combine this with \eqref{eq:twosided_distribution} to complete the proof.
\end{proof}

\begin{theorem}
For any open set $U \Subset \Omega$,
    \begin{equation}\label{eq:flux_conv_global}
        \lim_{\eps \to 0} \mathcal F_{U^{\eps}}(\partial U^{\eps}) = \mathcal F_U(\partial U).
    \end{equation}
If, in addition, $U \in \mathcal O_{\mu}$, then there exists $t_k \to 0$ such that each $U^{t_k} \in \mathcal O_{\mu}$, and
\begin{equation}\label{eq:flux_conv_local}
\mathcal F_{U^{t_k}}
\xrightharpoonup{\,\, * \,\,}
\mathcal F_U
\qquad \mbox{as measures in $\Omega$}.
\end{equation}
\end{theorem}

\begin{proof}
The global convergence of the flux sequence in \eqref{eq:flux_conv_global}
is immediate from the balance law \eqref{balancelaw2}, since $\bigcup_{\eps>0} U^{\eps} = U$.
 For the local convergence, since $U \in \mathcal O_{\mu}$,
 there is $\eps_k \searrow 0$ such that
 \begin{equation}
 M := \sup_{k \in \mathbb N} \frac1{\eps_k} \,\mu(U \setminus \overline{U}^{\eps_k})
 < \infty.
 \end{equation}
By Lemma \ref{lem:disint_flux}, we can disintegrate $\mu$ along $\partial U^t$ as
\begin{equation}
\mu = \mathcal L^1 \res [0,\infty) \otimes_{\partial U^t} \mu_{U^t}^{n-1} + \tau_{\mathrm{sing}} \otimes_{\partial U^t} \tilde \mu_t.
\end{equation}
Consider the function:
    \begin{equation}
        T(t) = \lvert \mu_{U^t}^{n-1} \rvert(\partial U^t)
    \end{equation}
defined $\mathcal L^1$--\textit{a.e.}\,\,$t>0$
such that \eqref{eq:disint_mut_conv} holds,
which also holds for any $t$ outside a null set $\mathcal N \subset [0,\infty)$.
We set $T(t) = 0$ whenever $t \in \mathcal N$.
Then $T \in \Leb^1([0,\infty))$ and
\begin{equation}
 \frac1{\eps_k} \int_0^{\eps_k} T(t) \,\d t \leq M
 \qquad \mbox{ for any $k$}.
 \end{equation}

 Applying the Markov inequality gives
    \begin{equation}
        \mathcal L^1(\{t \in (0,\eps_k) \setminus \mathcal N_k : T(t) \geq 2M\}) \leq \frac1{2M} \int_0^{\eps_k} T(t) \,\d t \leq \frac{\eps_k}2,
    \end{equation}
so that there exists a sequence $t_k \in (0,\eps_k)$ of distinct values such that $T(t_k) \leq 2M$ for all $k$.
 Then $t_k \to 0$ and
\begin{equation}
\lvert \mathcal F_{U^{t_k}} \rvert(\partial U^{t_k})
\leq \eta_{U^{t_k}}(\partial U^{t_k}) = T(t_k) \leq 2M
\qquad \mbox{ for any $k$}.
\end{equation}
Passing to a subsequence, we see that
$\mathcal F_{U^{t_k}}$ converges weakly${}^\ast$ in $\Omega$.

Now, by Theorem \ref{maintheorem},
$\mathcal F$ is represented by the normal trace of a field $\FF \in \mathcal{DM}^{\ext}(\Omega)$, so that
the weak${}^{\ast}$--limit of these measures can be identified with $\mathcal F_U = (\FF \cdot \nu)_{\partial U}$ by Theorem \ref{principal}.
\end{proof}

\section{Extension of the Normal Trace}\label{sec:extended_normaltrace}

So far, we have restricted our attention to the normal trace on the boundary of a set $E$ that is compactly contained in $\Omega$. However, in applications, it is useful to consider a set $E \subset \Omega$
whose boundary may touch $\partial \Omega$,
for which we introduce the following definition:

\begin{definition}
Let $\Omega \subset \mathbb R^n$ be open and $\FF \in \mathcal{DM}^{\ext}(\Omega)$.
For a Borel set $E \subset \Omega$, define the \emph{normal trace} of $\FF$ on the boundary of $E$ as
\begin{equation}
\langle \FF \cdot \nu, \,\phi \rangle_{\partial E}
= - \int_{E} \d \overline{\nabla\phi \cdot \FF} - \int_E \phi \,\d(\div \FF)
\qquad\text{for any } \phi \in \Sob^{1,\infty}(\Omega),
\end{equation}
where the middle term is defined from the product rule \rm{(Theorem \ref{productrule})}.
\end{definition}

If $E \Subset \Omega$, then this coincides with Definition \ref{defn:normal_trace}
when testing against $\phi \in C^1_{\mathrm{c}}(\Omega)$,
and it is precisely how we extended the normal trace to $\Sob^{1,\infty}(\Omega)$ in Corollary \ref{cor:normaltrace_lipschitz_extension}.
We do not require that $\phi$ vanishes on $\partial \Omega$,
since $\partial E \cap \partial \Omega$ may be non-empty in general.
We can also generalize \eqref{eq:trace_balance_law} as
\begin{equation}\label{eq:extended_balance_law}
    \langle \FF \cdot \nu, \,\mathbbm{1}_{\Omega}\rangle_{\partial U} = (\div \FF)(U)
    \qquad\text{for any } U \subset \Omega.
\end{equation}

We record that the normal trace remains supported on $\partial U$ in the sense of Theorem \ref{prop:normaltrace_support}.

\begin{lemma}\label{lem:extended_flux_support}
Let $U \subset \Omega$ be open, and let $\phi \in \Lip_{\mathrm{b}}(\Omega)$ vanish
on $\partial U$. Then
\begin{equation}
\langle \FF \cdot \nu, \,\phi \rangle_{\partial U} = 0.
\end{equation}
\end{lemma}

Equipped with this result, we can also argue as in Corollary \ref{eq:normaltrace_lipschitz_dual}
and view the normal trace on an open set $U \subset \Omega$
as a linear functional $N_U \in \Lip_{\mathrm{b}}(\partial U)^{\ast}$.

\begin{proof}
Since the argument is similar to the proof of Theorem \ref{prop:normaltrace_support},
we only outline the main modifications.
Observe that, since any function $\phi \in \Lip_{\mathrm{b}}(\Omega)$ admits
a unique continuous extension
to $\partial\Omega$, the condition: $\phi\rvert_{\partial U} = 0$ is well-posed.

Given $0 < \delta < 1$, let $d_{\delta}$ as in \eqref{eq:ddelta_defn} with $U$ in place of $E$ so that
$d_{\delta}$ is supported in $\Omega^{\delta}$.
Also, let $\chi_{\delta} \in \Sob^{1,\infty}(\mathbb R^n)$ be $1$-Lipschitz
such that $\chi_{\delta} \equiv 1$ in $B_{1/(2\delta)}$ with support in $B_{1/\delta}$.
Thus, letting $\psi_{\delta} := \chi_{\delta}d_{\delta}$,
we see that $\psi_{\delta}$ is $\frac1{\delta}$-Lipschitz and is supported
on $\widetilde\Omega^{\delta} = B_{1/\delta} \cap \Omega^{\delta} \Subset \Omega$.

Then, by Theorem \ref{prop:normaltrace_support} with $\psi_{\delta}\phi$
on $U \cap \widetilde\Omega^{\delta}$, we have
    \begin{equation}
      \begin{split}
        0 &= \langle \FF \cdot \nu, \,\psi_{\delta}\phi \rangle_{\partial(U \cap \widetilde\Omega^{\delta})}\\
        &= \int_{U \cap \widetilde\Omega^{\delta}} \psi_{\delta}\phi \,\d(\div \FF) + \int_{U \cap \widetilde\Omega^{\delta}} \d \overline{\nabla (\psi_{\delta}\phi) \cdot \FF} \\
        &= \int_{U} \psi_{\delta}\phi \,\d(\div \FF)
           + \int_{U}\d \overline{\nabla (\psi_{\delta}\phi) \cdot \FF} \\
        &= \langle \FF \cdot \nu, \,\psi_{\delta}\phi \rangle_{\partial U},
        \end{split}
    \end{equation}
by noting that both integrals vanish outside $\widetilde\Omega^{\delta}$.

We now pass to the limit in $\delta \searrow 0$.
Using the definition of the pairing from Theorem \ref{productrule},
\begin{equation}
\overline{\nabla(\psi_{\delta}\phi)\cdot \FF}
= \psi_{\delta}\,\overline{\nabla \phi \cdot \FF}
+ \phi\,\overline{\nabla \psi_{\delta} \cdot \FF},
\end{equation}
and since $\nabla\psi_{\delta}$ is supported on $A_{\delta} = \widetilde\Omega^{\delta} \setminus \Omega^{2\delta}$
and $\lvert \phi \rvert \leq 2 \delta \lVert \nabla\phi \rVert_{\Leb^{\infty}(\Omega)}$ on $A_{\delta}$,
analogously to \eqref{eq:ddelta_zero_conv}, we obtain
    \begin{equation}
        \Big\lvert \int_U \phi \,\d \overline{\nabla \psi_{\delta}\cdot \FF} \Big\rvert \leq 2 \lVert\nabla\phi\rVert_{\Leb^{\infty}(\Omega)} \lvert \FF \rvert(A_{\delta}),
    \end{equation}
    which vanishes as $\delta \to 0$.
Finally, since $\psi_{\delta}(x) \to 1$ for any $x \in U$,
by the Dominated Convergence Theorem and the above bounds, we have
\begin{equation}
0 = \lim_{\delta \to 0} \langle \FF \cdot \nu, \,\psi_{\delta}\phi \rangle_{\partial U}
= \langle \FF \cdot \nu, \,\phi \rangle_{\partial U}
\end{equation}
as required.
\end{proof}

\begin{remark}
While the corresponding result for $U \Subset \Omega$ ({\rm Theorem \ref{prop:normaltrace_support}})
is stated for any function $\phi \in \Sob^{1,\infty}(\Omega)$,
we have stated {\rm Lemma \ref{lem:extended_flux_support}} for the test function
in $\Lip_{\mathrm{b}}(\Omega)$.
This is because, for a general function $\phi \in \Sob^{1,\infty}(\Omega)$,
we only know that $\phi$ is continuous in $\Omega$,
which means that the condition: $\phi\rvert_{\partial U} = 0$ may be ill-posed if $\partial U \not\subset \Omega$.
However, if we take $\phi \in \Sob^{1,\infty}(\Omega)$ and additionally impose that
    \begin{equation}
        \lvert \phi(x) \rvert \leq C \delta \qquad \text{for any } x \in \partial U^{\delta}
    \end{equation}
    for some $C>0$ and all $\delta > 0$,
then we can argue analogously as in the above proof.
\end{remark}

We can also extend Theorems \ref{principal}, \ref{thm:reconstruct_normaltrace},
and \ref{thm:averaged_trace}
to hold for general $U \subset \Omega$.

\begin{theorem}\label{thm:extend_traceapprox}
Let $\FF \in \mathcal{DM}^{\ext}(\Omega)$, and let $U \subset \Omega$ be open and $d(x) := \dist(x,\partial U)$.
Then there exists an $\mathcal L^1$-null set $\mathcal N \subset (0,\infty)$ such that
\begin{enumerate}[label=\rm(\roman*)]
\item\label{item:extended_measure}
For any $\eps \notin\mathcal N$, the normal trace
$\langle\FF \cdot \nu, \,\cdot\,\rangle_{\partial U^{\eps}}$ is represented by a measure on $\partial U^{\eps}$,
denoted by $(\FF \cdot \nu)_{\partial U^{\eps}}$.

\item\label{item:extended_pointwise_limit}
For any $\phi \in \Sob^{1,\infty}(\Omega)$ and any sequence $\eps_k \to 0$ with $\eps_k \notin\mathcal N$,
\begin{equation}
\langle \FF \cdot \nu, \,\phi \rangle_{\partial U}
= \lim_{k \to \infty} \int_{\partial U^{\eps_k}} \phi \,\d(\FF \cdot \nu)_{\partial U^{\eps_k}}.
\end{equation}

\item\label{item:extended_averaged_limit}
For any $\phi \in \Sob^{1,\infty}(\Omega)$,
\begin{equation}
\langle \FF \cdot \nu, \,\phi\rangle_{\partial U}
= \lim_{\eps \to 0} \frac1{\eps} \int_{U \setminus \overline{U^{\eps}}} \phi \,\d\overline{\nabla d \cdot \FF}.
\end{equation}

\item\label{item:extended_recovery} For any bounded Borel function $\phi$ on $\Omega$,
 \begin{equation}
\int_0^{\infty} \int_{\partial U^t} \phi \,\d(\FF \cdot \nu)_{\partial U^t} \,\d t
= \int_U \phi \,\d \overline{\nabla d \cdot \FF},
\end{equation}
understanding that the integrand is defined for $t \in (0,\infty) \setminus\mathcal N$.
\end{enumerate}
\end{theorem}

\begin{proof}
For $0 < t < s$, define $\psi^U_{t,s}$ by \eqref{eq:psi_st_definition},
which lies in $\Lip_{\mathrm{b}}(\Omega) \subset \Sob^{1,\infty}(\Omega)$ and vanishes on $\partial U$.
Choose $s$ and $t$ such that $\lvert\overline{\nabla d\cdot\FF}\rvert(\partial U^t)
=\lvert\overline{\nabla d\cdot\FF}\rvert(\partial U^s) = 0$, which holds for all but countably many $s$ and $t$.
By Lemma \ref{lem:extended_flux_support} and the product rule (Theorem \ref{productrule}), we have
    \begin{equation}
    \begin{split}
        0 &= \langle \FF \cdot \nu, \,\psi_{t,s}^U\phi\rangle_{\partial U} \\
          &= \int_U \psi_{s,t}^U \,\d(\div(\phi \FF)) + \int_U \d\overline{\nabla \psi_{t,s}^U \cdot (\phi \FF)} \\
          &= \int_U \psi_{t,s}^U \,\d(\div(\phi\FF)) + \int_{U^t \setminus \overline{U^s}} \phi \,\d\overline{\nabla d \cdot \FF}.
    \end{split}
    \end{equation}
We can then argue as in the proof of Theorem \ref{principal} to show that
the disintegration of $\overline{\nabla d \cdot \FF}$ takes the form:
    \begin{equation}
        \overline{\nabla d \cdot \FF}
        = \mathcal L^1 \otimes_{\partial U^t} (\FF \cdot \nu)_{\partial U^t} + \tau_{\mathrm{sing}} \otimes_{\partial U^t} \mu_t,
    \end{equation}
from which \ref{item:extended_measure} and \ref{item:extended_pointwise_limit} follow.
For \ref{item:extended_recovery},
we can argue as in the proof of Theorem \ref{thm:reconstruct_normaltrace} to show that $\tau_{\mathrm{sing}}=0$,
by observing the test function $\phi(x) = \psi(d(x)) g(x)$ remains compactly supported
when $U \subset \Omega$.
Finally, \ref{item:extended_averaged_limit}
follows from \ref{item:extended_recovery}
by arguing exactly as in the proof of Theorem \ref{thm:averaged_trace}.
\end{proof}

The localization result from \S \ref{sec:localisation} also extends to this setting as follows:

\begin{theorem}\label{thm:localisation_extended}
Let $\Omega \subset \mathbb R^n$ be open and $\FF \in \mathcal{DM}^{\ext}(\Omega)$.
For $U, V \subset \Omega$, suppose that an open set $A \subset \mathbb R^n$ satisfies
    \begin{equation}
        U \cap A = V \cap A.
    \end{equation}
Then
\begin{equation}
\langle \FF \cdot \nu, \,\phi \rangle_{\partial U}
= \langle \FF \cdot \nu, \,\phi \rangle_{\partial V} \qquad
\text{for any } \phi \in \Sob^{1,\infty}_{\mathrm{c}}(A).
\end{equation}
\end{theorem}
As in Theorem \ref{thm:localisation}, if the normal traces on $\partial U$
and $\partial V$ are represented by measures, then a density argument implies that
\begin{equation}\label{eq:localisation_meausres2}
    (\FF \cdot \nu)_{\partial U} \res(\partial U \cap A)
    = (\FF \cdot \nu)_{\partial V} \res (\partial V \cap A)
\end{equation}
as measures.
The proof of Theorem \ref{thm:localisation_extended} is analogous to that of Theorem \ref{thm:localisation},
by noting that
$A$ can be replaced by a bounded open set $\tilde{A}$ containing
$A \cap \spt(\phi)$, since $\phi$ is compactly supported in $A$,
and that Theorem \ref{thm:extend_traceapprox}\ref{item:extended_averaged_limit} can be applied
in place of Theorem \ref{thm:averaged_trace}.

Finally, we can formulate the Cauchy flux
to be defined up to the boundary of $\Omega$.
Given a non-negative measure $\mu$ on $\Omega$, we introduce the set:
\begin{equation}
\overline{\mathcal O}_{\mu}
= \Big\{ U \subset \Omega \text{ open}\,:\,\liminf_{\eps \to 0}\frac1{\eps}\, \mu(U \setminus \overline{U^{\eps}}) < \infty\Big\}.
\end{equation}

\begin{definition}
An \emph{extended Cauchy flux} $\Omega$ is a mapping $\mathcal{F}$ defined on pairs $(S,U)$ with
$S \subset \partial U$ Borel and $U \subset \Omega$ open, which satisfies the balance law \eqref{balancelaw2}
for any $U \subset \Omega$ so that
there is a Radon measure $\mu$ such that properties \ref{item:additivity}--\ref{item:eta_bound}
hold with $\mathcal O_{\mu}$ replaced by $\overline{\mathcal O}_{\mu}$.
\end{definition}

\begin{theorem}
Let $\mathcal F$ be an extended Cauchy flux in $\Omega$.
Then there exists a unique divergence-measure field $\FF \in \mathcal{DM}^{\ext}(\Omega)$
representing $\mathcal F$ in the sense that
both the global recovery{\rm :}
    \begin{equation}\label{eq:extended_globalrecovery}
        \mathcal F_{U}(\partial U) = \langle \FF \cdot \nu, \,\mathbbm{1}_{\Omega} \rangle_{\partial U}
        \qquad\text{for any open set } U \subset \Omega,
    \end{equation}
and the local recovery:
\begin{equation}\label{eq:extended_localrecovery}
\mathcal F_U(S) = (F \cdot \nu)_{\partial U}(S) \qquad
\mbox{for any $U \in \overline{\mathcal O}_{\mu}$ and any Borel Set $S \subset \partial U$}
\end{equation}
hold.
Conversely, any $\FF \in \mathcal{DM}^{\ext}(\Omega)$ defines an extended Cauchy flux by \eqref{eq:extended_globalrecovery}--\eqref{eq:extended_localrecovery},
and every Cauchy flux $\mathcal F$ in the sense of {\rm Definition \ref{DefinitionCauchyFlux}}
extends uniquely to an extended Cauchy flux.
\end{theorem}

\begin{proof}
Since an extended Cauchy flux $\mathcal{F}$ is also a Cauchy flux in the sense of Definition \ref{DefinitionCauchyFlux},
the existence and uniqueness of a representing field
$\FF \in \mathcal{DM}^{\ext}(\Omega)$ satisfying $-\div \FF = \sigma$
follows from Theorems \ref{thm:flux_construct}--\ref{thm:flux_globalrecovery} and \ref{thm:flux_unique}.
Then, from \eqref{eq:extended_balance_law}, we obtain \eqref{eq:extended_globalrecovery}.

It remains to extend the local recovery result (Theorem \ref{thm:flux_localrecovery}) to hold
for any $U \in \overline{\mathcal O}_{\mu}$.
For this, we closely follow the proof of Theorem \ref{thm:flux_localrecovery},
approximating $U \in \overline{\mathcal O}_{\mu}$ by subsets $V_k \in \mathcal V_{\Omega}$;
since each $V_k \Subset \Omega$, the localization and two-sided properties are guaranteed
by Lemma \ref{lem:good_cube_union}.
The key difference is to show that \eqref{eq:Qeps_limit} holds
for all cubes $Q \subset \mathbb R^n$ and $\mathcal L^1$--\textit{a.e.}\,\,$\eps>0$,
regardless of whether $Q^{\eps}$ is contained in $\Omega$.
This is possible because $A_{\eps,k} = Q^{\eps} \cap (U \setminus \overline{V}_k) \subset \Omega$,
so that the balance law \eqref{balancelaw2} can be applied in this extended setting.
We choose $\eps>0$ such that
    \begin{equation}
        \limsup_{\delta \to 0} \frac1{\delta} \, \mu\big(\Omega \cap (Q^{\eps} \setminus Q^{\eps+\delta})\big) < \infty,
    \end{equation}
which is valid for $\mathcal L^1$--\textit{a.e.}\,\,$\eps>0$ by Lemma \ref{lem:disint_flux},
extending $\mu$ by zero to $\mathbb R^n$.
Then we can argue analogously as in the proof of Theorem \ref{thm:flux_localrecovery}
(especially \eqref{eq:Qeps_limit}) that
    \begin{equation}
        \mathcal{F}_U(Q^{\eps} \cap \partial U) = \lim_{k \to \infty} \mathcal{F}_{V_k}(Q^{\eps} \cap \partial V_k)
        \qquad \mbox{for any such $Q^{\eps}$}.
    \end{equation}
We then consider the collection of cubes $\mathcal{D}_{\mathbb R^n}^{\eps}$ as in \eqref{eq:rational_eps_cubes}
and observe that, for $\mathcal L^1$--\textit{a.e.}\,\,$\eps>0$,
both \eqref{eq:recovery_aecubes}--\eqref{eq:recovery_aecubes_boundary} hold
for each $Q \in \mathcal{D}_{\mathbb R^n}^{\eps}$.
Then, covering $A \subset \mathbb R^n$ by cubes in $\mathcal{D}_{\mathbb R^n}^{\eps}$,
we can infer that the local recovery \eqref{eq:localrecovery_open} holds for any open set  $A$,
from which a density argument extends the result to all Borel subsets.

Finally, the fact that a flux $\FF \in \mathcal{DM}^{\ext}(\Omega)$ defines an extended Cauchy
flux follows analogously as in the proof of Theorem \ref{prop:flux_sufficiency}, by using \eqref{eq:extended_balance_law} and
Theorems \ref{thm:extend_traceapprox}--\ref{thm:localisation_extended} if necessary.
Also, if $\mathcal{F}$ is a Cauchy flux, letting $\FF \in \mathcal{DM}^{\ext}(\Omega)$ be the associated field
guaranteed by Theorem \ref{maintheorem}, the extended Cauchy flux $\widetilde{\mathcal F}$ defined via the normal trace
of $\FF$ is the unique extension of $\mathcal F$.
\end{proof}

\section{Remarks on the Existence of Divergence-Measure Fields}\label{sec:alt_existence}
The characterization of the solvability of the equation with a Radon measure $\sigma$:
\begin{equation}\label{unicacosa}
    -\div \FF = \sigma
\end{equation}
has been obtained in Phuc-Torres \cite{PT1,PT2}
in several spaces of functions,
including continuous vector fields and vector fields in $\Leb^{p}$, $1 \leq p \leq \infty$.

The next theorem provides a way to show the existence of a solution of \eqref{unicacosa}
that is a vector-valued measure, that is, an extended divergence-measure field.

 \begin{theorem}\label{globalexistence}
Given a finite signed Radon measure $\sigma$ in a bounded open
set $\Omega \subset \mathbb{R}^{n}$,
there exists a vector-valued measure $\FF=(F_1,\cdots,F_n) \in \mathcal M(\Omega;\mathbb R^n)$
such that $-\div \FF=\sigma$ in the sense of distributions.
 \end{theorem}

\begin{proof}
Let $C_{\mathrm{c}}(\Omega)$ be the vector space of all real continuous functions
with compact support defined in $\Omega$, equipped with the norm:
$$
\norm{u}_{C_0(\Omega)} = \sup_{x \in \Omega}\,\lvert u(x) \rvert.
$$
Let $C_0(\Omega)$ denote the completion of $C_{\mathrm{c}}(\Omega)$ in this norm.
Similarly, let $C_{\mathrm{c}}^{1}(\Omega)$ be the vector space
of all differentiable functions with compact support in $\Omega$ such that
the derivatives are also continuous in $\Omega$, which is equipped with the norm:
 $$
 \norm{u}_{C_0^1(\Omega)}
 = \sup_{x \in \Omega}\,\lvert u(x)\rvert + \sup_{x \in \Omega}\,\lvert\nabla u(x)\rvert.
 $$
 Let $C_0^{1}(\Omega)$ denote the completion of $C_{\mathrm{c}}^{1}(\Omega)$ with this norm.
 Observe that $\sigma \in C_0^1(\Omega)^{*}$, since  we can estimate
 \begin{equation}
 \label{accion}
 \Big\lvert \int_{\Omega} \varphi  \,\d\sigma \Big\rvert
 \leq  \lvert\sigma\rvert(\Omega)\,\sup_{\Omega}\, \lvert\varphi\rvert
  = \lvert\sigma\rvert(\Omega) \norm{\varphi}_{C_0^1(\Omega)}
 \qquad \mbox{for any $\varphi \in C_0^1(\Omega)$}.
 \end{equation}
  We define
 \begin{equation}
 A: C_0^1(\Omega) \to C_0(\Omega), \qquad A(u)= \nabla u,
 \end{equation}
which is a bounded linear operator, since
\begin{eqnarray}
\norm{Au}_{C_0(\Omega)}
  &=& \norm{\nabla u}_{C_0(\Omega)}
   \leq \norm{u}_{C_0^1(\Omega)}\qquad \mbox{for every $u \in C_0^1(\Omega)$}.
\end{eqnarray}
Since $C_0^1(\Omega)$ is complete, the range $R(A)$ of $A$
is a closed subspace of $C_0(\Omega)$. We now define the functional
\begin{equation}
\tilde{L}:R(A) \to \mathbb{R}, \qquad R(A) \subset C_0(\Omega)
\end{equation}
as
\begin{equation}
\tilde{L}(\nabla u)= \sigma(u).
\end{equation}
We claim that $\tilde{L}$ is well defined.

Indeed, suppose that $\nabla u_1=\nabla u_2$. Then, on each connected component, $u_1=u_2 +C$ for some constant $C$.
We recall that $u \in C_0^1(\Omega)$ if and only if both $u$ and $\nabla u$ are continuous
in $\Omega$ and,
for any $\eps>0$, there exists $K \Subset \Omega$
such that $\lvert u(x)\rvert\leq \varepsilon$ and $\lvert\nabla u(x)\rvert \leq \varepsilon$
for any $x \in \Omega \setminus K$.
This characterization of $C_0^1(\Omega)$ implies that $C=0$ on each connected component.

We now show that $\tilde{L}$ is continuous on $R(A) \subset C_0(\Omega)$.
For any $u \in C_0^1(\Omega)$, using the first part in \eqref{accion}, we compute
\begin{equation}
\lvert\tilde{L}(\nabla u)\rvert
= \lvert\sigma (u)\rvert \leq \abs{\sigma}(\Omega) \norm{u}_{C_0(\Omega)}.
\end{equation}
By the Hahn-Banach Theorem, $\tilde{L}$ can be extended to a continuous linear functional
\begin{equation}
L:\,C_0(\Omega) \to \mathbb{R}.
\end{equation}
Hence, by the Riesz Representation Theorem,
there exists a unique $\mathbb{R}^n$-valued finite Radon measure
${\bf \FF}=(F_1,F_2,\cdots,F_n)$ such that
$$
L(u)=\sum_{i=1}^{n} \int_{\Omega} u_i \,\d F_i \qquad \text{for any } u \in C_0(\Omega).
$$
In particular, for any $\varphi \in C_{\mathrm{c}}^{\infty}(\Omega)$, we have
\begin{equation}
\label{uno}
\sigma(\varphi)= \tilde{L}(\nabla \varphi)= L(\nabla \varphi)
=\int_{\Omega} \nabla \varphi \cdot \, \d \FF.
\end{equation}
Since the distributional divergence of $\FF$ is given by
\begin{equation}
\label{dos}
\langle -\div \FF, \,\varphi\rangle= \int_{\Omega} \nabla \varphi \cdot \d \FF,
\end{equation}
we conclude from \eqref{uno}--\eqref{dos} that
$$
\langle-\div \FF, \,\varphi\rangle = \sigma (\varphi) \qquad \text{for any } \varphi \in C_{\mathrm{c}}^{\infty}(\Omega),
$$
that is, $-\div \FF = \sigma$.
 \end{proof}

\begin{remark}
A more classical method for solving equation \eqref{unicacosa}
is to use first the Newtonian potential to solve the equation: $-\Delta u =\sigma$ and
then define $\FF:= \nabla u$.
This approach is used in \cite[{\rm Example 3.3(i)}]{Silhavy1},
where it is shown that, if $\sigma$ has compact support in $\mathbb{R}^n$,
a solution to \eqref{unicacosa} is given by $\FF(x) = -\frac{1}{n\omega_n} \int_{\mathbb R^n} \frac{\d \sigma(y)}{\lvert x - y \rvert^{n-1}}$,
where $\omega_n$ is the volume of the unit ball in $\mathbb R^n$.
Note that, if $\sigma$ is a measure on a bounded domain $\Omega$,
we can apply this by taking a zero extension of $\sigma$ to $\mathbb R^n$.
Moreover, $u$ belongs to $\Sob^{1,p}_{\mathrm{loc}}(\mathbb R^n)$
for all $1 \leq p < \frac{n}{n-1}$ so that
$\FF = \nabla u \in \Leb^{p}_{\text{loc}}(\mathbb{R}^n, \R^n)$.
If $\sigma=0$, the Newtonian potential approach
clearly gives $\FF=0$.

Without the assumption of the measure having compact support, using similar techniques as in Theorem \ref{globalexistence}, it was shown in \cite[Theorem 3.1]{PT1} that, if $1 \leq p \leq \frac{n}{n-1}$ and $\sigma$ is any positive Radon measure in $\R^n$, equation \eqref{unicacosa} has a solution $\FF \in L^p(\R^n,\R^n)$ if and only if $\sigma =0$. Moreover, it was indicated in \cite[Theorem 3.2]{PT1} that,
for $\frac{n}{n-1} < p< \infty$,  \eqref{unicacosa} has a solution $\FF \in L^p(\R^n,\R^n)$ if and only if $I_1 \sigma \in L^p(\R^n)$, where $I_1\sigma$ is the Riesz potential of order 1 of $\sigma$ defined as $I_1 \sigma (x)= \int_{\R^n} \frac{\d \sigma (y)}{|x-y|^{n-1}} $. A characterization of this type is still unknown if $\sigma$ is a signed measure.

If $\sigma$ is represented by a function in $L^p$
for $1 < p < \infty$, there exists a function $u \in W^{2,p}(\Omega)$ satisfying $-\Delta u =\sigma \in L^p (\Omega)$,
and $\FF=\nabla u \in W^{1,p} (\Omega, \R^n)$ solves \eqref{unicacosa}.
For $p=1$, there are examples of functions $\sigma \in L^1(\Omega)$
for which there is no solution for \eqref{unicacosa}
with $\FF \in W^{1,1}(\Omega,\R^n)$ (\cite[\S 2.1]{BB}). For $p =\infty$, there are functions $\sigma \in L^{\infty}(\Omega)$
for which there is no solution for \eqref{unicacosa} in $W^{1,\infty}(\Omega, \R^n)$ (see \cite[\S 2.2]{BB} and \cite{McMullen}).

For the critical case $\sigma \in L^n(\Omega)$, even though there exists $u \in W^{2,n}(\Omega)$ that solves $-\Delta u=\sigma$, and hence $\nabla u \in W^{1,n}(\Omega, \R^n)$ solves \eqref{unicacosa}, we can not conclude that $\nabla u \in L^{\infty} (\Omega, \R^n)$ since it is a limiting case of the Sobolev imbedding. Moreover, if we consider the function $u(x)= \varphi(x)x_1 |\ln |x||^{\alpha}, 0 < \alpha < \frac{n-1}{n}$,
 where $\varphi$ is a smooth cut-off function with support near 0,
 then it holds that
$\Delta u \in L^n$, but $\nabla u \notin L^{\infty}$ (see \cite[Remark 7]{BB}). If $-\Delta u =\sigma \in L^p(\Omega)$
for $n< p < \infty$ and $u \in W^{2,p}(\Omega)$, then $\nabla u \in C^{0, 1-\frac{n}{p}}(\Omega, \R^n)$
(see \cite[\S 5.6.2, Theorem 5]{eg}), which implies that, for the case $p=n$, we can not conclude that $\nabla u$ is continuous.
Bounded and continuous vector fields that solve equation \eqref{unicacosa} with $\sigma \in L^n(\Omega)$ were constructed in Bourgain-Brezis \cite[Proposition 1, Theorem 1 and Theorem 1']{BB} by using other techniques.

For general distributions $\sigma,$ a characterization of the solvability of \eqref{unicacosa} in the class of continuous vector fields was obtained in Pfeffer-De Pauw \cite{PD} and De Pauw-Torres \cite{DT}. It was shown in \cite{PD} that there exists a continuous vector field in $\Omega$ that solves the equation if and only if $\sigma$ belongs to a space of distribution denoted as the space of strong charges (see \cite[Theorem 4.7]{PD}. In \cite{DT}, it was proved that there exists a continuous vector field $\FF \in C_0(\R^n,\R^n)$ ({\it i.e.}, vanishing at infinity) if and only if $\sigma$ belong to the space of charges vanishing at infinity (see \cite[Theorem 6.1]{DT}). The space $L^n$ belongs to both spaces of distributions,
in particular solving equation \eqref{unicacosa} in the class of continuous vector fields when $\sigma \in L^n(\Omega)$.
\end{remark}

Our construction in Theorem \ref{globalexistence} provides a \emph{characterization} of all fields $\FF$
such that $-\div \FF = \sigma$, as all possible extensions of $\tilde L$ from $R(A)$ to $C_0(\Omega)$.
Indeed, if $\FF$ is any such field, then, for any $\varphi \in C_{\mathrm{c}}^{\infty}(\Omega)$,
  \begin{equation}
    \int_{\Omega} \nabla \varphi\cdot \d \FF = \sigma(\varphi) = \tilde L(\nabla \varphi).
  \end{equation}
  Our proof relies on the Hahn-Banach Theorem to produce an extension
  and hence is non-constructive.

  Similar techniques as in Theorem \ref{globalexistence} have been used in \cite[Theorem 4.5, Theorem 3.5]{PT1} and \cite[Theorem 4.4, Theorem 7.4]{PT2}
  to obtain the necessary and sufficient condition
  on the measure $\sigma$ to solve $-\div \FF= \sigma$ in the spaces $\FF \in \Leb^{\infty}(\Omega, \R^n)$
  or $\FF\in C(\Omega,\R^n)$.

\section{Equivalence between Entropy Solutions of
Nonlinear PDEs of Divergence Form and the Mathematical Formulation of Physical Balance Laws} \label{sec:conversionlaws}

In this section, we employ the results obtained in \S \ref{sec:definitions}--\S \ref{sec:alt_existence}
to establish the equivalence between
solutions of nonlinear PDEs of divergence form and the mathematical formulation
of physical balance laws,
and to analyze entropy solutions of hyperbolic conservation laws.

\subsection{Mathematical formulation of physical balance laws}

As stated in \S \ref{sec:flux_definition}, a physical balance law on an open set $\Omega$ of $\R^n$
postulates that the {\it production}
of a vector-valued extensive quantity in any bounded open set $U\Subset\Omega$ is balanced
by the {\it Cauchy flux} of this quantity through the boundary $\partial U$ of the open set $U$;
see also Dafermos \cite{dafermos2010hyperbolic} and the references cited therein.

\smallskip
Like the Cauchy flux, the {\it production} is introduced through a functional $\mathcal{P}$, defined on any bounded open
set $U\Subset\Omega$, taking values in $\R^N$ and satisfying the conditions:
\begin{align}
&\mathcal{P}(U_1\cup U_2)=\mathcal{P}(U_1)+\mathcal{P}(U_2) \qquad \mbox{if $U_1\cap U_2=\emptyset$},\label{12.1}\\
&\lvert\mathcal{P}(U)\rvert\le \widetilde\mu(U), \label{12.2}
\end{align}
where $\widetilde{\mu}$ is a given Radon measure.
It follows (from, for example,
\cite{Fuglede}) that $\mathcal P$ extends to a measure, in which
there is a production
measure $\sigma\in\mathcal{M}(\Omega;\R^N)$, with $\sigma\ll \widetilde\mu$, such that
\begin{equation}\label{12.3}
\mathcal{P}(U)= \sigma(U) \qquad\mbox{for any $U \Subset \Omega$ open}.
\end{equation}
Then the physical principle of balance law on $\Omega$ can be mathematically formulated as
\begin{equation}\label{12.4}
\mathcal{F}(\partial U)= \mathcal{P}(U)
=\sigma(U)
\end{equation}
for any bounded open set $U\Subset \Omega$.

\subsection{Equivalence between weak solutions of the physical balance laws and nonlinear PDEs of divergence form}
First, combining Theorem \ref{maintheorem} with \eqref{12.3}--\eqref{12.4}, we conclude that there exists a unique divergence-measure field
$\FF\in \mathcal{DM}^{\rm ext}(\Omega)$ such that
\begin{equation}
-\div\FF=\sigma  \qquad \mbox{in the sense of distributions}, \label{12.5}\\
\end{equation}
and, for any  Borel subset $S\subset \partial U$ with $U\in \mathcal O_{\mu}$,
\begin{equation}
\mathcal F_U(S)=(\FF\cdot\nu)_{\partial U}(S),  \label{12.6}
\end{equation}
where $(\FF\cdot\nu)_{\partial U}$ is the normal trace of $\FF$ on $\partial U$.
By considering a vector-valued flux and working component-wise,
we can obtain a matrix-valued field $\FF = (\FF_1,\FF_2,\cdots,\FF_N)^{\intercal}$ with each row $\FF_i$ lying in $\mathcal{DM}^{\ext}(\Omega)$, so \eqref{12.5} becomes a \emph{system} of balance laws (understood with a row-wise divergence).

Consider the state of the medium under consideration that is described by a state vector field ${\bf u} = (u_1,\cdots, u_N)^{\intercal}$
taking values in $\R^N$,
which determines both the flux density field $\FF$ and the production density field $\sigma$
at point $y\in \Omega$ by the constitutive relations:
\begin{equation}\label{12.7}
\FF(y)= \FF(\bfu(y),y), \qquad \sigma(y):=\sigma(\bfu(y), y),
\end{equation}
where $\FF(\bfu,y)$ and $\sigma(\bfu, y)$ are given vector fields in $(\bfu,y)$, understood in a sense determined by the physical system.

In the case that $\FF(\bfu,y)$ and $\sigma(\bfu,y)$ are sufficiently regular,
\eqref{12.7} is always well-defined if $\bfu \in \Leb^{\infty}$.
This is not always so if $\bfu \in \Leb^p$ for some $1 \leq p < \infty$, and further assumptions may be necessary to ensure the composition $\FF(\bfu(y),y)$ is locally integrable.
If $\bfu$ is measure-valued, then we must give a meaning to \eqref{12.7};
this is a modeling issue, which is beyond the scope of this paper.
An example of such constitutive relations in the measure-valued case is
given in \cite{CF2} for the Euler equations for gas dynamics in Lagrangian coordinates.

Combining \eqref{12.5} with \eqref{12.7}, we obtain the first-order quasilinear system of PDEs
of divergence form:
\begin{align}
\div\FF(\bfu(y), y)+\sigma(\bfu(y), y)=0,
\label{12.8}
\end{align}
which is called a system of nonlinear PDEs of balance laws  ({\it cf}. \cite{dafermos2010hyperbolic}).

For the zero-production case: $\mathcal{P}=0$, which implies that $\sigma(\bfu(y), y)=0$, then the above derivation yields
\begin{align}
\div\FF(\bfu(y), y)=0,
\label{12.9}
\end{align}
which is called a nonlinear system of conservation laws. In particular, when the medium is homogeneous:
\begin{equation}\label{12.10}
\FF(\bfu,y)=\FF(\bfu),
\end{equation}
depending on $y$ only through the state vector $\bfu=\bfu(y)$, then system \eqref{12.9} becomes
\begin{align}
\div\FF(\bfu(y))=0.
\label{12.11}
\end{align}
Now suppose that the coordinate system $y$ is described by the time variable $t$
and the space variable $x=(x_1,\cdots, x_m)$:
$$
y=(t,x)=(t,x_1,\cdots, x_m),  \quad n=m+1,
$$
and the flux density is written as an $N \times n$ matrix:
$$
\FF(\bfu)=(\bfu, \bff(\bfu))=(\bfu, \bff_{1}(\bfu), \cdots, \bff_{m}(\bfu))
$$
where $\bff_i^{\intercal} \colon \mathbb R^N \to \mathbb R^m$ for each $i = 1,\cdots,m$.
Then
we obtain the following standard form for the system of conservation laws:
\begin{equation}\label{12.12}
\partial_t \bfu +\div_x \bff(\bfu)=0, \qquad (t,x)\in\R^n, \,\, \bfu\in \R^N,
\end{equation}
where the spacial divergence is taken row-wise.

This shows the equivalence between the weak solutions of physical flows
governed by the balance laws \eqref{12.4} and weak
solutions of the corresponding system of nonlinear PDEs of divergence form \eqref{12.8},
or even a system of conservation laws
\eqref{12.12}, through the relations \eqref{12.5}--\eqref{12.6}.

\bigskip
\begin{theorem}  The following statements hold{\rm :}
\begin{enumerate}
\item[\rm (i)] The state variables $\bfu\in \R^N$ $($even measure-valued$)$ that are governed respectively by the physical balance laws \eqref{12.4}
and the corresponding nonlinear PDEs of divergence form \eqref{12.8}
are equivalent on any open set $U$.
That is, the weak solutions $\bfu\in \R^N$ $($even measure-valued$)$
for the state variables of the physical balance laws \eqref{12.4}
and the corresponding nonlinear PDEs of divergence form \eqref{12.8} are equivalent.

\item[\rm (ii)]  Given an open set $U$, the divergence-measure field $\FF\in \mathcal{DM}^{\rm ext}(\Omega)$ is endowed with
the normal trace  $(\FF\cdot\nu)_{\partial U}$ on $\partial U$ such that \eqref{12.6} hold
for any  Borel subset $S\subset \partial U$ with $U\in \mathcal O_{\mu}$.

\item[\rm (iii)] If \eqref{12.9} is satisfied, across any discontinuity surface $S$ on which the underlying field $\FF$ does not concentrate, the weak Rankine-Hugoniot condition holds{\rm :}
The exterior and interior
normal traces of the divergence-measure field $\FF$ for weak solutions are equal, {\it i.e.}, the normal trace is continuous across
$S$.
\end{enumerate}
\end{theorem}

However, the exterior and interior normal traces of the corresponding entropy divergence-measure fields on a shock wave must have
a jump, inferred by the second law of thermodynamics
as indicated in \S \ref{sec:entropy_sols} below.

\medskip
\subsection{Entropy solutions of hyperbolic conservation laws}\label{sec:entropy_sols}
We now apply the results established in \S \ref{sec:prop_disintegration} through \S \ref{sec:flux_localrecovery} to the recovery of Cauchy entropy fluxes and the corresponding
entropy balance laws through the Lax entropy inequality ({\it i.e.}, the second law of thermodynamics)
for entropy solutions of hyperbolic conservation laws
by capturing entropy dissipation.
That is, for hyperbolic conservation laws, even though there may be no production
for a weak solution,
the entropy solutions must obey the entropy balance law
with non-zero production in general, especially when
the entropy solution contains shock waves.
This leads to the intrinsic connection between the second law of thermodynamics
and the entropy balance laws in Continuum Mechanics.

We focus now on system \eqref{12.12}, which is assumed to be hyperbolic.
A function $\eta: \R^N\to \R$ is called an entropy of \eqref{12.12}
if there exists $\bfq = ({q}_1,\cdots,{q}_m)^{\intercal}: \R^N\to \R^m$ such that
\begin{equation}\label{12E-1}
\nabla q_j(\bfu)=\nabla\eta(\bfu) \nabla \bff_j(\bfu), \qquad j=1,\cdots, m.
\end{equation}
Then the vector function $\bfq(\bfu)$ is called an entropy flux
associated with entropy $\eta(\bfu)$,
and the function pair
$(\eta(\bfu), \bfq(\bfu))$ is called an entropy pair.
The entropy pair $(\eta(\bfu), \bfq(\bfu))$ is called a convex entropy pair
on the state domain $D\subset \R^N$ if the Hessian matrix
$\nabla^2\eta(\bfu)\ge 0$ for any $\bfu\in D$.

It is observed that most systems of conservation laws that
result from continuum mechanics are endowed with a globally defined
convex entropy (see Dafermos \cite{dafermos2010hyperbolic}
and Friedrichs-Lax \cite{Friedrichs-Lax}).
The available existence theories show that solutions of
\eqref{12.12} generally fall within the following class of entropy solutions:

\begin{definition}
A vector function $\bfu=\bfu(t,x)\in \mathcal{M}_{\rm loc}(\R_+\times \R^m)$ or $\Leb^p(\R_+\times \R^m)$
for some $p\ge 1$ is called an entropy solution of \eqref{12.12} if $\bfu(t,x)$ satisfies the Lax entropy inequality{\rm :}
\begin{equation}\label{12E-2}
\partial_t\eta(\bfu(t,x))+\div_x \bfq(\bfu(t,x))\le 0
\end{equation}
in the sense of distributions for any convex entropy pair $(\eta, \bfq)$ such that
$$
(\eta(\bfu(t,x)),\,\bfq(\bfu(t,x))) \qquad \mbox{is a distributional field}.
$$
\end{definition}

Clearly, an entropy solution is a weak solution, which can be seen by choosing $(\eta(\bfu), \bfq(\bfu))=\pm (u_i, \bff_i(\bfu))$, $i = 1,\cdots,N$,
in \eqref{12E-2}.

One of the main issues in hyperbolic conservation laws is to study the behavior of entropy
solutions in this class to explore to the fullest extent possible all questions relating
to uniqueness, stability, large-time behavior, structure, and traces of entropy solutions,
with neither specific reference to any particular method for constructing the
solutions nor additional regularity assumptions.
The Lax entropy inequality \eqref{12E-2}
indicates that the distribution:
$$
\partial_t\eta(\bfu(t,x))+\div_x \bfq(\bfu(t,x))
$$
is nonpositive. Then we conclude that it is, in fact, a Radon measure; that is,
there exists
$\sigma_\eta\in \mathcal{M}(\R_+\times \R^m)$ with $\sigma_\eta\ge 0$ such that
\begin{equation}\label{12E-3}
-\div_{(t,x)}(\eta(\bfu(t,x)), \bfq(\bfu(t,x)))=:\sigma_\eta.
\end{equation}
Therefore, the vector field
$(\eta(\bfu(t,x)), \bfq(\bfu(t,x)))$ is a divergence-measure field:
$$
(\eta(\bfu(t,x)), \bfq(\bfu(t,x)))\in \DM^{\rm ext}_{\rm loc}(\R_+\times \R^m),
$$
provided that $(\eta(\bfu(t,x)), \bfq(\bfu(t,x)))\in \mathcal{M}_{\rm loc}(\R_+\times \R^m)$.

We introduce a functional on any surface $S\subset\partial U$ with $U\in \mathcal{O}_\mu$:
\begin{equation}\label{12E-4}
\mathcal{F}_U^\eta(S):=-((\eta(\bfu), \bfq(\bfu))\cdot \nu)_{\partial U}(S)
\end{equation}
where $((\eta(\bfu), \bfq(\bfu))\cdot \nu)_{\partial U}$ is the normal trace of $(\eta(\bfu), \bfq(\bfu))$
on $\partial U$ in the sense of Theorem \ref{principal},
since $(\eta(\bfu(t,x)), \bfq(\bfu(t,x)))\in \DM^{\rm ext}_{\rm loc}(\R_+\times \R^m)$.

By Theorem \ref{prop:flux_sufficiency}, the functional $\mathcal{F}_U^\eta$ defined
by \eqref{12E-4} is a Cauchy flux in the sense
of Definition \ref{DefinitionCauchyFlux},
taking $\mu$ to be the total variation measure $\lvert(\eta(\bfu(t,x)), \bfq(\bfu(t,x))\rvert$.

\begin{definition}[Cauchy entropy fluxes]\label{def:12-2}
A functional $\mathcal{F}_U^\eta$ as defined by \eqref{12E-4} is
called a Cauchy entropy flux with respect to entropy $\eta({\bf u})$.
\end{definition}

In particular, when $(\eta, \bfq)$ is a convex entropy pair,
$\mathcal{F}_U^\eta(S)\ge 0$
for any Borel set $S\subset \partial U$ with $U\in \mathcal{O}_\mu$.
Furthermore, we can now reformulate the balance law of entropy
from the recovery of entropy production by capturing entropy dissipation.

\begin{theorem}[Entropy balance laws]\label{def:12-3}
The production $\sigma_\eta$ of entropy dissipation in any bounded open set $U\Subset \Omega$ is balanced
by the entropy Cauchy flux $\mathcal{F}_U^\eta$ through the boundary $\partial U$ of the open set $U\Subset\Omega$.
\end{theorem}

Furthermore, we have
\begin{theorem}  Assume that $\bfu=\bfu(t,x)\in \mathcal{M}_{\rm loc}(\R_+\times \R^m)$ or $\Leb^p(\R_+\times \R^m)$
for some $p\ge 1$ is an entropy solution to \eqref{12.12}.
Then
\begin{enumerate}
\item[\rm (i)]
Across the discontinuity,
the normal trace of the vector field $(\bfu(t,x), \bff(\bfu(t,x)))$ is continuous,
provided the underlying field does not concentrate on the discontinuity surface{\rm ;}
in this case the weak Rankine-Hugoniot condition holds.

\item[\rm (ii)] For a convex entropy pair $(\eta(\bfu), \bfq(\bfu))$ with $\nabla^2\eta(\bfu)>0$,
the normal trace of the vector field $(\eta(\bfu(t,x)), \bfq(\bfu(t,x)))$ has a jump across
a shock wave, which increases across the shock in the $t$-direction.

\end{enumerate}
\end{theorem}

The continuity of the normal trace of the vector field $(\bfu(t,x), \bff(\bfu(t,x)))$ is due to the fact that
$\div (\bfu(t,x), \bff(\bfu(t,x)))$ is zero
and
$\lvert(\bfu,\bff(\bfu))\rvert(\partial U) = 0$ as assumed.
On the other hand, the jump of the normal trace of the vector field $(\eta(\bfu(t,x)), \bfq(\bfu(t,x)))$
is because $\div(\eta(\bfu(t,x)), \bfq(\bfu(t,x)))$ has a concentration on a shock wave,
due to the entropy dissipation,
concentrated on the shock wave.

\smallskip
Furthermore, for characteristic discontinuities such as vortex sheets and entropy waves, in general,
$\div(\eta(\bfu(t,x)), \bfq(\bfu(t,x)))$ does not have a concentration due to the loss of the entropy dissipation
along such a discontinuity,
and the normal trace for the vector field $(\eta(\bfu(t,x)), \bfq(\bfu(t,x)))$ may generally be continuous across the characteristic
discontinuity normally.

\smallskip
Moreover, it is clear that understanding further properties of divergence-measure
fields can advance our understanding of the behavior of entropy solutions for hyperbolic
conservation laws and other related nonlinear PDEs by selecting appropriate
entropy pairs. As examples, we refer the reader to \cite{Chen, ChenComiTorres,CF1,CF2,CF3,CF4,ChenLiTorres,ChenTorres, ChenTorresStructure, ChenTorresNotices,Rascle,Vasseur}
for such applications.

\bigskip
\bigskip
\noindent{\bf Acknowledgements}.
The research of Gui-Qiang G. Chen was also supported in part by the UK Engineering and Physical Sciences Research
Council Award EP/L015811/1, EP/V008854, and EP/V051121/1.
The research of Christopher Irving was also supported in part by the UK Engineering and Physical Sciences Research
Council Award EP/L015811/1. The research of Monica Torres was supported in part by NSF DMS-1813695.
For the purpose of open access, the authors have applied a CC BY public copyright license to any Author Accepted Manuscript (AAM) version
arising from this submission.

\bigskip
\medskip
\noindent{\bf Conflict of Interest:} The authors declare that they have no conflict of interest.
The authors also declare that this manuscript has not been previously published,
and will not be submitted elsewhere before your decision.

\bigskip
\noindent{\bf Data availability:} Data sharing is not applicable to this article as no datasets were generated or analyzed during the current study.

\end{document}